\documentclass[12pt,english]{amsart}

\usepackage[left=3cm,top=3cm,right=3cm,bottom=3cm]{geometry}
\usepackage[utf8]{inputenc}
\usepackage[pdfstartpage=1,colorlinks=true,bookmarks=false,pdfstartview={FitH}]{hyperref}

\DeclareMathOperator{\dist}{dist}

\usepackage{amsmath,amssymb}
\usepackage{mathtools}
\usepackage{esint}
\usepackage{verbatim}
\usepackage{comment}
\usepackage{cleveref}

\DeclarePairedDelimiter\abs{\lvert}{\rvert}
\makeatletter
\let\oldabs\abs
\def\abs{\@ifstar{\oldabs}{\oldabs*}}
\newcommand{\eps}{\varepsilon}

\newcommand{\R}{\mathbb{R}}
\newcommand{\p}{\partial}

\newcommand{\Ds}{(-\Delta)^s_{\rm RFL}}
\newcommand{\Dssp}{(-\Delta)^{s}_{\rm SFL}}
\newcommand{\Ints}{(-\Delta)^{-s}}

\newcommand{\PV}{\textnormal{P.V.}\,}
\newcommand{\loc}{\textnormal{loc}}
\newcommand{\norm}[2][]{\left\|{#2}\right\|_{#1}}

\newcommand{\sign}{\textnormal{sign}\,}

\newcommand{\set}[1]{\left\{#1\right\}}
\newcommand{\bG}{\mathbb{G}}
\newcommand{\bH}{\mathbb{H}}
\newcommand{\bM}{\mathbb{M}}

\newcommand{\bS}{\mathbb{S}}

\newcommand{\cG}{\mathcal{G}}
\newcommand{\cH}{\mathcal{H}}
\newcommand{\cJ}{\mathcal{J}}
\newcommand{\cL}{\mathcal{L}}
\newcommand{\cM}{\mathcal{M}}
\newcommand{\cK}{\mathcal{K}}

\newcommand{\textas}{\text{ as }}
\newcommand{\texton}{\text{ on }}

\newcommand{\textor}{\text{ or }}
\newcommand{\textin}{\text{ in }}
\newcommand{\textfor}{\text{ for }}
\newcommand{\textforall}{\text{ for all }}
\newcommand{\textand}{\text{ and }}

\newcommand{\angles}[1]{\left\langle#1\right\rangle}

\DeclareMathOperator{\diam}{diam}

\DeclareMathOperator{\spanned}{span}

\theoremstyle{plain}
\newtheorem{thm}{Theorem}[section]
\newtheorem*{thm*}{Theorem}
\newtheorem{lem}[thm]{Lemma}
\newtheorem{lemma}[thm]{Lemma}
\newtheorem{cor}[thm]{Corollary}
\newtheorem{prop}[thm]{Proposition}

\theoremstyle{definition}
\newtheorem{defn}{Definition}[section]

\theoremstyle{remark}
\newtheorem{remark}{Remark}[section]

\numberwithin{equation}{section}

\vbadness=\maxdimen

\usepackage{enumerate}

\begin{document}

\title[Nonlocal eigenvalue problems]{Blow-up phenomena in nonlocal eigenvalue problems: when theories of $L^1$ and $L^2$ meet}

\author{Hardy Chan}
\address[{\bf H. Chan}]{ \sc Dept. of Mathematics, ETH Z\"urich,
R\"amistrasse 101 \\ 8092 Z\"urich, Switzerland }
\email[]{hardy.chan@math.ethz.ch}
\urladdr{https://people.math.ethz.ch/~hchan/}

\author{David G\'omez-Castro}
\address[{\bf D. G\'omez-Castro}]{\sc Instituto de Matemática Interdisciplinar,
	Universidad Complutense de Madrid\\
	28040 Madrid, Spain}
\email[]{dgcastro@ucm.es}
\urladdr{http://blogs.mat.ucm.es/dgcastro/}

\author{Juan Luis V\'{a}zquez}
\address[{\bf J.~L.~V\'azquez}]{\sc Depto. de Matem\'aticas\\
Univ. Aut\'onoma de Madrid (UAM)\\
28049 Madrid, Spain}
\email[]{juanluis.vazquez@uam.es}
\urladdr{http://verso.mat.uam.es/~juanluis.vazquez/}

\begin{abstract}
	
We develop a linear theory of very weak solutions for nonlocal eigenvalue problems $\mathcal L u = \lambda u + f$ involving integro-differential operators posed in bounded domains with homogeneous Dirichlet exterior condition, with and without singular boundary data. We consider mild hypotheses on the Green function and the standard eigenbasis of the operator. The main examples in mind are the fractional Laplacian operators.

Without singular boundary datum and when $\lambda$ is not an eigenvalue of the operator, we construct an $L^2$-projected theory of solutions, which we extend to the optimal space of data for the operator $\mathcal L$. We present a Fredholm alternative as $\lambda$ tends to the eigenspace and characterise the possible blow-up limit. The main new ingredient is the transfer of orthogonality to the test function.

We then extend the results to singular boundary data and study the so-called large solutions, which blow up at the boundary.  For that problem we show that, for any regular value $\lambda$, there exist ``large eigenfunctions'' that are singular on the boundary and regular inside. We are also able to present a Fredholm alternative in this setting, as $\lambda$ approaches the values of the spectrum.

We also obtain a maximum principle for weighted $L^1$ solutions when the operator is $L^2$-positive. It yields a global blow-up phenomenon as the first eigenvalue is approached from below.

Finally, we recover the classical Dirichlet problem as the fractional exponent approaches one under mild assumptions on the Green functions. Thus ``large eigenfunctions'' represent a purely nonlocal phenomenon.

\bigskip


\

\end{abstract}

\maketitle

\pagebreak

\section{Introduction}

Let $\Omega\subset\R^n$, $n\ge 1$, be bounded of class $C^{1,1}$. Let $\cL$ be an integro-differential operator of order $2s\in(0,2)$ in $\Omega$ with exterior or boundary Dirichlet condition, whichever is applicable. Typical choices  for $\cL$ are the restricted and spectral versions  of the fractional Laplacian $(-\Delta)^s$ for all $s\in(0,1)$, briefly RFL and SFL, but our theory includes a wider class of operators described in \Cref{sec:L}. Let $\Omega^c = \mathbb R^n \setminus \overline \Omega$. From  standard theory, it is known that non-trivial solutions for the problems
\[
\begin{cases}
\Ds u=\lambda u
    & \textin \Omega\\
u=0
    & \textin {\partial \Omega \cup \Omega^c },
\end{cases}
    \quad \textor \quad
\begin{cases}
\Dssp u=\lambda u
    & \textin \Omega\\
u=0
    & \texton \p\Omega,
\end{cases}
\]
 exist if and only if $\lambda$ belongs to the standard spectrum, $\Sigma$, a discrete subset of $(0,+\infty)$. Moreover, the corresponding eigenfunctions  are continuous up to the boundary, though with different regularity. Surprisingly, even the homogeneous Dirichlet problem
\[\begin{cases}
\Ds u=0
    & \textin B_1\\
u=0
    & \textin {B_1^c},
\end{cases}\]
without condition \emph{on} the boundary is ill-posed  as shown by the non-trivial solution given by
\[
u(x)=
\begin{cases}
(1-|x|^2)^{s-1}
	& \textin B_1\\
0
	& \textin {B_1^c},
\end{cases}\]
found in Hmissi \cite{Hmissi} and later Bogdan--Byczkowski--Kulczycki--Ryznar--Song--Vondra\v{c}ek \cite{Bogdan}.
This initiated the study of \emph{large solutions}, or solutions with boundary blow-up, which are the leitmotif of this paper. We need one further notation for the distance function to the boundary, $\delta:\Omega \rightarrow(0,+\infty)$ :
\[
 \delta(x)=\dist(x,\p\Omega).
\]
Among other things, Abatangelo \cite{A} finds large RFL-harmonic solutions blowing up at the boundary of a general domain like $\delta^{s-1}$, and Abatangelo--Dupaigne \cite{AD} establishes large solutions for SFL growing like $\delta^{2s-2}$. Using two-sided estimates for Green's function, Abatangelo--G\'{o}mez-Castro--V\'{a}zquez \cite{AGV} provides a unified theory which also includes the censored fractional Laplacian (CFL), where no large solutions exist. The authors of \cite{AGV} also show the boundary behavior for the large solution (or its candidate) is $\delta^{-(1-2s+\gamma)}$, where $\gamma$ is the parameter that determines the optimal boundary behavior; thus $\gamma=s$ for RFL, $\gamma=1$ for SFL, and $\gamma=2s-1\in(0,1)$ for CFL. Notice that for the classical Laplacian, $s=\gamma=1$, so $1-2s+\gamma=0$, consistent with the fact that no large harmonic functions exist.

\medskip

\noindent {\bf Problem and type of results.} In this paper we are interested in ``large eigenfunctions'', or \emph{large solutions} of the eigenvalue problem, namely
\[\begin{cases}
\cL{u}-\lambda u=0
	& \textin \Omega,\\
u(x)\to+\infty
	& \textas x\to\p\Omega,\\
u=0
    & \textin \Omega^c
        \quad \text{(if applicable)}
\end{cases}\]
with  $\lambda\in\R$. The precise equations are introduced in \eqref{eq:intro-1}, \eqref{eq:main-f-h}. The class of operators is precisely defined in \Cref{sec:L}, and a suitable notion of  solution in \Cref{ssec.notion}. In contrast to the standard theory, we will establish the existence of a unique large solution for \emph{any} $\lambda\in\R\setminus\Sigma$, and for all prescribed (non-zero) boundary data in a suitable class. Moreover, as $\lambda$ approaches the standard spectrum $\Sigma$ of $\cL$, one expects a blow-up phenomenon, which we will precisely identify. See \Cref{thm:blow}. It follows from the study that ``large eigenfunctions'' are different from the standard eigenfunctions in three striking ways:
\begin{itemize}
\item they exist when $\lambda$ stays away from the standard spectrum;
\item they are singular on the boundary;
\item they blow up in the interior, when $\lambda$ tends to the standard spectrum.
\end{itemize}
For these reasons, although ``large eigenfunctions'' solve an equation similar to standard eigenfunctions, they show drastically different behavior from the standard ones. We emphasize this fact by using the quotation marks.

\medskip

\noindent {\bf Related work.} For the homogeneous eigenvalue problem, to our best knowledge, only locally bounded large solutions are known for $\lambda\in(-\infty,0)$ (absorption) or for $\lambda>0$ (reaction) that is sufficiently small \cite{A,AD}. %

\medskip

D\'{i}az--G\'{o}mez-Castro--V\'{a}zquez \cite{DGV} study very weak solutions of the RFL-Schr\"{o}dinger equation where the potential $-\lambda$ (when $\lambda\leq 0$) is replaced by a general non-negative function $V\geq0$.

\medskip

Interesting boundary blow-up solutions of nonlinear elliptic equations are also found in various settings, see for instance \cite{A,AD,ACDFGW,BB,BFV,CFQ,FQ}.

\medskip

In the classical case $\cL=-\Delta$, Keller \cite{K} and Osserman \cite{O} have found independently a necessary and sufficient condition for the existence of solutions of
\begin{equation}\label{eq:KO}
{-\Delta} u+f(u)=0
    \quad \textin \Omega,
\end{equation}
with explosive boundary behavior,\footnote{More precisely, unbounded behavior on $\p\Omega$ for $\cL=-\Delta$, and more generally, blowing up faster than $\delta^{s-1}(x)$ as $x\to\p\Omega$ for $\cL=\Ds$, $s\in(0,1)$.} when $f$ is non-negative and non-decreasing with $f(0)=0$. In fact, they show the equivalence of
\begin{itemize}
\item \emph{Keller--Osserman condition}, or superlinear growth of the nonlinearity as in $\int^{+\infty}(\int_0^t f)^{-1/2}\,dt<+\infty$ (which implies an \emph{a priori} estimate in terms of the radial solution),
\item existence of large solutions of \eqref{eq:KO} when $\Omega$ is a bounded smooth domain, and
\item non-existence of entire solutions of \eqref{eq:KO} when $\Omega=\R^n$.
\end{itemize}
In particular, when $f(u)=-\lambda u$ (for $\lambda\leq0$) the first condition is violated so no large solutions exist. See also Dumont--Dupaigne--Goubet--R\u{a}dulescu \cite{DDGV} for a study with oscillating nonlinearity. In the fractional case, a sufficient condition for the existence of solutions of $\Ds+f(u)=0$ with a blow-up rate faster than $\delta^{s-1}$, or a partial fractional Keller--Osserman condition, is obtained by Abatangelo \cite{A2}.

\bigskip

\section{Preliminaries, definitions and precise statement of results}\label{sec:prelim}

\subsection{The operator}
\label{sec:L}

The operator $\cL$ that we consider includes the two most popular versions of the fractional Laplacian in a bounded domain. $\cL$ takes the sign convention which is the one of $-\Delta$, hence is a non-negative operator. We make the following assumptions on $\cL$ throughout the paper:

\begin{enumerate}
\item[(K1)]
    $\cL$ has a Green's function\footnote{The zero subscript in $\cG_0$ suggests $\lambda=0$.} $\cG_0:(\Omega\times\Omega)\setminus\set{(x,y)\in\Omega\times\Omega:x=y}\to(0,+\infty)$, such that\footnote{$f\asymp g$ means $C^{-1}f\leq g\leq Cf$ for a constant $C$. $a\wedge b$ denotes the minimum of $a$ and $b$.}
\[\begin{split}
\cG_0(x,y)=\cG_0(y,x)
&\asymp
    \dfrac{1}{|x-y|^{n-2s}}
    \left(
        \dfrac{
            \delta^\gamma(x)
        }{
            |x-y|^{\gamma}
        }
        \wedge 1
    \right)
    \left(
        \dfrac{
            \delta^\gamma(y)
        }{
            |x-y|^{\gamma}
        }
        \wedge 1
    \right)\\
&\asymp
    \dfrac{1}{|x-y|^{n-2s}}
    \left(
        \dfrac{
            \delta^\gamma(x)
            \delta^\gamma(y)
        }{
            |x-y|^{2\gamma}
        }
        \wedge 1
    \right),
\end{split}\]
for $\gamma\in(0,1]$, $n\neq 2s$.
\item[(K2)] $\cG_0$ has an (inner) $\gamma$-normal derivative $D_\gamma\cG_0:\p\Omega\times\Omega\to(0,+\infty)$,
\begin{equation*}%
D_\gamma\cG_0(z,x)
:=
\lim_{\Omega\ni y\to z}
\dfrac{
    \cG_0(y,x)
}{
    \delta^\gamma(y)
}
\asymp
    \dfrac{
        \delta^\gamma(x)
    }{
        |x-z|^{n-2s+2\gamma}
    },
        \quad x\in \Omega,\,z\in\p\Omega,
\end{equation*}
which is referred to as the \emph{Martin kernel} in this setting.
\end{enumerate}

\begin{remark}
	In similar settings, some authors included the following additional assumptions: the linear operator $\cL : \mathrm{dom} (\cL) \subset L^1 (\Omega) \to L^1 (\Omega)$ is assumed to be densely defined and sub-Markovian,
	\begin{enumerate}
		\item[(A1)] $\cL$ is $m$-accretive on $L^1 (\Omega)$,
		\item[(A2)] If $0 \le f \le 1$, then $0 \le e^{-t \cL } f \le 1$.
	\end{enumerate}
	We do not need them in this elliptic context.
\end{remark}

\begin{remark}
Note that when $n=2s=1$, $\cG_0$ has a logarithmic singularity.
\end{remark}

\begin{remark}
	By \cite[Proposition 5.1]{BFV},  (K1) (in fact $0 \le \cG_0 \le C |x-y|^{-(n-2s)}$) are sufficient for the Green function $\bG_0$ defined in \eqref{eq:G0-def} to be a compact operator $L^2(\Omega)\to L^2 (\Omega)$. They apply the Riesz–Fréchet–Kol\-mo\-go\-rov Theorem. Hence $\cL$ has a discrete standard spectrum $\Sigma=(\lambda_i)_{i\geq1}$, containing a non-decreasing, divergent sequence of positive Dirichlet eigenvalues
	\[
	0<\lambda_1\le \lambda_2\leq\cdots
	\nearrow+\infty,
	\]
	repeated according to multiplicity, and the corresponding standard eigenfunctions, $(\varphi_j)_{j\geq1}\subset C(\overline{\Omega})$, form an orthonormal basis of $L^2(\Omega)$, with $\varphi_1\ge 0$.
	
	The parameter $\gamma$ is responsible for the optimal boundary behavior in the sense of Hopf, and the first eigenfunction satisfies $\varphi_1\asymp \delta^\gamma$ (see \cite{BFV}).
	\end{remark}
	
\begin{remark}
Considering the range of $\gamma\in(0,1]$, we will at times look at the extra condition $\gamma<2s$, a case that is covered but not distinguished in \cite{AD,AGV}. The advantage in such regime is that $L^1$ weak-dual solutions can be directly defined, see \Cref{rem:cond}. On the other hand, from \cite{AGV} or the discussion of \Cref{sec:trace}, one obtains a large solution if and only if $\gamma>2s-1$. If $\gamma\leq 2s-1$, then  an $L^2$-theory suffices and \Cref{thm:blow} becomes simply an exercise.
\end{remark}

We discuss two concrete examples. They take the usual form
\begin{equation*}
	\cL u (x) = \int_\Omega (u(x) - u(y)) \cJ (x,y) dy + \kappa (x) u(x).
\end{equation*}
In \cite{BFV}, the authors check that the following two operators are $m$-accretive and sub-Markovian.
\begin{itemize}
\item
    The \emph{restricted fractional Laplacian} (RFL) is defined by imposing the exterior Dirichlet condition $u=0$ in $\Omega^c$ in the singular integral formula in the whole space, i.e. for $u\in C^{2s+\epsilon}_\loc(\R^n)\cap L^1(\R^n;(1+|x|)^{-n-2s})$,
\[
\Ds u(x)
=C_{n,s}\PV
    \int_{\R^n}
        \dfrac{
            u(x)-u(y)
        }{
            |x-y|^{n+2s}
        }
    \,dy,
\]
where $C_{n,s}=\frac{2^{2s}\Gamma(\frac{n+2s}{2})}{|\Gamma(-s)|\pi^{n/2}}$. It leads to
\[
\cJ(x,y)=\dfrac{C_{n,s}}{|x-y|^{n+2s}},
    \quad
\kappa(x)=
    \int_{\Omega^c}
        \dfrac{C_{n,s}}{
            |x-y|^{n+2s}
        }
    \,dy
\asymp \delta^{-2s}(x).
\]
By \cite{RS1}, $\gamma=s$ in \textrm{(K1)}. Note that $\gamma<2s$ is satisfied for all $s\in(0,1)$. Considering the exterior Dirichlet condition as part of the definition of the operator, we will not repeat it in the equations.
\item
    The \emph{spectral fractional Laplacian} (SFL) is defined by taking a fractional power of the Dirichlet Laplacian, typically done by an eigenfunction expansion. %
    This operator is constructed so that $\varphi_j$ are the basis of $L^2(\Omega)$-eigenfunctions of $-\Delta$ associated to eigenvalues (say) $\mu_j$, and $\lambda_j=\mu_j^s$.
In terms of the heat kernel $\cK(t,x,y)$ of $-\Delta$ in $\Omega$, one has \cite{SoVo,AD},
\[\begin{split}
\cJ(x,y)
&=\dfrac{s}{\Gamma(1-s)}
    \int_{0}^{+\infty}
        \cK(t,x,y)
    \,\dfrac{dt}{t^{1+s}}
\asymp
    \dfrac{1}{|x-y|^{n+2s}}
    \left(
        \dfrac{
            \delta(x)\delta(y)
        }{
            |x-y|^2
        }
        \wedge 1
    \right)\\
\kappa(x)
&=\dfrac{s}{\Gamma(1-s)}
    \int_{0}^{+\infty}
        \left(1-
            \int_{\Omega}
                \cK(t,x,y)
            \,dy
        \right)
    \,\dfrac{dt}{t^{1+s}}
\asymp \delta^{-2s}(x).
\end{split}\]
By \cite{CSt}, $\gamma=1$ in \textrm{(K1)}. See also \cite{CT}. Note that $\gamma<2s$ if and only if $s\in(\frac12,1)$.
\end{itemize}

\begin{remark}
	In \cite{BSV,BFV,AGV} the reader may find some other operators of fractional type satisfying (A1), (A2) and (K1). Whether hypothesis (K2) is satisfied is usually a difficult problem, which we will not discuss here. The study of other operators of  Schrödinger relativistic type $\cL = (-\Delta + m^2 {\rm Id})^{s} - m^{2s} {\rm Id}$ have also been of significant recent interest. Conditions (A1)--(A2) are easy to check. We do not know if (K1) and (K2) hold.
\end{remark}

We will make some further comments on the sharp structure of compact embeddings and functional spaces related to our operators in \Cref{sec:compactness}.

\subsection{Large $\cL$-harmonic functions}
\label{sec:trace}

This material corresponds to the case $\lambda=0$ of our equation. The knowledge of the precise boundary blow-up rate of the large $\cL$-harmonic functions, \cite{AGV}, leads us to introduce the important exponent
\begin{equation}\label{eq:b}
b=1-2s+\gamma.
\end{equation}
In particular,
\[
b=\begin{cases}
1-s
	& \textfor {\cL}=\Ds,\,s\in(0,1)\\
2-2s
	& \textfor {\cL}=\Dssp,\,s\in(0,1)\\
0
	& \textfor {\cL}={-\Delta}
		\quad \textor \quad
		{\cL}={(-\Delta)^s_{\rm CFL}},\,s\in(\frac12,1).\\
\end{cases}\]
Here $(-\Delta)^s_{\rm CFL}$ is the censored fractional Laplacian defined by
\[
(-\Delta)^s_{\rm CFL}u(x)=C_{n,s}\PV\int_{\Omega}\frac{u(x)-u(y)}{|x-y|^{n+2s}}\,dy.
\]
Notice that the parameter $\gamma=2s-1$, as discussed in \cite{AGV}.
We are interested in the case $b>0$. Define the Martin operator
\[\begin{split}
\bM:C(\p\Omega)
    &\longrightarrow \delta^{-b}L^\infty(\Omega)\\
h&\longmapsto \bM(h)
\end{split}\]
by
\begin{equation}\label{eq:martin}\begin{split}
\bM(h)(x)
    &=\int_{\p\Omega}
        D_\gamma\cG_0(z,x)h(z)
    \,d\cH^{n-1}(z),
        \quad \textfor x\in\p\Omega,
\end{split}\end{equation}
where $D_\gamma\cG_0(z,x)$ is the (inner) $\gamma$-normal derivative defined in \textrm{(K2)}. (When $\gamma=1$, it is reduced to the Poisson kernel.) For the reader's convenience, we include, in \Cref{prop:martin}, a short proof showing that the Martin operator %
is well-defined, in the hope of demystifying the naturally-occurring exponent $b=1-2s+\gamma$. When $h\equiv1$, in particular, \cite[Corollary 4.3]{AGV} asserts that
\[
\bM(1)(x)
=\int_{\p\Omega}
    D_\gamma\cG_0(z,x)
\,d\cH^{n-1}(z)
\asymp \delta^{-b}.
\]
When normalized by $\bM(1)$, the equation \eqref{eq:martin} actually leads to the representation formula for $\cL$-harmonic functions with prescribed (large) boundary data, i.e. $u=\bM(h)$ solves
\[\begin{cases}
\cL{u}=0
    & \textin \Omega\\
\dfrac{u}{\bM(1)}=h
    & \texton \p\Omega, \\
    u = 0 & \textin \Omega^c \text{ (if applicable)}.
\end{cases}\]
Since we will only focus on the case of null exterior condition, we do not  indicate this in the future. Indeed, by \cite[Theorem 4.13]{AGV} (see also \cite{A,AD}), for $h\in C(\p\Omega)$, the limit
\[
\dfrac{\bM(h)}{\bM(1)}(z)
:=\lim_{\Omega\ni x\to z}
    \dfrac{
        \bM(h)(x)
    }{
        \bM(1)(x)
    }
\]
exists uniformly in $z\in\p\Omega$ and equals $h(z)$. Even when $h\in L^1(\p\Omega)$, such boundary trace can still be understood in a weaker sense, as an average over constricting tubular neighborhoods of $\p\Omega$, see \cite[Theorem 4.15]{AGV}.

\medskip

On the other hand, the weighted trace operator
\[\begin{split}
B:\delta^{-b}C(\overline{\Omega})
    &\longrightarrow L^\infty(\p\Omega)\\
u&\longmapsto Bu
\end{split}\]
given by
\[
Bu(z)
=\lim_{\Omega\ni x\to z}
    \dfrac{
        u(x)
    }{
        \bM(1)(x)
    }
    \quad \textfor z\in\p\Omega,
\]
is well-defined. In the case $\cL=\Ds$, thanks to the connection pointed out to us by Ros-Oton, we even have the explicit formula
\begin{equation}\label{eq:trace-def}
Bu(z)=
	\Gamma(1+s) \Gamma(s)
	\normalcolor 
    \lim_{\Omega\ni x\to z}
        \delta^{1-s}(x)u(x),
\end{equation}
The Ros-Oton--Serra constant $\Gamma(1+s)^2$ appears already in the fractional Poho\v{z}aev identity \cite{RS2}. Through its relation to the integration-by-parts formula of Abatangelo \cite{A}, the precise expression \eqref{eq:trace-def} can be obtained. 
See \Cref{lem:constant} for the proof.

\subsection{Main equation and notions of solution}\label{ssec.notion}

Given $h\in C(\p\Omega)$ and $\lambda\in\R\setminus\Sigma$, we consider $L^1$-solutions of the eigenvalue problem
\begin{equation}\label{eq:intro-1}
\begin{cases}
{\cL}u-\lambda u=0
	& \textin \Omega\\
Bu=h
	& \texton \p\Omega,\\
\end{cases}\end{equation}
understood in a certain weak sense that we describe below. One of our main results, \Cref{thm:blow}, which is also the motivation of this work, is to determine the precise blow-up behavior  of its solutions.

More generally, consider the equation
\begin{equation}\label{eq:main-f-h}\begin{cases}
{\cL}u-\lambda u=f
	& \textin \Omega\\
Bu=h
	& \texton \p\Omega.\\
\end{cases}\end{equation}

Let us recall the notion of $L^1$ or very weak solution used in the literature \cite{A,AD,AGV}. Denote the Green function of $\cL$ in $\Omega$ by $\cG_0$, and the Green operator by
\begin{equation}\label{eq:G0-def}
\bG_0(f)(x)=\int_{\Omega}\cG_0(x,y)f(y)\,dy.
\end{equation}
We stress two facts: it was shown in \cite{AGV} that the largest class of functional data $f$ admissible is $L^1 (\Omega; \delta^\gamma)$ and that, if $f \in L^\infty_c (\Omega)$, then $|\bG_0 (f)| \le C \delta^\gamma$.

\bigskip

The notion of \emph{very weak formulation} first introduced by H. Brezis for local operators, was later extended to the notion of \emph{weak-dual formulation} and extensively used for instance in \cite{BFV, AGV}. Its merit is that it avoids defining $\cL$ on the test function. As a matter of fact, the regularity needed to define $\cL$ classically calls for more (i.e. $2s+\epsilon$, $\epsilon>0$) than what a Schauder estimate can recover (i.e. $2s$, or $2s-\epsilon=1-\epsilon$ when $2s=1$). This issue will not be present when the Green function is used instead. We will specify the duality precisely by naming the Green function used in the definition.

\begin{defn}[$\bG_0(C_c^\infty(\Omega))$-weak solutions \cite{A,AD}]
Suppose $f\in L^1(\Omega;\delta^\gamma)$ and $h\in C(\p\Omega)$. We say that $u\in L^1(\Omega; \delta^\gamma )$ is a \emph{$\bG_0(C_c^\infty(\Omega))$-weak solution} of \eqref{eq:main-f-h} if
\begin{equation}
\label{eq:def-L1}
\int_{\Omega}
	u({\cL}\zeta-\lambda\zeta)
\,dx
=\int_{\Omega}
	f\zeta
\,dx
+\int_{\p\Omega}
    D_\gamma\zeta
    \,h
\,d\cH^{n-1},
	\quad \forall\zeta\in \bG_0(C_c^\infty(\Omega)),
\end{equation}
or, equivalently,
\begin{equation}\label{eq:weak1}
\int_{\Omega}
	u(\psi-\lambda\bG_0(\psi))
\,dx
=\int_{\Omega}
	f\bG_0(\psi)
\,dx
+\int_{\p\Omega}
    D_\gamma\bG_0(\psi)
    \,h
\,d\cH^{n-1},
	\quad \forall\psi\in C_c^\infty(\Omega).
\end{equation}
\end{defn}
Formulation \eqref{eq:def-L1} corresponds to the notion of very weak solution introduced by Brezis, where all ``derivatives'' are transferred to the test function. Formulation \eqref{eq:weak1} corresponds to the notion of weak-dual solution, in which derivatives are not present.

\begin{remark}
	We point out that
	\begin{enumerate}
		\item The integrability $u\in L^1(\Omega; \delta^\gamma )$ is sufficient in this setting. Notice that, since $\psi$ is bounded and compactly supported, then $u \psi$ is integrable. Furthermore, since $|\cG_0 (\psi)|\le C \delta^\gamma$, then $u \cG_0 (\psi)$ is also integrable.
		\item Notice that the $\gamma$-normal derivative $D_\gamma\bG_0(\psi)$ is well-defined on $\p\Omega$ by the assumption \textrm{(K2)}, see \Cref{prop:D-gamma-def}. To gain an intuition of the boundary integral in \eqref{eq:weak1}, suppose $\lambda=0$ and $f=0$. Testing the representation formula
		\[
		u(x)=
		\int_{\p\Omega}
		D_\gamma\cG_0(z,x)
		h(z)
		\,d\cH^{n-1}(z)
		\]
		against $\psi\in C_c^\infty(\Omega)$, \Cref{prop:D-gamma-def} implies
		\[\begin{split}
		\int_{\Omega}
		u\psi
		\,dx
		&=\int_{\Omega}
		\left(
		\int_{\p\Omega}
		D_\gamma\cG_0(z,x)
		h(z)
		\,d\cH^{n-1}(z)
		\right)
		\psi(x)
		\,dx\\
		&=\int_{\p\Omega}
		\left(
		\int_{\Omega}
		D_\gamma\cG_0(z,x)
		\psi(x)
		\,dx
		\right)
		h(z)
		\,d\cH^{n-1}(z)\\
		&=\int_{\p\Omega}
		D_\gamma\bG_0(\psi)(z)
		h(z)
		\,d\cH^{n-1}(z).
		\end{split}\]
	\end{enumerate}
\end{remark}

For an $L^1$-theory, it is convenient to enlarge the class of test functions to bounded functions with compact support.
\begin{defn}[$\bG_0(L_c^\infty(\Omega))$-weak solutions \cite{AGV}]
Suppose $f\in L^1(\Omega;\delta^\gamma)$ and $h\in C(\p\Omega)$. We say that $u\in L^1(\Omega; \delta^\gamma)$ is a \emph{$\bG_0(L_c^\infty(\Omega))$-weak solution} of \eqref{eq:main-f-h} if
\begin{equation*}
\int_{\Omega}
	u(\psi-\lambda\bG_0(\psi))
\,dx
=\int_{\Omega}
	f\bG_0(\psi)
\,dx
+\int_{\p\Omega}
    D_\gamma\bG_0(\psi)
    \,h
\,d\cH^{n-1},
	\quad \forall\psi\in L_c^\infty(\Omega).
\end{equation*}
\end{defn}

Let us now restrict ourselves to the case $h=0$, which happens in most parts of the present article. %
\begin{equation}
\label{eq:main-f}
\begin{cases}
{\cL}u-\lambda u=f
    & \textin \Omega\\
B u=0
    & \texton \p\Omega.\\
\end{cases}
\end{equation}
Note that in this case, the boundary condition can be rewritten as $\delta^b u=0$.

For later purposes, it is more convenient to use the test function space $\bG_\lambda(L^\infty(\Omega))$, where $\bG_\lambda$ is the solution operator of the full Helmholtz problem \eqref{eq:main-f}, when $\lambda$ is not an eigenvalue of $\cL$.  By spectral decomposition, it can  be  defined simply by \eqref{eq:G-lambda-f} below.
In \Cref{thm:lin} we will show this operator has a unique extension to $L^1 (\Omega; \delta^\gamma)$ data.
We note that, to our best knowledge, the exact behavior of its  Green  function $\cG_\lambda$ has  not yet been studied.

\begin{remark}\label{rem:cond}
	Further along,  we will want to perform projections onto the basis of eigenfunctions, which do not preserve the compact support. Hence $\psi \in L^\infty_c (\Omega)$ seems not to be convenient for us. Fortunately, we know that the eigenfunction $\varphi_j \in \delta^\gamma L^\infty (\Omega)$. We need sharp information on the boundary behaviour. In \cite{AGV} the authors prove that, for $\alpha > - 1-\gamma$,
	\begin{equation}
	\label{eq:improvement of weights}
		\bG_0 (\delta^\alpha) \asymp \begin{dcases}
		\delta^{\alpha + 2s} & \textfor \alpha + 2s < \gamma, \\
		\delta^\gamma (1 + |\ln \delta| ) & \textfor \alpha + 2s = \gamma, \\
		\delta^{\gamma} & \textfor \alpha + 2s > \gamma.
		\end{dcases}
	\end{equation}
	In particular,
	\begin{equation*}
	\bG_0 (\delta^\gamma) \asymp \delta^\gamma, \qquad \bG_0 (1) \asymp \begin{dcases}
	\delta^\gamma & \textfor \gamma < 2s, \\
	\delta^{2s} ( 1+ |\ln \delta| ) & \textfor \gamma = 2s, \\
	\delta^{2s} & \textfor \gamma > 2s.
	\end{dcases}
	\end{equation*}
	It is easy to see that 	\begin{equation}
	\label{eq:comparison weight gamma}
		\bG_0 (\delta^\alpha) \ge c_\alpha \delta^\gamma.
	\end{equation}
	The case of equality  $\alpha+2s=\gamma$   involves a logarithm. Notice that, for the case of the SFL, we have $\gamma = 1$ so  the range $\gamma>2s$  is achieved for $s < \frac 1 2$. We must adapt our data for this setting. This relation can be seen in the singular solution $\bM (1) \asymp \delta^{2s- \gamma -1}$. If $2s > \gamma$ then $\bM (1) \in L^1 (\Omega)$.  This  ranges simplifies the theory immensely. Nevertheless, $\bM(1) \in L^1 (\Omega; \delta^\gamma)$ for all $s$ and $\gamma$.
\end{remark}

\begin{defn}[$\bG_\lambda( \delta^\gamma L^\infty(\Omega))$-weak solutions]
\label{def:lambda-wd}
Suppose $f\in L^1(\Omega;\delta^\gamma)$. We say that $u\in L^1(\Omega;\delta^\gamma)$ is a \emph{$\bG_\lambda(\delta^\gamma  L ^\infty(\Omega))$-weak solution} of \eqref{eq:main-f} if
\begin{equation*}
\int_{\Omega}
	u\psi
\,dx
=\int_{\Omega}
	f\bG_\lambda(\psi)
\,dx,
	\quad \forall\psi\in \delta^\gamma  L^\infty(\Omega).
\end{equation*}
\end{defn}
Note that by \Cref{prop:lin-Lq}, $\bG_\lambda(\psi)$ is well-defined and controlled in $L^\infty(\Omega)$ by $\delta^\gamma$, which is crucial for the right hand integral to be defined.

After establishing the linear theory, one may talk about solutions of \eqref{eq:main-f} in the sense of Green's operator, or briefly, \emph{Green solution}.\footnote{This is in fact a convenient abuse of nomenclature because it is defined by solving the PDE and not by studying the Green function for the operator $\cL-\lambda$.}

\begin{defn}[Green solutions]
Let $u, f\in L^1(\Omega;\delta^\gamma)$. We say that $u$ is a $\bG_0$-\emph{Green solution} of \eqref{eq:main-f} if
\[
u=\lambda \bG_0(u) + \bG_0(f)
	\quad \text{ a.e. in } \Omega,
\]
for $\bG_0$ defined by the integral formula in \eqref{eq:G0-def}. We say that $u$ is a $\bG_\lambda$-\emph{Green solution} of \eqref{eq:main-f} if
\[
u=\bG_\lambda(f)
	\quad \text{ a.e. in } \Omega,
\]
where $\bG_\lambda:L^1(\Omega;\delta^\gamma)\to L^1(\Omega;\delta^\gamma)$ is defined in \Cref{thm:lin}.
\end{defn}

\begin{remark}
	Notice that, in order to define $\bG_0$-Green solutions, $u$ must be in the admissible class of data for $\bG_0$.
\end{remark}

When $f\in L^2(\Omega)$, solutions are in
\[
\bH^{2}_{\cL}(\Omega)
=\set{
	v\in L^2(\Omega):
	\sum_{j\geq1}
	\lambda_j^{2}
	\angles{v,\varphi_j}^2
	<+\infty
},
\]
and are naturally understood in the \emph{spectral} sense.

\begin{defn}[Spectral solutions]
\label{def:spectral}
Suppose $f\in L^2(\Omega)$. %
We say that $u\in \bH^{2}_{\cL}(\Omega)$ is a \emph{spectral solution} of \eqref{eq:main-f} if
\begin{equation*}
(\lambda_j-\lambda)
\int_{\Omega}
    u\varphi_j
\,dx
=\int_{\Omega}
	f\varphi_j
\,dx,
	\quad \forall j\geq1.
\end{equation*}
\end{defn}
\begin{remark}
	We will show below that the eigenfunctions enjoy the optimal boundary regularity $\varphi_j\in \delta^\gamma L^\infty(\Omega)$, therefore these equations are defined for $f \in L^1 (\Omega; \delta^\gamma)$. However, when $f \notin L^2 (\Omega)$, $u$ is not in $\bH^{2}_{\cL} (\Omega)$, so the resulting series may not converge to $u$.
\end{remark}

It is a simple exercise of approximation of the test functions to show that
\begin{thm}[Equivalence of notions of solution]
\label{thm:notions}
Assume {\textrm{(K1)}}. Suppose $f,u\in L^1(\Omega; \delta^\gamma)$. Then the following are equivalent:
\begin{enumerate}
	\item
	$u$ is a $\bG_\lambda$-Green solution of \eqref{eq:main-f}.
	\item
	$u$ is a $\bG_\lambda(\delta^\gamma L^\infty(\Omega))$-weak solution of \eqref{eq:main-f}.
	\item
	$u$ is a $\bG_0(L_c^\infty(\Omega))$-weak solution of \eqref{eq:main-f}.
	\item
	$u$ is a $\bG_0(C_c^\infty(\Omega))$-weak solution of \eqref{eq:main-f}.
	\item
	$u$ is a $\bG_0$-Green solution of \eqref{eq:main-f}.
\end{enumerate}
Suppose, in addition, that $f\in L^2(\Omega)$ and $u\in \bH^{2}_{\cL}(\Omega)$. Then any of the above is equivalent to
\begin{enumerate}
\setcounter{enumi}{5}
\item
	$u$ is a spectral solution of \eqref{eq:main-f}.
\end{enumerate}
\end{thm}

The proofs  are  given in \Cref{sec:notions}, using the linear theory  in  \Cref{thm:lin} (see below) and the integration-by-parts formulae in \Cref{sec:by-parts}.

\subsection{General linear theory}

Secondly, we prove that $\bG_\lambda$ is well-defined in various Lebesgue spaces and establish its mapping properties. %
Denote the eigenvalue of $\cL$ just larger than $\lambda$ by%
\begin{equation}\label{eq:bar-lambda}
\bar\lambda = \lambda_{I+1},
    \qquad \text{where} \qquad
I
=\min\set{ j : \lambda_j\ge \lambda \textand \lambda_{j+1} > \lambda_j }.
\end{equation}
and the distance of $\lambda$ to the nearest eigenvalue by
\begin{equation}\label{eq:d-lambda}
d_\Sigma(\lambda)
=\min_j {|\lambda_j-\lambda|}.
\end{equation}
Notice that, with this choice, $\lambda_{I+1} >  \lambda_I \ge \lambda$.

\begin{thm}[The solution operator $\bG_\lambda$]
\label{thm:lin}
Assume \textrm{(K1)}. Let $\lambda\in\R\setminus\Sigma$. Let $\bar\lambda,d_\Sigma(\lambda)$ be defined in \eqref{eq:bar-lambda}, \eqref{eq:d-lambda} respectively. Then
\begin{enumerate}
\item
	$\bG_\lambda$ is well-defined on $L^1(\Omega;\delta^\gamma)$ and maps continuously from
\begin{align*}
        L^1(\Omega;\bG_0(\delta^\alpha))
&\longrightarrow L^1 (\Omega; \delta^\alpha), && \textfor \alpha > -1-\gamma,\\
L^1(\Omega)
&\longrightarrow L^p(\Omega),
&& \textfor p\in [1, \tfrac{n}{n-2s} ),\\
L^2(\Omega)
    &\longrightarrow \bH^{2}_{\cL}(\Omega),\\
   L^{p_0} (\Omega) & \longrightarrow L^{p_1} (\Omega) && \textfor p_0\in(1, \tfrac n {2s}) \textand 	\tfrac{1}{{p_1}} = \tfrac 1 {p_0} - \tfrac{2s}n, \\
L^q(\Omega)
    &\longrightarrow L^\infty(\Omega),
        && \textfor q\in(\tfrac{n}{2s},+\infty),\\
\delta^\alpha L^\infty (\Omega) &\longrightarrow \bG_0(\delta^\alpha) L^\infty(\Omega), && \textfor \alpha > -1-\gamma.
\end{align*}
with corresponding operator norms depending only on $\bar{\lambda},d_\Sigma(\lambda)$ and, where appropriate, $p$ or $q$.

\item
	$\bG_\lambda$ maps continuously from
	\begin{align*}
L^1(\Omega;\delta^\gamma)
&\longrightarrow
L^{p}_\loc(\Omega),
\quad \textfor
p \in(1,\tfrac{n}{n-2s}),\\
L^1(\Omega;\delta^\gamma)
\cap L^{p_0}_\loc(\Omega)
&\longrightarrow
L^{p_1}_\loc(\Omega),
\quad \textfor
p_0\in (1,\tfrac{n}{2s}) \textand  \tfrac{1}{p_1} = \tfrac{1}{p_0} - \tfrac {2s} n ,\\
L^1(\Omega;\delta^\gamma)
\cap L^{q}_\loc(\Omega)
&\longrightarrow
L^{\infty}_\loc(\Omega),
\quad \textfor q\in(\tfrac{n}{2s},+\infty).
\end{align*}
Furthermore, for $f\in L^1(\Omega;\delta^\gamma)\cap L^\infty_\loc(\Omega)$, $K\Subset K_0\Subset\Omega$ and $u=\bG_\lambda(f)$ we have that
\[
\norm[L^\infty(K)]{u}
\leq C(\bar\lambda,d_\Sigma(\lambda),K,K_0)
    \left(
		\norm[L^\infty(K_0)]{f}
		+\norm[L^1(\Omega)]{f\delta^\gamma}
	\right).
\]
\end{enumerate}
\end{thm}

The $L^2$ theory is classical and the regularity with highly-integrable data follows essentially from the behavior of the Green function. In fact, under stronger assumptions on the Green function, one may even obtain $u\in C^\alpha(\overline{\Omega})$, for $\alpha>\gamma$. Indeed, for RFL and SFL, it follows from \cite[Theorem 1.1]{F} and \cite[Theorem 1.5 (2)]{CSt} respectively.\footnote{In \cite{F} the Morrey space $\cM_\beta$ is used. Note that by \Cref{lem:embed-Lp}, $L^q(\Omega)$ is continuously embedded into $\cM_{n/q}$.} The case of $L^1(\Omega;\delta^\gamma)$ data is an elaboration of the classical paper of Brezis--Cazenave--Martel--Ramiandrisoa \cite{BCMR}. Note that we do not use the maximum principle, in order to cover the case of large $\lambda$ (i.e. $\lambda\geq \lambda_1$). Also, we obtain directly $L^p$-estimates for some $p>1$. The proof is provided in \Cref{sec:lin} for (1) global estimates and \Cref{sec:lin-loc} for (2) local boundedness.

\begin{remark}
Since the fundamental notions and results concerning $\bG_\lambda$ are grouped for easier presentation and are not ordered linearly, we make an explicit note on the logical line of thought.
\begin{enumerate}
\item
	For $\psi\in L^2(\Omega)$, $\zeta=\bG_\lambda(\psi)$ is defined in \Cref{lem:lin-L2} as the explicit spectral solution of $\cL\zeta-\lambda\zeta=\psi$.
\item
	For $\psi$ such that $\psi / \delta^\gamma \in L^\infty(\Omega)$, $\bG_\lambda(\psi)$ also  enjoys the boundary regularity $\bG_\lambda(\psi)/\delta^\gamma\in L^\infty(\Omega)$, and so  do  the eigenfunctions, as shown in \Cref{prop:eigenfunctions regularity}.
\item
	The $\bG_\lambda(\delta^\gamma L^\infty(\Omega))$-weak formulation in \Cref{def:lambda-wd} makes sense.
\item
	The definition of $\bG_\lambda$ is extended to $L^1(\Omega;\delta^\gamma)\to L^1(\Omega; \delta^\gamma )$ in \Cref{prop:lin-L1}.
\item
	An integration-by-parts formula is obtained for $\bG_\lambda$ in \Cref{lem:by-parts-G-lambda}.
\item
	All notions of solutions are shown to be equivalent in \Cref{sec:notions}.
\item
	The remaining theory, including the local bounds, the projection into $E^\perp$  \footnote{orthogonal complement of $E$ in $L^2(\Omega)$}  and the maximum principle, is developed.
\end{enumerate}
\end{remark}

\begin{remark}
When $\lambda=\lambda_i$ is an eigenvalue, the classical Fredholm alternative applies: $\bG_\lambda$ is well-defined on the $E_i^\perp$, the space of functions orthogonal to all eigenfunctions associated to $\lambda_i$, and in which case the solution space has the same dimension as $E_i$. Results can be obtained as a limiting case of \Cref{thm:plin} and \Cref{thm:plin-2}.
\end{remark}

\subsection{Projected linear theory}

Our third result concerns the solution operator $\bG_\lambda$ on the subspace $E^\perp\subset L^2(\Omega)$, whose elements are orthogonal to any eigenfunction with eigenvalue at most $\bar\lambda$. For further clarity, let us write
\begin{equation*}
E = \spanned \{\varphi_1, \ldots, \varphi_ {I}\}, \qquad E^\perp = \spanned\{ \varphi_{I+1} , \varphi_{I+2}, \ldots \}
\end{equation*}
where $\lambda\leq \lambda_I<\lambda_{I+1}=\bar\lambda$ and $\spanned$ is taken topologically in $L^2 (\Omega)$. We prove uniform estimates that are independent of $d_\Sigma(\lambda)$. Furthermore, let us define
\begin{equation*}
\widetilde{ E^{\perp} } = \left\{  f  \in L^1 (\Omega; \delta^\gamma) : \int_\Omega f \varphi_j = 0 , \text{for } j = 1, \cdots, {I}  \right\}.
\end{equation*}
It is easy to see that $$\widetilde{ E^{\perp} }  = \overline{ E^\perp }^{L^1(\Omega;\delta^\gamma)}.$$
Indeed, the backward inclusion follows by taking the $L^1(\Omega;\delta^\gamma)$-closure in $E^\perp \subset \widetilde{E^\perp}$ (note that the pairing with eigenfunction is continuous in $L^1(\Omega;\delta^\gamma)$), and the converse follows by considering the projection in $E^\perp$ of the sequence $f_k=(|f|\wedge k)\sign(f)\in L^2(\Omega)$, which converges to $f$ in $L^1(\Omega;\delta^\gamma)$.

\begin{thm}[The solution operator $\bG_\lambda$ on $E^\perp$]
\label{thm:plin}
Assume \textrm{(K1)}. Let $\lambda\in\R\setminus\Sigma$. Let $\bar\lambda$ be defined in \eqref{eq:bar-lambda}. Then $\bG_\lambda$ maps continuously from
\begin{align*}
L^1(\Omega;\bG_0(\delta^\alpha))
    \cap \widetilde{E^\perp}
&\longrightarrow
L^1 (\Omega; \delta^\alpha)
    \cap \widetilde{E^\perp},
    && \textfor \alpha > -1-\gamma,\\
L^1(\Omega)\cap \widetilde{ E^{\perp} } &\longrightarrow  L^p(\Omega)\cap \widetilde{ E^{\perp} }, && \textfor p \in [1,\tfrac{n}{n-2s}),\\
L^2(\Omega)\cap E^\perp
    &\longrightarrow \bH^{2}_{\cL}(\Omega)\cap E^\perp,\\
    L^{p_0} (\Omega) \cap   \widetilde{ E^{\perp} }  &\longrightarrow %
L^{p_1}(\Omega)\cap \widetilde{E^\perp}
, && \textfor p_0 \in(1,\tfrac n {2s}) \textand \tfrac{1}{{p_1}} = \tfrac 1 {p_0} - \tfrac{2s}n, \\
L^q(\Omega)\cap E^\perp
    &\longrightarrow L^\infty(\Omega) \cap E^\perp,
	&& \textfor q\in(\tfrac{n}{2s},+\infty),\\
\delta^\alpha L^\infty (\Omega)
    \cap E^\perp
&\longrightarrow
    \bG_0(\delta^\alpha) L^\infty(\Omega)
    \cap E^\perp,
    && \textfor \alpha > -1-\gamma.
\end{align*}
with corresponding operator norms depending only on $\bar{\lambda}$ and, where appropriate, $p$ or $q$.
\end{thm}

By the orthogonal decomposition $L^2(\Omega)=E\oplus E^\perp$, one writes down the solution of the eigenvalue problem as an explicit part in $E$ and a uniformly controlled part in $E^\perp$. When $f \in L^2 (\Omega)$ we can write
\begin{equation*}
	f =  \sum_{j=1} ^{{I}}  {
		\angles{f,\varphi_j}
	}\varphi_j  + \sum_{j=I+1} ^{+\infty}  {
		\angles{f,\varphi_j}
	}\varphi_j = P_E (f) + P_{E^\perp} (f).
\end{equation*}
If $f \in L^1 (\Omega; \delta^\gamma)$ this projection is not possible in general. Nevertheless, since $\varphi_j \in \delta^\gamma L^\infty (\Omega)$ for any eigenspace we have the well defined projection
\begin{eqnarray*}
	P_{E_i} : L^1 (\Omega; \delta^\gamma) &\longrightarrow& E_i \subset \delta^\gamma L^\infty (\Omega) \\
	f &\longmapsto& \sum_{j:\lambda_j = \lambda_i} \langle f , \varphi_j \rangle \varphi_j .
\end{eqnarray*}
 and
 \begin{equation*}
 	P_E (f) = \sum_{j=1} ^{{I}}  {
 		\angles{f,\varphi_j}
 	}\varphi_j .
 \end{equation*}
We can define
\begin{equation}
	\label{eq:f perp}
	f^\perp = f - P_E(f) = f - \sum_{j=1} ^{{I}}  {
		\angles{f,\varphi_j}
	}\varphi_j .
\end{equation}
It is clear that $f^\perp \in L^1 (\Omega; \delta^\gamma)$. In fact, $f^\perp \in \overline{ E^\perp }^{L^1(\Omega;\delta^\gamma)}$ and, for $f \in L^2 (\Omega)$, $f^\perp = P_{E^\perp} (f)$.

\begin{thm}[Projection and uniform estimates]
\label{thm:plin-2}
Assume \textrm{(K1)}. Let $\lambda\in\R\setminus\Sigma$. Let $b,\bar\lambda$ be defined in \eqref{eq:b}, \eqref{eq:bar-lambda} respectively. For $f\in L^1(\Omega;\delta^\gamma)$ let $u\in L^1(\Omega; \delta^\gamma)$ be the unique $\bG_\lambda( \delta^\gamma L^\infty(\Omega))$-weak solution of
\[\begin{cases}
{\cL}u-\lambda u=f
    & \textin \Omega\\
Bu=0
    & \texton \p\Omega.\\
\end{cases}\]
Then, for $f^\perp$ given by \eqref{eq:f perp} and $u^\perp = \bG_\lambda (f^\perp)$, i.e. the unique solution of
\[\begin{cases}
{\cL}u^\perp -\lambda u^\perp=f^\perp
& \textin \Omega\\
Bu=0
& \texton \p\Omega,\\
\end{cases}\]
we have that
\[
u  = \sum_{j=1}^{{I}}
    \dfrac{
        \angles{f,\varphi_j}
    }{
        \lambda_j-\lambda
    }\varphi_j + u^\perp.
\]
We have the uniform estimate
\[
	\norm[L^1(\Omega)]{u^\perp \delta^\gamma  }
\leq C(\bar\lambda)
	\norm[L^1(\Omega)]{f\delta^\gamma}.
\]
If $f \in L^{1}(\Omega)$ then, for any $p \in [1, \frac{n}{n-2s})$ we have that
\begin{equation*}
\norm[L^{1}(\Omega)]{u^\perp }
\leq C(p,\bar\lambda)
\norm[L^{p}(\Omega)]{f},
\end{equation*}
If $f \in L^{p_0}(\Omega)$ then, for any $\frac{1}{p_1} = \frac{1}{p_0} - \frac{2s}n$ we have that
\begin{equation*}
	  \norm[L^{p_1}(\Omega)]{u^\perp }
	\leq C(p_0,p_1,\bar\lambda)
	\norm[L^{p_0}(\Omega)]{f},
\end{equation*}
Moreover, if $f\in L^\infty_\loc(\Omega)$, then for any $K\Subset K_0\Subset \Omega$,
\[
\norm[L^\infty(K)]{u^\perp}
\leq C(\bar\lambda,K,K_0)
	\left(
        \norm[L^\infty(K_0)]{f}
        +\norm[L^1(\Omega)]{f\delta^\gamma}
    \right).
\]
\end{thm}

\begin{remark}
	The uniform dependence on $\bar \lambda = \lambda_{I+1}$ is actually a dependence on $\lambda_{I+1} - \lambda_{I}$.
\end{remark}

These closely related main theorems are proved in \Cref{sec:plin}. While the bootstrap regularity arguments remain the same as in \Cref{thm:lin}, the $L^2$-norm of $u^\perp$ can now be bounded independently of $d_\Sigma(\lambda)$. In the (weighted) $L^1$-theory, it is impossible to write down the full eigenfunction expansion. Nonetheless, it still makes sense to project an $L^1(\Omega;\delta^\gamma)$-function into a finite dimensional subspace. Such interplay of $L^1$- and $L^2$-theory is demonstrated in the proof of \Cref{thm:plin-L1}, which is the cornerstone of this article. As the uniform estimate is available only in the orthogonal complement $E^\perp$, we absolutely need to transfer the orthogonality to the test function. This simple while decisive step is seen in \eqref{eq:plin-weak-perp}.

The basic idea behind, i.e. to single out ``bad'' directions, runs in parallel with the finite dimensional Lyapunov--Schmidt reduction in the construction of solutions of PDEs. A nonlinear analogue can be seen in the recent paper \cite{DDMR}, where the authors construct helical vortex filaments in the three-dimensional Ginzburg--Landau equation, and the orthogonality of the larger part of the nonlinear perturbation is crucially used.

\begin{remark}
	\label{rem:Fredholm}
A natural question is what happens as $\lambda \to \lambda_i$. Since we are going to move $\lambda$, let us be precise about notation. The value $\overline \lambda$ was defined for $\lambda $ fixed.  We want to be able approach $\lambda_i$ from above or below, let $\overline \lambda = \lambda_{I+1} > \lambda_i$.
Looking at the decomposition $f = P_E (f) + f^\perp$ giving a solution $u = \bG_\lambda (P_E (f)) + \bG_\lambda (f^\perp)$, the estimates above tell us that $\bG_\lambda(f^\perp)$ is uniformly bounded in the corresponding spaces. On the contrary, looking at
\begin{equation*}
	\bG_\lambda (P_E (f)) = \sum_{j=1}^{{I}}
	\dfrac{
		\angles{f,\varphi_j}
	}{\lambda_j - \lambda} \varphi_j  = \sum_{ \substack {j=1 \\ \lambda_j \ne \lambda_i  }}^{{I}}\dfrac{
	\angles{f,\varphi_j}
}{\lambda_j - \lambda} + \dfrac{
1
}{\lambda_i  - \lambda} \sum_{ \substack {j: \lambda_j = \lambda_i }}{
\angles{f,\varphi_j}
	} \varphi_j,
\end{equation*}
there is a blow-up at any point such that $P_{E_i} (f) (x) = \sum_{ \substack {j: \lambda_j = \lambda_\ell }}{
	\angles{f,\varphi_j}
} \varphi_j (x) \ne 0$. We recover the classical Fredholm alternative:
\begin{quote}
	As $\lambda \to \lambda_i$,
	\begin{enumerate}
		\item If $P_{E_i}(f) = 0$, then $\bG_\lambda (f)  \to \bG_{\lambda_i} (f^\perp) = \bG_{\lambda_i} (f)$ in $L^1 (\Omega; \delta^\gamma)$
		\item  If $P_{E_\ell}(f) \ne 0$, then $|\bG_\lambda (f)|  \to + \infty$ on a set of positive measure.
	\end{enumerate}
\end{quote}
Notice that, if $\lambda \to \lambda_i^{\pm}$, we can estimate the sign of the blow-up at a point $x$, by combining the sign of $P_{E_i}(f)(x)  \ne 0 $. At the end we point out that, since the eigenfunctions are linearly independent $P_{E_i}(f) = 0$ if and only if $\langle f, \varphi_j \rangle$ for all $j$ such that $\lambda_i = \lambda_j$.
\end{remark}

\subsection{Maximum principle}

When $\lambda<\lambda_1$, one obtains a maximum principle for weighted $L^1$-solutions, even though the positivity via Poincar\'{e} inequality makes sense only in the $L^2$-setting. This is our fourth result.

\begin{thm}[Maximum principle for $L^1$-solutions]
\label{thm:MP}
Assume \textrm{(K1)}. Let $\lambda\in(-\infty,\lambda_1)$. Let $b$ be defined in \eqref{eq:b}. Suppose $f\in L^1(\Omega;\delta^\gamma)$, and $u\in L^1(\Omega;\delta^\gamma)$ is a $\bG_\lambda(\delta^\gamma L^\infty(\Omega))$-weak solution of
\[\begin{cases}
{\cL}u-\lambda u=f
    & \textin \Omega\\
B u=0
    & \texton \p\Omega.\\
\end{cases}\]
Then, if $f\geq 0$ a.e. in $\Omega$, we have $u\geq 0$ a.e. in $\Omega$.
\end{thm}
The proof is given in \Cref{sec:MP}. A related maximum principle can be found in \cite{CD}. As customary, for $L^1$-solution one applies an $L^2$ maximum principle to the test function; the negative part of an $L^2$-solution is a valid test function. Classically, one uses the  pointwise  fact that $\nabla u_- \cdot \nabla u_+=0$ when $u\in H^1$  to deal with the integral cross term $\int \nabla u_- \cdot \nabla u$.
In the nonlocal case, one choice of dealing with the cross term is by Kato's inequality, tested against the solution itself, as can be seen in the singular integral formulation. Such cross term appears for example in the recent breakthrough related to the De Giorgi conjecture \cite{FS}. For the full pointwise  Kato's  inequality see \cite{CV,CSi,DGV} for RFL and \cite{AD} for SFL.  In our integral equation expressed in terms of the non-negative Green function, 
the cross term clearly has a sign. The burden of proof shifts to the Poincar\'{e} inequality, which has already been established in \cite{BFV}.

The maximum principle for the fractional Laplacian ($\lambda=0$) or associated coercive operators (including $\lambda<0$) are well-known. In the case of the fractional Laplacian it follows automatically from the condition $\cG_0 (x,y) \ge 0$. See for example \cite{Sil,AD,DGV}.

\subsection{Inhomogeneous eigenvalue problem}
We are finally addressing the problem of singular boundary data, i.\,e., large solutions. Let $g\in L^1(\Omega;\delta^\gamma)\cap L^\infty_\loc(\Omega)$ and $h\in C(\p\Omega)$. Consider $\bG_0(C_c^\infty(\Omega))$-weak solutions of
\begin{equation}\label{eq:main}\begin{cases}
{\cL{v}}-\lambda v=g
	& \textin \Omega\\
Bv=h
	& \textin \p\Omega.\\
\end{cases}\end{equation}
Let us consider the large $\cL$-harmonic function $v_h\in \delta^{-b}L^\infty(\Omega) \subset L^1 (\Omega; \delta^\gamma )$ solving
\begin{equation}\label{eq:vh}\begin{cases}
{\cL}v_h=0
& \textin \Omega\\
Bv_h=h
& \texton \p\Omega.\\
\end{cases}
\end{equation}
Existence, uniqueness and kernel representation has been extensively studied in \cite{A} (RFL), \cite{AD} (SFL) and \cite{AGV} (unified approach). The solution is obtained by either justifying the Martin kernel $\bM$ representation or as a limit of the interior theory. We present the corresponding results in \Cref{sec:martin kernel}.
The difference $u = v - v_h$ solves
\begin{equation*}\begin{cases}
{\cL{u}}-\lambda u=g + \lambda v_h
& \textin \Omega\\
Bu=0
& \textin \p\Omega.\\
\end{cases}\end{equation*}

Since $g + \lambda v_h \in L^1 (\Omega; \delta^\gamma)$, when $\lambda \notin \Sigma$ this problem has a unique solution  by the  theory above, given by $u = \bG_\lambda ( g + \lambda v_h )$. Then
\begin{equation}
	v = v_h + \bG_\lambda ( g + \lambda v_h ).
\end{equation}

Using \Cref{thm:plin-2}, we obtain, as our third main result, the structure of solutions and Fredholm alternative  as  $\lambda\to\Sigma$.

\begin{thm}\label{thm:blow}
Assume \textrm{(K1)} and \textrm{(K2)}. Let $\lambda\in\R\setminus\Sigma$. Let $b,\bar\lambda$ be defined in \eqref{eq:b}, \eqref{eq:bar-lambda} respectively. Let
\begin{equation}\label{eq:g-space}
g\in
	L^1(\Omega;\delta^\gamma)
	\cap L^\infty_\loc(\Omega).
\end{equation}
Given $h\in C(\p\Omega)$, let
\begin{equation}\label{eq:vh-space}
v_h\in
	\delta^{-b}L^\infty(\Omega)
\end{equation}
be the large $\cL$-harmonic function solving \eqref{eq:vh}. %
Then the following hold.
\begin{enumerate}
\item (Existence-uniqueness)
For any $\lambda\in\R\setminus\Sigma$, \eqref{eq:main} has a unique $\bG_0(C_c^\infty(\Omega))$-weak solution
\[
v_\lambda
\in L^1(\Omega;\delta^\gamma)
	\cap L^\infty_\loc(\Omega).
\]
\item (Representation formula)
The solution $v_\lambda$ is given by
\[
v_\lambda
=v_h+\sum_{\varphi_j\in E}
    \dfrac{
        \angles{g+\lambda v_h,\varphi_j}
    }{
        \lambda_j-\lambda
    }\varphi_j
    +u^\perp,
\]
where $u^\perp\in L^p(\Omega)\cap L^\infty_\loc(\Omega)\cap \widetilde{E^\perp}$, for all $p\in[1,\frac{n}{n-2s})$, is the unique solution (in any of the sense \eqref{eq:sense-1}--\eqref{eq:sense-5} in \Cref{thm:notions}) of
\begin{equation}\label{eq:perp}\begin{cases}
{\cL}u^\perp-\lambda u^\perp
=g^\perp + \lambda v_h^\perp
    & \textin \Omega\\
Bu^\perp=0
	& \texton \p\Omega,\\
\end{cases}\end{equation}
where
\[
v_h^\perp
=v_h-\sum_{\varphi_j\in E}
        \angles{v_h,\varphi_j}
        \varphi_j,
    \quad
g^\perp
=g-\sum_{\varphi_j\in E}
        \angles{g,\varphi_j}
        \varphi_j.
\]
\item (Uniform estimates)
We have that
\[
\norm[L^1(\Omega)]{u^\perp \delta^\gamma }
\leq C\left(
\norm[L^1(\Omega)]{g}
+\norm[L^\infty(\p\Omega)]{h}
\right),
\]
where $C=C(\bar\lambda)$, and for any $K\Subset K_0\Subset \Omega$ we have that
\begin{equation}\label{eq:u-perp-loc}
\norm[L^\infty(K)]{v_h}
+\norm[L^\infty(K)]{u^\perp}
\leq C\left(
        \norm[L^\infty(K_0)]{g}
        +\norm[L^1(\Omega)]{g\delta^\gamma}
        +\norm[L^\infty(\p\Omega)]{h}
    \right),
\end{equation}
for $C=(\bar\lambda,K,K_0)$.
\item (Fredholm alternative)
For $i=1,2,\dots$, let $E_i$ be the eigenspace associated to the $\lambda_i\in\Sigma$. Let $\lambda \to \lambda_i$. %
Then exactly one of the following holds.
\begin{enumerate}
	\item If $P_{E_i} (g+\lambda_i v_h) = 0$, then
	\begin{equation*}
		v_\lambda \longrightarrow  v_h - P_{E_i} (v_h) + \bG_{\lambda_i}  (  g + \lambda_i v_h  ).
	\end{equation*}
	in $L^1 (\Omega; \delta^\gamma)$.
	\item If $P_{E_i} (g+\lambda_i v_h) \ne 0$ then $|v_\lambda| \to +\infty$ uniformly in a set of positive measure. Furthermore, let, for $\varepsilon>0$
	\[
	A_{i}^+
	:=\set{
		P_{E_i} (g + \lambda_i v_h)>\varepsilon
	},
	\quad
	A_{i}^-
	:=\set{ P_{E_i} (g + \lambda_i v_h) <-\varepsilon
	}.
	\]
	Then for any $K\Subset\Omega$,
	\[
	\begin{cases}
	\displaystyle
	\lim_{\lambda\nearrow\lambda_i}
	\inf_{K\cap A_i^+}v_\lambda
	=\lim_{\lambda\searrow\lambda_i}
	\inf_{K\cap A_i^-}v_\lambda
	=+\infty,\\
	\displaystyle
	\lim_{\lambda\nearrow\lambda_i}
	\sup_{K\cap A_i^-}v_\lambda
	=\lim_{\lambda\searrow\lambda_i}
	\sup_{K\cap A_i^+}v_\lambda
	=-\infty.
	\end{cases}
	\]
\end{enumerate}

\item (Global blow-up)
If $g\geq0$ in $\Omega$ and $h> 0$ on $\p\Omega$,
and $i=1$, then
\[
\lim_{\lambda\nearrow\lambda_1}
    \inf_{\Omega}v_\lambda
=+\infty.
\]
\end{enumerate}
\end{thm}

We give the proof in \Cref{sec:blow}. The main idea is straightforward: by subtracting the large $\cL$-harmonic function we obtain an equation with zero weighted (by $\delta^b$) trace, and a right hand side in $\delta^{-b}L^\infty(\Omega)\subset L^1(\Omega;\delta^\gamma)$. Then the previously developed linear theory in \Cref{thm:plin-2} applies. The interior blow-up behaviors as $\lambda\to\Sigma$ can be seen from the explicit part with the low eigenfunctions, because the orthogonal part is shown to be bounded in compact subsets.

\begin{remark}
Note that by Fubini's Theorem, one may rewrite
\[
\angles{\lambda_i v_h,\varphi_j}
=\int_{\p\Omega}
	D_\gamma \varphi_j(z)
	h(z)
\,d\cH^{n-1}(z),
	\quad \textfor \varphi_j\in E_i.
\]
\end{remark}

\begin{remark}
It is possible that the solution does not blow up in the interior, i.e. $A_+=A_-=\emptyset$. For example, this happens when $n=3$, $\Omega=B_1$, $g=0$, $h=1$ and $i=2$, in which case $U=(1-|x|^2)_+^{s-1}$ \cite{Bogdan} and all eigenfunctions corresponding to $\lambda_2$ are anti-symmetric \cite{Ferreira}. The anti-symmetry of higher eigenfunctions for the restricted fractional Laplacian appears to be a challenging problem.
\end{remark}

\subsection{Convergence and compactness as $s\nearrow1$}

As the parameter $s$ tends to $1$, we also have $\gamma\to 1$. Thus the exponent $b=1-2s+\gamma$ tends to $0$. Therefore the limit no longer blows up at the boundary. In fact, an $L^2$-theory is sufficient due to the slow blow-up rate for $s$ close to $1$.

Under additional assumptions on the Green function, we show various convergence results and prove that the solution of the inhomogeneous eigenvalue problem converges to that of the corresponding Dirichlet problem. Since the limiting solution is bounded, ``large eigenfunctions'' are purely nonlocal objects.

These results are presented in \Cref{sec:sto1}. One main ingredient for the strong convergence is the compactness of the sequence of Green's operator on bounded $L^2$ functions, in the spirit of the Riesz--Fr\'{e}chet--Kolmogorov theorem.

\subsection{Notations and organization}
\label{sec:ideas}

For convenience, we list some notations used throughout the paper.

\begin{itemize}
\item $\Omega^c:=\R^n\setminus\overline{\Omega}$ is the complement of $\Omega$.
\item $K_1 \Subset K_0$ means that the closure of $K_1$ is contained in the interior of $K_0$.
\item
    $s$ is half the order of $\cL$, and ranges over $(0,1)$ for RFL and over $(\frac12,1)$ for SFL.
\item $\gamma$ is the parameter for the boundary behavior; $\gamma=s$ for RFL and $\gamma=1$ for SFL.
\item $\bar\lambda$ is the eigenvalue just above $\lambda$, defined in \eqref{eq:bar-lambda}.
\item
    $(\lambda_j,\varphi_j)$ are the eigenpairs, with $\lambda_j\in\Sigma$ repeated according to multiplicity.
\item $d_\Sigma(\lambda)$ is the distance of $\lambda$ to the spectrum, defined in \eqref{eq:d-lambda}.
\item $E_i$ is the eigenspace associated to $\lambda_i$.
\item $E$ is the span of all eigenfunctions with eigenvalue at most $\bar\lambda$.
\item $E^\perp$ is the orthogonal complement of $E$ in $L^2(\Omega)$.
\item $\angles{u,\varphi}=\int_{\Omega}u\varphi\,dx$ is the $L^2(\Omega)$ inner product.
\item
    $\cG_0(x,y)$ is the Green function for $\cL$.
\item
    $\bG_0(u)(x)=\int_{\Omega}\cG_0(x,y)u(y)\,dy$ is the  Green  operator for $\cL$.
\item
    $\bG_\lambda(u)$ is the solution operator for $\cL-\lambda$.
\item
    $a\wedge b$ is the minimum of the real numbers $a$ and $b$.
\item
    $a_+=\max\set{a,0}$ is the positive part of $a$.
\item
    $f \asymp g$ means $C^{-1}f \leq g \leq Cf$, for positive functions $f$ and $g$.
\item $\eta_\eps=\eps^{-n}\eta(\eps^{-1}\cdot)$ is the standard mollifier, where $\eta\in C_c^\infty(B_1)$ is an $L^1$-normalized bump function.
\item Constants depending only on $n$, $s$, $\Omega$, $\gamma$ are considered universal and denoted generically by $C$. Any additional dependence is indicated.
\end{itemize}

The paper is organized as follows. We derive regularity estimates for the Green function in \Cref{sec:G0}, and apply them to establish a linear theory for $L^1$-solutions in \Cref{sec:general-lin}. %
We show the equivalence of various notions of solution, and obtain local boundedness of $L^1$-solutions. Then we derive the $L^1$-linear theory in the $L^2$-projected space in \Cref{sec:plin}. In \Cref{sec:MP} we prove a maximum principle for $L^1$-solutions. All these are applied in \Cref{sec:blow} to prove the target theorem, classifying blow-up phenomena of the nonlocal eigenvalue problem.
We recover the classical eigenvalue problem as $s\nearrow 1$ in \Cref{sec:sto1}. In \Cref{sec:compactness}, we discuss the natural Hilbert spaces and some compactness properties.
We give an explicit expression for the weighted trace for the RFL in \Cref{sec:RFL-trace} and collect elementary embedding results into Morrey spaces in \Cref{sec:embed}.

\section{Estimates for $\lambda = 0$}
\label{sec:G0}

 We assume (K1) throughout this section. 

\subsection{Regularisation between weighted $L^p$ spaces}

\subsubsection{Regularisation between $L^1$ weight spaces.}
The first item to notice is that, as shown in \cite{AGV}, we have that
\begin{prop}
	\label{prop:regularisation between L^1 weight spaces}
	The  Green  operator $\bG_0$ for $\cL$ is continuous from
	\begin{equation*}
	\bG_0 : L^1 (\Omega; \bG_0 (\delta^\alpha)) \longrightarrow L^1 (\Omega; \delta^\alpha).
	\end{equation*}
	for $\alpha > -1-\gamma$, with its operator norm  equal to  $1$.
\end{prop}

\begin{proof}
If $u=\bG_0(f)$, then one can take $\psi = \sign (u) \delta^\alpha $ as a test function to see
\begin{equation*}
	\int_\Omega |u| \delta^\alpha = \int_\Omega f \mathbb G_0 (\textrm{sign} (u) \delta^\alpha ) \le \int_\Omega |f| \mathbb G_0 (\delta^\alpha).
\end{equation*}
 The choice $f=1$ shows that equality can hold. 
\end{proof}

\begin{remark} For $\alpha = \gamma$ we have that
	\begin{equation*}
		\bG_0 : L^1 (\Omega; \delta^\gamma) \to L^1 (\Omega; \delta^\gamma).
	\end{equation*}
	Since, for $\alpha > -1-\gamma$, $\bG_0 (\delta^\alpha) \ge c_\alpha \delta^\gamma$, we have that $L^1 (\Omega; \delta^\alpha ) \hookrightarrow L^1 (\Omega; \delta^\gamma)$. Hence,  $\bG_0(f)$ is always in $L^1 (\Omega; \delta^\gamma)$ whenever $f\in L^1(\Omega;\delta^\alpha)$ .
\end{remark}

\begin{remark}\label{rem:3.2}
	Since, by \eqref{eq:improvement of weights} we reduce the power of the weight needed, we always improve the integrability. In particular, if $\gamma < 2s$ then
	\begin{align*}
	\bG_0 : L^1 (\Omega; \delta^\gamma) \longrightarrow L^1 (\Omega)
	\end{align*}
	and, if $\gamma > 2s$ then
	\begin{align*}
	\bG_0 : L^1 (\Omega; \delta^{2s} ) &\longrightarrow L^1 (\Omega ), \\
	\bG_0 : L^1 (\Omega;\delta^\beta) &\longrightarrow L^1 (\Omega; \delta^{\beta - 2s} ), & &\text{for any } \beta \in(  \gamma , 2s) , \\
	\bG_0 : L^1 (\Omega;\delta^\gamma) &\longrightarrow L^1 (\Omega; \delta^{\alpha} ), & &\text{for any } \alpha > \gamma - 2s.
	\end{align*}
	Again, $\gamma < 2s$, all admissible data give $L^1$ solutions.
\end{remark}

\subsubsection{Regularisation from $L^p$ to $L^q$ and Marcinkiewicz spaces}
We start with a  definition.
\begin{defn}
	Let $u$ be a measurable function in $\Omega$. Let $p \in (1,+\infty)$ and let $p'$ be its conjugate exponent. We define the Marcinkiewicz norm  $\norm[M^p(\Omega)]{\cdot}$  as
	\begin{equation*}
	\| u \|_{M^p(\Omega)} = \sup_{ \substack{ K  \subset  \Omega \\ \text{measurable} } }  \frac{ \int_K |u (x)| dx}{ |K|^{\frac 1 {p'}}  }.
	\end{equation*}
	We define the Marcinkiewicz space $M^p (\Omega)$ as the space of  measurable  functions with bounded $M^p(\Omega)$ norm.
\end{defn}
Notice that the extension by zero produces an isometric embedding $M^p (\Omega) \hookrightarrow M^p (\mathbb R^n)$. The embedding $M^p (\Omega) \hookrightarrow L^q(\Omega)$ for $q < p$ is well known (see, e.g. \cite[Lemma A.2]{BBC}).

The aim of this subsection is to prove
\begin{prop}
	\label{prop:regularisation L1 to Marcinkiewicz}
	The  Green  operator $\bG_0$ for $\cL$ is continuous from
	\begin{equation}
	\bG_0: L^{1}(\Omega)
	\longrightarrow
	M^{\frac {n}{n-2s}  } (\Omega).
	\end{equation}
	In particular $\bG_0: L^{1}(\Omega)
	\longrightarrow
	L^p (\Omega)$
	and
	\begin{align}
	\bG_0: L^{p_0}(\Omega)
	\longrightarrow
	L^{p_1} (\Omega)  & \textfor  p_0 \in (1, \tfrac{n}{2s}) \textand	\tfrac{1}{{p_1}} = \tfrac 1 {p_0} - \tfrac{2s}n.
	\end{align}
\end{prop}
\begin{remark}
	\label{rem:G_0^k from L^1 to L^q}
	Notice that $p_1 = p_0 \frac{n}{n-2sp_0}$. As $p_0 \to \frac{n}{2s}$ then $p_1 \to +\infty$. Hence, for every $q\ge 1$, there exists $k$ finite such that $\bG_0^k : L^1 (\Omega) \to L^{q} (\Omega)$.
\end{remark}

In \cite{GV}, the authors give a direct computation of the second result. We will give better results using stronger  machinery. 
For $p = 1$, we will give a different proof applying of the convolution estimates to Marcinkiewicz spaces given in \cite[Appendix]{BBC}. For $p > 1$ we apply  Hardy--Littlewood--Sobolev estimates for Riesz potentials.
\begin{prop} [Chapter V, \S 1, Theorem 1 in \cite{Stein}]
	\label{prop:Riesz potentials}
	Let $\alpha\in(0,n)$ and
\begin{equation*}
 \mathcal{I}_\alpha  (f) (x) = \int_{\mathbb R^n} \frac{f(y)}{|x-y|^{\alpha} } \,dy.
\end{equation*}
Then, for $p_0 \in (1, \frac n {2s})$,
\begin{equation*}
\|  \mathcal{I}_\alpha  (f) \|_{M^{\frac n {n-2s} }(\mathbb R^n)} \le C \| f \|_{L^1 (\mathbb R^n)}, \qquad
\|  \mathcal{I}_\alpha  (f) \|_{L^{p_1}} \le C(p_0) \| f \|_{L^{p_0} (\mathbb R^n)}
\end{equation*}
where $\frac 1 {p_1} = \frac 1 {p_0} + 1 - \frac \alpha n$.
\end{prop}

\begin{proof}[Proof of \Cref{prop:regularisation L1 to Marcinkiewicz}]
We have that
\begin{equation*}
	|u(x)| \le C \int_\Omega \frac{|f(y)|}{|x-y|^{n-2s}} dy = C \int_{\mathbb R^n} \frac{|f(y)| \chi_\Omega }{|x-y|^{n-2s}} dy = C  \mathcal{I}_{n-2s}  (|f| \chi_\Omega).
\end{equation*}
Therefore,
\begin{equation*}
	\|u \|_{M^{\frac n {n-2s}} (\Omega )} \le C \| f \chi_ \Omega \|_{L^1 (\mathbb R^n)} =  C \| f  \|_{L^1 (\mathbb R^n)}, \qquad \|u \|_{L^{p_1}(\Omega)} \le C(p_0) \| f \|_{L^{p_0} (\Omega)}.
\end{equation*}
This completes the proof.
\end{proof}

\subsubsection{Regularisation from $L^{p}$ to $L^{\infty}$}
By duality from \Cref{prop:regularisation L1 to Marcinkiewicz}, we have that
\begin{prop}
	\label{prop:G_0 is L^p to L^infty}
	The  Green  operator $\bG_0$ for $\cL$ is continuous from
	\begin{equation}
	\bG_0: L^{q}(\Omega)
	\longrightarrow
	L^\infty (\Omega), \qquad q \in( \tfrac{n}{2s},+\infty).
	\end{equation}
	In particular
	\begin{equation}
		\norm[L^\infty(\Omega)]{\bG_0(f)}
		\leq
		C(q_0)\norm[L^{q_0}(\Omega)]{f}.
		\label{eq:G0-Lq}
	\end{equation}
\end{prop}

\subsubsection{Improvement of weighted integrability}
It is immediate to prove that
\begin{prop}
	\label{prop:regularisation between L^infty weight spaces}
 Let $\alpha>-1-\gamma$. 
	The  Green  operator $\bG_0$ for $\cL$ is continuous from
	\begin{equation*}
	\bG_0 : \delta^\alpha L^\infty (\Omega) \longrightarrow \bG_0( \delta^\alpha ) L^\infty (\Omega),
	\end{equation*}
with operator norm  equal to   $1$.
\end{prop}

\begin{proof}
	Assume that $|f| \le C \delta^\alpha$. Then
	\begin{equation*}
		\left| \frac{\bG_0(f) } {\bG_0(\delta^\alpha)} \right| \le 	\frac{\bG_0(|f|) } {\bG_0(\delta^\alpha)} \le \left \| \frac{f}{\delta^\alpha} \right\|_{L^\infty ( \Omega )}  \frac{\bG_0(\delta^\alpha ) } {\bG_0(\delta^\alpha)}
 =  \left \| \frac{f}{\delta^\alpha} \right\|_{L^\infty ( \Omega )} .
	\end{equation*}
 Equality is achieved by setting $f=\delta^\alpha$. 
\end{proof}

We recall that $\bG_0( \delta^\alpha ) \asymp \delta^{ (\alpha + 2s) \wedge \gamma }$ if $\alpha + 2s \ne \gamma$.

In particular, for $\gamma < 2s$ we have that
\begin{align*}
\bG_0 : L^\infty (\Omega ) &\longrightarrow \delta^{\gamma} L^\infty (\Omega )
\end{align*}
and for $\gamma > 2s$ we have that
\begin{align*}
\bG_0 : L^\infty (\Omega ) &\longrightarrow \delta^{2s} L^\infty (\Omega ), \\
\bG_0 : \delta^{\alpha}  L^\infty (\Omega) &\longrightarrow \delta^{\alpha + 2s} L^\infty (\Omega), & &\text{for any } \alpha \in(  0 , \gamma - 2s ), \\
\bG_0 : \delta^{\alpha} L^\infty (\Omega) &\longrightarrow \delta^\gamma L^\infty (\Omega ), & &\text{for any } \alpha > \gamma - 2s.
\end{align*}

\subsubsection{Local integrability}

\begin{prop}
	\label{prop:local boundedness}
	The operator $\bG_0$ is continuous from
	\begin{align*}
	L^1(\Omega;\delta^\gamma)
	&\longrightarrow
	L^{p}_\loc(\Omega),
	\quad \textfor
	p \in (1,\tfrac{n}{n-2s}),\\
	L^1(\Omega;\delta^\gamma)
	\cap L^{p_0}_\loc(\Omega)
	&\longrightarrow
	L^{p_1}_\loc(\Omega),
	\quad \textfor
	p_0\in (1,\tfrac{n}{2s}), \tfrac{1}{p_1} = \tfrac{1}{p_0} - \tfrac {2s} n ,\\
	L^1(\Omega;\delta^\gamma)
	\cap L^{q_0}_\loc(\Omega)
	&\longrightarrow
	L^{\infty}_\loc(\Omega),
	\quad \textfor q_0\in(\tfrac{n}{2s},+\infty),
	\end{align*}
	For $K_1\Subset K_0\Subset \Omega$ we have estimates
	\begin{align}
	\norm[L^{p}(K_1)]{\bG_0(f)}
	&\leq C(p,K_1)
	\norm[L^1(\Omega)]{f\delta^\gamma},
\label{eq:3.5}
	\\
	\norm[L^{p_1}(K_1)]{\bG_0(f)}
	&\leq C(p_0,K_0,K_1)
	\left(
	\norm[L^{p_0}(K_0)]{f}
	+\norm[L^1(\Omega)]{f\delta^\gamma}
	\right),
	\label{eq:G0-Lp-loc}\\
	\norm[L^\infty(K_1)]{\bG_0(f)}
	&\leq
	C(q_0,K_0,K_1)
	\left(
	\norm[L^{q_0}(K_0)]{f}
	+\norm[L^1(\Omega)]{f\delta^\gamma}
	\right).
	\label{eq:G0-Lq-loc}
	\end{align}
\end{prop}

\begin{proof}
			We write
	\begin{equation*}
	f = f \chi_{K_0} + f \chi_{\Omega \setminus K_0}.
	\end{equation*}
	 For  $x \in K_1$ we have that
	\begin{align*}
	|\bG_0 ( f \chi_{\Omega \setminus K_0}) (x)| &\le C \int_{\Omega \setminus K_0} \frac{|f(y)| \delta(y)^\gamma }{|x-y|^{n-2s+\gamma}} dy
	 \\&
	 \le C \textrm{dist} (K_1, \Omega \setminus K_0)^{2s-n-\gamma} \| f \delta^\gamma \|_{L^1}
	\end{align*}
	 where  $C$ depends only on the kernel estimate. Therefore,
	\begin{equation*}
		|u(x)| \le |\bG_0 (f \chi_{K_0})| + C \textrm{dist} (K_1, \Omega \setminus K_0)^{2s-n-\gamma} \| f \delta^\gamma \|_{L^1 (\Omega)}.
	\end{equation*}
	\begin{enumerate}
		\item If $f \in L^1 (\Omega; \delta^\gamma)$ then
		\begin{align*}
			\|f \chi_{K_0} \|_{L^1 (\Omega)} &= \|f \|_{L^1 (K_0)} \le \|\delta^{-\gamma }\|_{L^\infty (K_0)} \| f \delta^\gamma \|_{L^1 (\Omega)} \\
			&\le \mathrm{dist} (K_0, \partial \Omega)^{-\gamma} \| f \delta^\gamma \|_{L^1 (\Omega)} .
		\end{align*}
		Hence,
		\begin{equation*}
			\|\bG_0 (f \chi_{K_0}) \|_{L^p (\Omega)} \le C(p) \|f \chi_{K_0} \|_{L^1 (\Omega)}  \le C(p) \mathrm{dist} (K_0, \partial \Omega)^{-\gamma} \| f \delta^\gamma \|_{L^1 (\Omega)} .
		\end{equation*}
Note that in \eqref{eq:3.5} the constant can be chosen to depend only on $K_1$, for example by fixing $K_0=\set{x\in\Omega:\delta(x)>\dist(K_1,\Omega)/2}$.
		\item
		If $f \chi_{K_0} \in L^{p_0}(K_0)$ with $p_0\in (1,\frac{n}{2s})$
		we have, analogously
		\begin{equation*}
			\|\bG_0 ( f \chi_{K_0}) \|_{L^{p_1} (\Omega)} \le C(p_0) \|f \chi_{ K_0} \|_{L^{p_0} (\Omega)} = C(p_0) \|f \|_{L^{p_0} (K_0)}.
		\end{equation*}
		
		\item
		If $f \chi_{K_0} \in L^{q_0}(K_0)$ with $q_0>\frac{n}{2s}$ we have
		\begin{equation*}
		\|\bG_0 ( f \chi_{K_0}) \|_{L^{\infty} (\Omega)} \le C(q_0) \|f \|_{L^{q_0} (K_0)}.
		\qedhere
		\end{equation*}

	\end{enumerate}
\end{proof}

\subsection{Eigenfunction estimates}

In \cite[Proposition 5.3 and Proposition 5.4]{BFV} the authors prove that the eigenfunctions satisfy $|\varphi_j| \le \kappa_j \delta^\gamma$ and that $\varphi_1 \asymp \delta^\gamma$. For the sake of completeness, we give a short proof of this fact.
\begin{prop}
	\label{prop:eigenfunctions regularity}
	Let $\varphi \in L^1 (\Omega; \delta^\gamma)$ be an eigenfunction of $\bG_0$, i.e. $\varphi = \lambda \bG_0 (\varphi)$ for some $\lambda$. Then $\varphi \in \delta^\gamma L^\infty (\Omega)$.
\end{prop}

\begin{proof}
	We apply a bootstrap argument. Clearly, $$\varphi = \lambda \bG_0 (\varphi) = \lambda \bG_0 (\lambda \bG_0 (\varphi) ) = \lambda^2 \bG_0^2 (\varphi) = \cdots =  \lambda^k \bG_0^k (\varphi).$$
	Since $\varphi \in L^1 (\Omega; \delta^\gamma)$, by \Cref{prop:regularisation between L^1 weight spaces} for $k$ large enough $\bG_0^k (\varphi) \in L^1 (\Omega)$. By \Cref{rem:3.2} %
we have that, for $k$ large enough $\bG_0^{k} (\varphi) \in L^p (\Omega)$ for all $p >\frac{n}{2s}$. Therefore, by \Cref{prop:G_0 is L^p to L^infty},  $\bG_0^{k+1} (\varphi) \in L^\infty (\Omega)$. Finally, by \Cref{prop:regularisation between L^infty weight spaces}, for $k$ large enough $\varphi = \lambda^k \bG_0^{k} (\varphi) \in \delta^\gamma L^\infty (\Omega)$.
\end{proof}

\subsection{Martin kernel}
\label{sec:martin kernel}
We recall the following  result.
\begin{thm*}[{\cite[Corollary 4.3, Theorem 4.6, Theorem 4.13]{AGV}}]
	Let $b=1-2s+\gamma$. For any $h\in C(\p\Omega)$, there exists a unique $\bG_0(L_c^\infty(\Omega))$-weak solution $v_h\in \delta^{-b}L^\infty(\Omega)$ given by
	\begin{equation}\label{eq:vh-martin}
	v_h(x)
	:={M}(h)(x)
	=\int_{\p\Omega}
	D_\gamma\cG_0(z,x)
	h(z)
	\,d\cH^{n-1}(z),
	\quad \textfor x\in\Omega,
	\end{equation}
	which satisfies
	\begin{equation}\label{eq:vh-bdry}
	\lim_{\Omega\ni x\to z}
	\dfrac{v_h(x)}{v_1(x)}
	=h(z),	
	\quad \textfor z\in\p\Omega,
	\end{equation}
	uniformly in $z\in\p\Omega$, where $v_1:={M}(1)\asymp \delta^{-b}$. In particular,
	\begin{equation*}%
	\norm[L^\infty(\Omega)]{v_h\delta^b}
	\leq C\norm[L^\infty(\p\Omega)]{h},
	\end{equation*}
	and so for any $K\Subset\Omega$,
	\begin{equation}\label{eq:vh-loc}
	\norm[L^\infty(K)]{v_h}
	\leq C(K)
	\norm[L^\infty(\p\Omega)]{h}.
	\end{equation}
\end{thm*}
Furthermore, we can show the following
\begin{prop}[Martin's operator]
	\label{prop:martin}
 Let $b=1-2s+\gamma$. The Martin  operator
	\[\begin{split}
	{\bM}:L^\infty (\p\Omega)
	&\longrightarrow
	\delta^{-b}L^\infty(\Omega)\\
	h&\longmapsto
	{\bM}(h)
	\end{split}\]
	defined by
	\[
	\bM(h)(x)
	=\int_{\p\Omega}
	D_\gamma\cG_0(z,x)h(z)
	\,d\cH^{n-1}(z)
	\]
	is a continuous operator.
\end{prop}

\begin{proof}
	Let $h\in L^\infty(\p\Omega)$. Using the behavior of $D_\gamma\cG_0$ stated in \textrm{(K2)}, write
	\[\begin{split}
	\delta^{b}(x){\bM}(h)(x)
	&=\int_{\p\Omega}
	\delta^{1-2s+\gamma}(x)
	D_\gamma\cG_0(z,x)
	h(z)
	\,d\cH^{n-1} (z) \\
	&\leq
	C\norm[L^\infty(\p\Omega)]{h}
	\int_{\p\Omega}
	\dfrac{
		\delta^{1-2s+2\gamma}(x)
	}{
		|z-x|^{n-2s+2\gamma}
	}
	\,d\cH^{n-1}(z)\\
	\end{split}\]
	Using a system of coordinates centered at $z_x=\arg\min\set{|z-x|:z\in\p\Omega}$ and aligning $e_n$ with the inward normal at $z_x$, the last integral converges as $\delta(x)\searrow0$ to
	\[
	\int_{\R^{n-1}}
	\dfrac{1}{
		|z-e_n|^{n-2s+2\gamma}
	}
	\,dz
	\]
	after considering a dilation of factor $1/\delta(x)$, and hence is bounded independently of $\delta(x)$.
\end{proof}

\subsection{The range $2s > \gamma$}

For $\lambda=0$, we prove regularity estimates in various Lebesgue spaces. With the assumption $\gamma\in(0,2s)$, $|\cdot|^{-(n-2s+\gamma)}$ is integrable, so that either factor $\delta^\gamma(x)$ or $\delta^\gamma(y)$ becomes available. The results on the whole domain is classical, and we include the proof for the reader's convenience. We are particularly interested in the local regularity for $L^1(\Omega;\delta^\gamma)$ data, where we use a localized version of the classical Hardy--Littlewood--Sobolev inequality. Compare with \cite[Theorem 2.1, Theorem 2.10, Theorem 2.11]{AGV}, where the full range $\gamma\in(0,1]$ is covered.

\begin{thm}[Local mapping properties of $\bG_0$]
\label{thm:G0-map}
The  Green  operator $\bG_0$ for $\cL$ is continuous from
\[\begin{split}
L^1(\Omega; \delta^\gamma )
&\longrightarrow
    L^p(\Omega),
        \quad \textfor p\in[1,\tfrac{n}{n-2s+\gamma}),\\
L^{q}(\Omega)
&\longrightarrow
    \delta^\gamma L^\infty(\Omega)
    	\quad \textfor q\in(\tfrac{n}{2s-\gamma},+\infty).
\end{split}\]
Moreover,
\begin{align}
\norm[L^p(\Omega)]{\bG_0(f)}
&\leq C(p)
    \norm[L^1(\Omega)]{f\delta^\gamma},
\label{eq:G0-L1}\\
\norm[L^\infty(\Omega)]{\bG_0(f)/\delta^\gamma}
&\leq C(q)
    \norm[L^{q}(\Omega)]{f}.
\label{eq:G0-Lq-d}
\end{align}
\end{thm}

\begin{remark}
Under additional assumptions on the Green function,  highly-integrable data lead to H\"{o}lder continuous solutions. %
However, we will not need this for the rest of the paper. %
\end{remark}

\begin{proof}
\begin{enumerate}
\item
    For any $f\in L^1(\Omega;\delta^\gamma)$, and $x\in\Omega$, we have
\[
|\bG_0(f)(x)|
\leq
    C
    \int_{\Omega}
        \dfrac{1}{|x-y|^{n-2s}}
        \dfrac{\delta^\gamma(y)}{|x-y|^{\gamma}}
        |f(y)|
    \,dy.
\]
Since $|\cdot-y|^{-(n-2s+\gamma)}$ is uniformly integrable in $y\in \Omega$, we have
\[
\norm[L^1(\Omega)]{\bG_0(f)}
\leq C
    \norm[L^1(\Omega)]{f\delta^\gamma}.
\]
This proves the case $p=1$. When $p\in(1,\frac{n}{n-2s+\gamma})$, we apply Jensen's inequality with the probability measure
\[
d\mu(y)
=\dfrac{
    |f(y)|\delta^\gamma(y)\,dy
}{
    \norm[L^1(\Omega)]{f\delta^\gamma}
},
\]
to obtain, from
\[\begin{split}
|\bG_0(f)(x)|
\leq C
    \norm[L^1(\Omega)]{f\delta^\gamma}
    \int_{\Omega}
        \dfrac{1}{
            |x-y|^{n-2s+\gamma}
        }
    \,d\mu(y),
\end{split}\]
that
\[\begin{split}
\norm[L^p(\Omega)]{\bG_0(f)}^p
&\leq
    C\norm[L^1(\Omega)]{f\delta^\gamma}^p
    \int_{\Omega}
        \left(
            \int_{\Omega}
                \dfrac{1}{
                    |x-y|^{n-2s+\gamma}
                }
            \,d\mu(y)
        \right)^p
    \,dx\\
&\leq
    C\norm[L^1(\Omega)]{f\delta^\gamma}^p
    \int_{\Omega}
        \int_{\Omega}
            \dfrac{1}{
                |x-y|^{(n-2s+\gamma)p}
            }
        \,dx
    \,d\mu(y)\\
&\leq
    C(p)
    \norm[L^1(\Omega)]{f\delta^\gamma}^p.
\end{split}\]

\item
	For any $f\in L^{q}(\Omega)$ with $q>\frac{n}{2s-\gamma}$, by H\"{o}lder's inequality, we have for any $x\in\Omega$,
\[\begin{split}
\dfrac{
	|\bG_0(f)(x)|
}{
	\delta^\gamma(x)
}
&\leq
	\int_{\Omega}
		\dfrac{
			|f(y)|
		}{
			|x-y|^{n-2s+\gamma}
		}
	\,dy\\
&\leq
	\norm[L^{q}(\Omega)]{f}
	\left(\int_{\Omega}
		\dfrac{1}{
			|x-y|^{(n-2s+\gamma)q'}
		}
	\,dy\right)^{\frac{1}{q'}}\\
&\leq
	C(q)\norm[L^q(\Omega)]{f},
\end{split}\]
since the conjugate exponent $q'=\frac{q}{q-1}$ satisfies $(n-2s+\gamma)q'<n$.
\end{enumerate}
\end{proof}

Then we prove that for bounded data the $\bG_0$-Green solution has well-defined $\gamma$-normal derivative. Here we use again the fact that $\gamma<2s$. Compare with \cite[Theorem 3.15]{AGV}.

\begin{prop}[$\gamma$-normal derivative of Green's operator]
\label{prop:D-gamma-def}
 Assume (K1) and (K2). Let $n>2s>\gamma$. 
\[
D_\gamma\bG_0:
L^\infty(\Omega)
    \longrightarrow
L^\infty(\p\Omega)
\]
is a continuous operator. More precisely,
for any $f\in L^\infty(\Omega)$, $D_\gamma\bG_0(f):\p\Omega\to\R$ is well-defined and equal to
\[
D_\gamma\bG_0(f)(z)
=\int_{\Omega}
    D_\gamma\cG_0(z,y)
    f(y)
\,dy,
    \quad \forall z\in\p\Omega.
\]
\end{prop}

\begin{proof}
	
For any $x_m\in\Omega$ with $x_m\to z\in\p\Omega$,
\[
\dfrac{
    \bG_0(f)(x_m)
}{
    \delta^\gamma(x_m)
}
=\int_{\Omega}
    \dfrac{
        \cG_0(x_m,{y})
    }{
        \delta^\gamma(x_m)
    }f({y})
\,d{y}
=:\int_{\Omega}
    F_m({y})
\,d{y}.
\]
Since
\[
|F_m({y})|
\leq
    \dfrac{
        C\norm[L^\infty(\Omega)]{f}
    }{
        |{x_m-y}|^{n-2s+\gamma}
    }
\]
are uniformly bounded in $L^p$ for $p\in(1,\frac{n}{n-2s+\gamma})$ (because $2s > \gamma$), there exists an $L^p$-weak-convergent subsequence. By hypothesis (K2) we recover the pointwise limit of $F_m$,
\begin{equation*}
	\lim_{m\to+\infty}F_m({y}) 
= D_\gamma \bG_0 (z,y) f(y),
\end{equation*}
and it coincides with the weak limit. Since the constant function $1$ is in $L^{p'} (\Omega)$, we have that the integral of $F_m$ converges (up to our subsequence). Since every convergent subsequence has the same limit, the whole sequence converges.
\end{proof}

\section{General linear theory}
\label{sec:general-lin}

\subsection{Global estimates}
\label{sec:lin}
Let $\lambda\in\R\setminus\Sigma$, and $\bar\lambda$ and $d_\Sigma(\lambda)$ be as in \eqref{eq:bar-lambda} and \eqref{eq:d-lambda}. Let $b=1-2s+\gamma$. %
Consider the equation
\begin{equation}\label{eq:lin}\begin{cases}
{\cL}u-\lambda u=f
    & \textin \Omega\\
B u=0
    & \texton \p\Omega.
\end{cases}\end{equation}
The definition of $\bG_\lambda$ relies on the well-posedness of \eqref{eq:lin}, which we will show on $L^2(\Omega)$, $L^q(\Omega)$ for large $q$, and finally $L^1(\Omega;\delta^\gamma)$.

We start with an $L^2$-theory.  Recall that $\bH^{2}_{\cL}(\Omega)$ is explained before \Cref{def:spectral}. 

\begin{lem}[Solvability in $L^2$]
\label{lem:lin-L2}
 Let $\lambda\in\R\setminus\Sigma$. 
If $f\in L^2(\Omega)$, then there exists a unique spectral solution $u\in \bH^{2}_{\cL}(\Omega)$ of \eqref{eq:lin}, given by\footnote{ Note that the series defined below converges since $f\in L^2(\Omega)$ and $d_\Sigma( \lambda) = \inf_{j\geq 1} |\lambda_j-\lambda|>0$. }
\[
u(x)
=\sum_{j\geq1}
    \dfrac{
        \angles{f,\varphi_j}
    }{
        \lambda_j-\lambda
    }\varphi_j(x).
\]
As a consequence,
\[
\norm[L^2(\Omega)]{u}
\leq d_\Sigma(\lambda)^{-1}
	\norm[L^2(\Omega)]{f}.
\]
\end{lem}

\begin{proof}
The result follows directly from eigenfunction decomposition.
\end{proof}

Hence, for any $\lambda\in\R\setminus\Sigma$, one may define $\bG_\lambda:L^2(\Omega)\to \bH^{2}_{\cL}(\Omega)$ %
by
\begin{equation}\label{eq:G-lambda-f}
\bG_\lambda(f)
=\sum_{j\geq1}
    \dfrac{
        \angles{f,\varphi_j}
    }{
        \lambda_j-\lambda
    }\varphi_j,
    	\quad \forall f\in L^2(\Omega).
\end{equation}
This includes the particular case $\lambda=0$, where $\bG_0$ is alternatively given by the integral formula \eqref{eq:G0-def}.

\begin{lem}[Spectral solutions are Green solutions]
\label{lem:spec-Green}
Suppose $f\in L^2(\Omega)$ and $u\in \bH^{2}_{\cL}(\Omega)$. Let $\bG_\lambda$ and $\bG_0$ be as defined in \eqref{eq:G-lambda-f}. Then the following are equivalent.
\begin{enumerate}
\item
	$u=\bG_\lambda(f)$ a.e.
\item
	$(\lambda_j-\lambda)\angles{u,\varphi_j}=\angles{f,\varphi_j}$ for all $j\geq1$.
\item
	$u-\lambda \bG_0(u)=\bG_0(f)$ a.e.
\end{enumerate}
\end{lem}

\begin{proof}
(1) $\Leftrightarrow$ (2) is by projection. To see (2) $\Leftrightarrow$ (3), one may rearrange
\[
(\lambda_j-\lambda)\angles{u,\varphi_j}
=\angles{f,\varphi_j}
\]
to
\[
\angles{u,\varphi_j}
-\lambda
	\dfrac{
		\angles{u,\varphi_j}
	}{
		\lambda_j
	}
=\dfrac{
	\angles{f,\varphi_j}
}{
	\lambda_j
}. \qedhere
\]
\end{proof}

\begin{remark}
	Notice that
	\begin{align*}
	u&=\lambda\bG_0(\lambda\bG_0(u)+\bG_0(f))+\bG_0(f) = \lambda^2 \bG_0^2 (u) + \lambda \bG_0^2 (f) + \bG_0 (f) = \cdots = \\
	&=\lambda^k \bG_0^k (u) + \sum_{m=1}^k \lambda^{m-1} \bG_0^m (f).
	\end{align*}
 Since the last summation is always as regular as $\bG_0(f)$, we can choose $k$ large and bootstrap the regularity of $u$ 
as in the proof of \Cref{prop:eigenfunctions regularity}. This shows us that the regularisation properties of $\bG_\lambda$ are precisely those of $\bG_0$.
\end{remark}

\begin{prop}[Solvability for highly-integrable data]
\label{prop:lin-Lq}
Let $q\in(\frac{n}{2s},+\infty]$.
For any $f\in L^{q}(\Omega)$ there is a unique spectral solution $u\in L^\infty(\Omega)$ of \eqref{eq:lin}, with
\[
\norm[L^\infty(\Omega)]{u}
\leq
    C(q,\bar\lambda)
    \left(
        \norm[L^2(\Omega)]{u}
        +\norm[L^q(\Omega)]{f}
    \right),
\]
and
\[
\norm[L^\infty(\Omega)]{u}
\leq
    C(q,\bar\lambda,d_\Sigma(\lambda))
	\norm[L^q(\Omega)]{f}.
\]
If $f \in \delta^\alpha L^\infty (\Omega)$  for $\alpha>-1-\gamma$,  then $u \in \bG_0 (\delta^\alpha) L^\infty(\Omega)$ we have that
\begin{equation*}
	\norm[L^\infty(\Omega)]{ \frac u { \bG_0 (\delta^\alpha )}}
	\leq
	C(\alpha ,\bar\lambda) \left (
	 \norm[L^2(\Omega)]{u} + \norm[L^\infty(\Omega)]{  \frac f {\delta^\alpha }} \right) .
\end{equation*}
Furthermore, if $\gamma < 2s$, then for $q\in(\frac{n}{2s-\gamma},+\infty]$,
\[
\norm[L^\infty(\Omega)]
    {\dfrac{u}{\delta^\gamma}}
\leq C(q,\bar\lambda)
    \left(
        \norm[L^2(\Omega)]{u}
        +\norm[L^q(\Omega)]{f}
    \right).
\]
\end{prop}

\begin{proof}
By \Cref{lem:lin-L2} and \Cref{lem:spec-Green}, \eqref{eq:lin} has a unique spectral solution which satisfies
\[
u=\lambda\bG_0(u)+\bG_0(f).
\]
We can now bootstrap the regularity using the results from \Cref{sec:G0}.
We proceed as follows:
\begin{enumerate}
\item
	Let $p_0=2$ and define $p_k$ by
	\[
		\frac{1}{p_k}
		=\frac{1}{p_0}
			-k\frac{2s}{n}
	\]
	Let $k_0$ be the smallest integer  such  that $p_{k_0}>\frac{n}{2s}$. (We can avoid $p_{k_0}=\frac{n}{2s}$ by decreasing $p_0$.)
\item
	For $k=0,\dots,k_0-1$, since $\lambda u+f\in L^{p_k}(\Omega)$, we have
\[\begin{split}
\norm[L^{p_{k+1}}(\Omega)]{u}
&\leq C(\bar\lambda)
	\left(
		\norm[L^{p_0}(\Omega)]{u}
		+\norm[L^{p_0}(\Omega)]{f}
	\right)\\
&\leq C(\bar\lambda)
	\left(
		\norm[L^2(\Omega)]{u}
		+\norm[L^q(\Omega)]{f}
	\right).
\end{split}\]

\item
	Since $\lambda u+f\in L^{p_{k_0}}(\Omega)$ for $p_{k_0}>\frac{n}{2s}$, by \eqref{eq:G0-Lq} we have
\[\begin{split}
\norm[L^\infty(\Omega)]{u}
&\leq C(\bar\lambda)
	\left(
		\norm[L^{p_{k_0}}(\Omega)]{u}
		+\norm[L^{p_{k_0}}(\Omega)]{f}
	\right)\\
&\leq C(\bar\lambda)
	\left(
		\norm[L^2(\Omega)]{u}
		+\norm[L^q(\Omega)]{f}
	\right).
\end{split}\]

\item
If $f \in \delta^\alpha L^\infty (\Omega)$ then, applying the previous steps $u \in L^\infty(\Omega)$.  Applying \Cref{prop:regularisation between L^infty weight spaces} a finite number of times, the result follows.\footnote{ Note that here we need the assumption on $\alpha$.}

\item
If $\gamma < 2s$, since $\lambda u+f\in L^q(\Omega)$ for $q>\frac{n}{2s-\gamma}$, by \eqref{eq:G0-Lq-d} we have
\[\begin{split}
\norm[L^\infty(\Omega)]{u/\delta^\gamma}
&\leq C(q,\bar\lambda)
	\left(
		\norm[L^\infty(\Omega)]{u}
		+\norm[L^{q}(\Omega)]{f}
	\right)\\
&\leq C(q,\bar\lambda)
	\left(
		\norm[L^2(\Omega)]{u}
		+\norm[L^q(\Omega)]{f}
	\right).
\end{split}\]
\end{enumerate}
The last estimate without $\norm[L^2(\Omega)]{u}$ follows from \eqref{lem:lin-L2}.
\end{proof}

\begin{prop}[Solvability for low integrability data]
\label{prop:lin-L1}
Let $f \in L^1 (\Omega; \bG_0 (\delta^\alpha))$ for $\alpha \ge 0$.
Then, there exists a unique $\bG_\lambda ( \delta^\gamma L^\infty (\Omega))$-solution $u \in L^1 (\Omega; \delta^\alpha)$  of \eqref{eq:lin}.
\begin{equation}\label{eq:lin-L1-est}
\norm[L^{1}(\Omega)]{u \delta^\alpha }
\leq C\norm[L^{1}(\Omega)]{f \bG_0 (\delta^\alpha)},
\end{equation}
$C=C(\alpha,\bar\lambda,d_\Sigma(\lambda))$. Furthermore, if $f\in L^{1} (\Omega)$, then $u\in L^p(\Omega)$ for any $p\in [1,\frac {n}{n-2s})$ and
\begin{equation}
\norm[L^{p}(\Omega)]{u}
\leq C\norm[L^{1}(\Omega)]{f},
\end{equation}
for $C=C(p,\bar\lambda,d_\Sigma(\lambda))$.
\end{prop}

\begin{proof}
Consider the cut-off data $f_k=(f\wedge k)\vee(-k)\in L^\infty(\Omega)$, which tends to $f$ in $L^1(\Omega;\bG_0 (\delta^\alpha))$ by the Dominated Convergence Theorem, and hence in $L^1 (\Omega; \delta^\gamma)$ by \eqref{eq:comparison weight gamma}. By \Cref{prop:lin-Lq}, there is a unique spectral solution $u_k\in  L^\infty(\Omega)$, satisfying  (recall \Cref{lem:spec-Green} and that $\bG_\lambda$ is self-adjoint) 
\begin{equation}\label{eq:lin-L1-k}
\int_{\Omega}
    u_k\psi
\,dx
=\int_{\Omega}
    f_k\bG_\lambda(\psi)
\,dx,
	 \quad\forall \psi\in \delta^\gamma L^\infty(\Omega).
\end{equation}
If $\psi \in L^\infty(\Omega)$, we can approximate it in $\delta^\gamma L^\infty (\Omega)$ by
\[
\psi_m
=\psi\left(
        \delta^{-\gamma}\wedge m
    \right)\delta^\gamma
\in \delta^\gamma L^\infty(\Omega).
\]
 By \Cref{thm:lin}(1),  $\bG_0 (\psi_m) \to \bG_0 (\psi)$ in $ L^\infty (\Omega)$. Since $u_k, f_k \in L^\infty (\Omega)$, passing to the limit under the integral
\begin{equation}
\label{eq:estimate sequence}
\int_{\Omega}
u_k\psi
\,dx
=\int_{\Omega}
f_k\bG_\lambda(\psi)
\,dx,
\quad\forall \psi\in  L^\infty(\Omega).
\end{equation}

For any $k,\ell\geq0$,
\[
\int_{\Omega}
    (u_k-u_\ell)\psi
\,dx
=\int_{\Omega}
    (f_k-f_\ell)\bG_\lambda(\psi)\,dx,
    \qquad  \forall \psi\in  L^\infty(\Omega).
\]
Taking $\psi=\sign(u_k-u_\ell) \delta^\gamma \in \delta^\gamma L^\infty(\Omega)$, and then using the fact that $ |\bG_\lambda (\delta^\gamma) |\le C (\bar \lambda, d_\Sigma (\lambda))  \delta^\gamma$, we have that
\begin{equation*}%
\int_{\Omega}
|u_k-u_\ell| \delta^\gamma
\,dx
\leq C(\bar\lambda,d_\Sigma(\lambda))
\int_{\Omega}
|f_k-f_\ell| \delta^\gamma
\,dx .
\end{equation*}
Therefore,  $u_k$  is a Cauchy sequence in $L^1 (\Omega; \delta^\gamma)$. Passing to the limit  $k\to+\infty$ in \eqref{eq:estimate sequence}  under the integral sign, we conclude that $u$ is the $\bG_\lambda( \delta^\gamma L^\infty(\Omega))$-weak solution of \eqref{eq:lin}.

Taking $\psi=\sign(u_k-u_\ell) \delta^\alpha \in \delta^\alpha L^\infty(\Omega)$, and using that $|\bG_\lambda(\delta^\alpha)| \le  C(\alpha , \bar\lambda,d_\Sigma(\lambda)) \delta^\alpha$ we recover analogously that $u_k$ is a Cauchy sequence in $L^1(\Omega; \delta^\alpha)$ and has a limit $u$. To see \eqref{eq:lin-L1-est}, we put $\psi=\sign{(u_k)} \delta^\alpha \in \delta^\alpha L^\infty(\Omega)$ in \eqref{eq:estimate sequence}
and use \Cref{prop:lin-Lq} to obtain
\[
\int_{\Omega}
|u_k| \delta^\alpha
\,dx
\leq C(\alpha, \bar\lambda,d_\Sigma(\lambda))
\int_{\Omega}
|f_k|\bG_0 (\delta^\alpha)
\,dx.
\]
Passing to the limit we recover \eqref{eq:lin-L1-est}.

If $\tilde{u}$ is any other solution, then
\[
\int_{\Omega}
    (u-\tilde{u})\psi
\,dx
=0,
	\quad \forall \psi\in \delta^\gamma L^\infty(\Omega),
\]
and by taking $\psi=\sign(u-\tilde{u}) \delta^\gamma $ we must have $u=\tilde{u}$ a.e. in $\Omega$.

For $p \in (1,\frac{n}{n-2s})$, let $k=1,2,\dots$, and take instead $\psi=(|u|\wedge k)^{p-1}\sign(u)\in L^\infty(\Omega)$ which is moreover uniformly bounded in $L^{\frac{p}{p-1}}(\Omega)$. Since $p<\frac{n}{n-2s}$, we have $p' = \frac{p}{p-1} > \frac n {2s}$ and so \Cref{prop:lin-Lq} again applies to yield
\begin{align*}
\int_{\Omega}
	(|u|\wedge k)^{p}
\,dx
&\leq
\int_{\Omega}
	|u|(|u|\wedge k)^{p-1}
\,dx
= \|f \|_{L^1} \Big \|\bG_\lambda (  (|u|\wedge k)^{p-1}\sign(u)  ) \Big \|_{L^\infty}
 \\
& \leq
	C(p,\bar\lambda,d_\Sigma(\lambda))
	\|f \|_{L^1}
		\Big\|
			{(|u|\wedge k)^{p-1}\sign(u)}
			\Big\|_{L^{ \frac{p}{p-1} } (\Omega) } \\
&\leq C(p,\bar\lambda,d_\Sigma(\lambda))
	\left(
		\int_{\Omega}
			(|u|\wedge k)^p
		\,dx
	\right)^{\frac{p-1}{p}}
		\|f \|_{L^1(\Omega)}.
\end{align*}
Hence,
\[
\norm[L^p(\Omega)]{|u|\wedge k}
\leq C	\|f \|_{L^1(\Omega)},
\]
and the \eqref{eq:lin-L1-est} follows by taking $k\to+\infty$.
\end{proof}
\begin{proof}[Proof of \Cref{thm:lin} (1)]
It follows from \Cref{lem:lin-L2}, \Cref{prop:lin-Lq} and \Cref{prop:lin-L1}.
\end{proof}
\subsection{Integration-by-parts formulae}
\label{sec:by-parts}

We prove an integration by parts formula for $\bG_0$ and then for $\bG_\lambda$. This confirms the duality seen in their mapping properties.
\begin{lem}[Integration by parts with $\bG_0$]
\label{lem:by-parts-G0}
 Assume (K1).  Suppose $f\in L^1(\Omega;\delta^\gamma)$ and $g\in \delta^\gamma L^\infty(\Omega)$. Then we have
\[
\int_{\Omega}
	f\bG_0(g)
\,dx
=\int_{\Omega}
	g\bG_0(f)
\,dx.
\]
\end{lem}

\begin{proof}
Both sides are equal to
\[
\int_{\Omega}\int_{\Omega}
	\cG_0(x,y)f(x)g(y)
\,dx\,dy.
\]
 The Fubini--Tonelli argument is justified by \Cref{thm:lin}(1). 
\end{proof}
The linear theory  in  \Cref{thm:lin}  and \eqref{eq:improvement of weights} show  in particular that $\bG_\lambda$ is also a bounded operator from
\[\begin{split}
\bG_\lambda: L^1(\Omega;\delta^\gamma)
&\longrightarrow
    L^1(\Omega; \delta^{\gamma })\\
\bG_\lambda: \delta^{\gamma} L^\infty(\Omega)
&\longrightarrow
    \delta^\gamma L^\infty(\Omega).
\end{split}\]
Therefore, an analogous integration by parts formula holds for $\bG_\lambda$ because of that for the differential operator $\cL-\lambda$, even without knowledge of  its  kernel.

\begin{lem}[Integration by parts with $\cL$]
\label{lem:by-parts-L}
Suppose $u,v\in \bH^{2}_{\cL}(\Omega)$. %
Then
\[
\int_{\Omega}
    u{\cL}v
\,dx
=\int_{\Omega}
    v{\cL}u
\,dx.
\]
\end{lem}

\begin{proof}
In the eigenbasis,
\[
u(x)
=\sum_{j\geq1}
    \angles{u,\varphi_j}
    \varphi_j(x),
        \quad
v(x)
=\sum_{j\geq1}
    \angles{v,\varphi_j}
    \varphi_j(x),
\]
both integrals are equal to
\[
\sum_{j\geq1}
    \lambda_j
    \angles{u,\varphi_j}
    \angles{v,\varphi_j}.\qedhere
\] \qedhere
\end{proof}

\begin{lem}[Integration by parts with $\bG_\lambda$]
\label{lem:by-parts-G-lambda}
Suppose $f\in L^1(\Omega;\delta^\gamma)$ and $g\in \delta^\gamma L^\infty(\Omega)$. Then we have
\[
\int_{\Omega}
    f\bG_\lambda(g)
\,dx
=\int_{\Omega}
    g\bG_\lambda(f)
\,dx.
\]
\end{lem}

\begin{proof}
Suppose first $f,g\in L^2(\Omega)$. Write $u=\bG_\lambda(f)$ and $v=\bG_\lambda(g)$, so that $u,v\in \bH^{2}_{\cL}(\Omega)$  by \Cref{lem:lin-L2}. Then\footnote{Alternatively, the same computations can be done using the spectral decomposition using \eqref{eq:G-lambda-f}, as in \cite{BSV}.}
\[\begin{split}
\int_{\Omega}
    f\bG_\lambda(g)
\,dx
=\int_{\Omega}
    (\cL{u}-\lambda u)v
\,dx
=\int_{\Omega}
    (\cL{v}-\lambda v)u
\,dx
=\int_{\Omega}
    g\bG_\lambda(f)
\,dx.
\end{split}\]

In general, let $f_k,g_k\in C_c^\infty(\Omega)\subset L^2(\Omega)$ such that $f_k\to f$ in $L^1(\Omega;\delta^\gamma)$ and $g_k\to g$ in $\delta^\gamma L^\infty(\Omega)$. Since $g_k$ are uniformly bounded in $L^\infty(\Omega)$, $\bG_\lambda(g_k)$ is uniformly bounded in $\delta^\gamma L^\infty(\Omega)$. On the other hand, $f_k$ are uniformly bounded in $L^1(\Omega;\delta^\gamma)$, and so $\bG_\lambda(f_k)$ are uniformly bounded in $L^1(\Omega; \delta^\gamma )$  thanks to the observation made after \Cref{lem:by-parts-G0}.  By \Cref{thm:lin},
\[
\norm[L^1(\Omega)]
    {(\bG_\lambda(f_k)-\bG_\lambda(f))\delta^\gamma }
=\norm[L^1(\Omega)]
    {\bG_\lambda(f_k-f) \delta^\gamma }
\leq C\norm[L^1(\Omega)]
    {(f_k-f)\delta^\gamma},
\]
\[
\norm[L^\infty(\Omega)]
    {\frac {\bG_\lambda(g_k)-\bG_\lambda(g)} {\delta^\gamma }}
=\norm[L^\infty(\Omega)]
    {\frac{\bG_\lambda(g_k-g)}{\delta^\gamma}}
\leq C\norm[L^\infty(\Omega)]{\frac{g_k-g } {\delta^\gamma }}.
\]
This implies
\[\begin{split}
    \int_{\Omega}
        f_k\bG_{\lambda}(g_k)
    \,dx
	&= \int_{\Omega}
	f_k \delta^\gamma \frac{ \bG_{\lambda}(g_k) }{\delta^\gamma}
	\,dx
\longrightarrow \int_{\Omega}
f \delta^\gamma \frac{ \bG_{\lambda}(g) }{\delta^\gamma}
\,dx
=\int_{\Omega}
f\bG_\lambda(g)
\,dx
\end{split}\]
and, analogously,
\[\begin{split}
        \int_{\Omega}
            g_k\bG_\lambda(f_k)
        \,dx
       = \int_{\Omega}
       \frac{ g_k } {\delta^\gamma} \bG_\lambda(f_k) \delta^\gamma
       \,dx
       \longrightarrow
       \int_{\Omega}
       \frac{ g } {\delta^\gamma} \bG_\lambda(f) \delta^\gamma
       \,dx
       =  \int_{\Omega}
       g \bG_\lambda(f)
       \,dx
\end{split}\]
as $k\to+\infty$. This proves the result.
\end{proof}

\subsection{Equivalent notions of solution}
\label{sec:notions}

Suppose $f\in L^1(\Omega;\delta^\gamma)$ and $u\in L^1(\Omega;\delta^\gamma)$. We show that the notions of $L^1$-solution are equivalent, for the equation
\begin{equation}\label{eq:lin-weak-dual}
\begin{cases}
\cL{u}-\lambda u=f
    & \textin \Omega\\
B u=0
    & \texton \p\Omega.
\end{cases}
\end{equation}
For convenience we restate the statements in \Cref{thm:notions}.
\begin{equation*}
u=\bG_\lambda(f)
\quad \textin L^1(\Omega; \delta^\gamma ).
    \tag{1}\label{eq:sense-1}
\end{equation*}
\begin{equation*}
\int_{\Omega}u\psi\,dx
=\int_{\Omega}f\bG_\lambda(\psi)\,dx,
\quad \forall \psi\in \delta^\gamma L^\infty(\Omega).
    \tag{2}\label{eq:sense-2}
\end{equation*}
\begin{equation*}
\int_{\Omega}u(\phi-\lambda\bG_0(\phi))\,dx
=\int_{\Omega}f\bG_0(\phi)\,dx,
\quad \forall \phi\in L_c^\infty(\Omega).
    \tag{3}\label{eq:sense-3}
\end{equation*}
\begin{equation*}
\int_{\Omega}u(\varphi-\lambda\bG_0(\varphi))\,dx
=\int_{\Omega}f\bG_0(\varphi)\,dx,
\quad \forall \varphi\in C_c^\infty(\Omega).
    \tag{4}\label{eq:sense-4}
\end{equation*}
\begin{equation*}
u-\lambda\bG_0(u)=\bG_0(f)
\quad \textin L^1(\Omega;\delta^\gamma).
    \tag{5}\label{eq:sense-5}
\end{equation*}
\begin{equation*}
(\lambda_j-\lambda)\angles{u,\varphi_j}
=\angles{f,\varphi_j}
\quad \forall j\geq1.
    \tag{6}\label{eq:sense-6}
\end{equation*}
We remark that when $f\in L^2(\Omega)$ and $u\in \bH^{2}_{\cL}(\Omega)$, (1) $\Leftrightarrow$ (5) $\Leftrightarrow$ (6) is already proved in \Cref{lem:spec-Green}.

\begin{proof}[Proof of \Cref{thm:notions}]
We prove the equivalences pair by pair.

\eqref{eq:sense-1} $\Leftrightarrow$ \eqref{eq:sense-2}: Given \eqref{eq:sense-1}, we multiply both sides by $\psi\in \delta^\gamma L^\infty(\Omega)$ and integrate by parts using \eqref{lem:by-parts-G-lambda} to get \eqref{eq:sense-2}. Given \eqref{eq:sense-2}, integrate by parts using \eqref{lem:by-parts-G-lambda} to yield
\[
\int_{\Omega}
    (u-\bG_\lambda(f))\psi
\,dx=0,
    \quad \forall \psi\in \delta^\gamma L^\infty(\Omega).
\]
and take $\psi=\sign(u-\bG_\lambda(f))\delta^\gamma$ to get \eqref{eq:sense-1}.

\eqref{eq:sense-2} $\Leftrightarrow$ \eqref{eq:sense-3}: Given \eqref{eq:sense-2} and $\phi\in L_c^\infty(\Omega)$, choose $\zeta=\bG_0(\phi)$ and take $\psi=\phi-\lambda\zeta\in \delta^\gamma L^\infty(\Omega)$ so that by uniqueness $\zeta=\bG_\lambda(\psi)$. Then $\psi=\phi-\lambda\bG_0(\phi)$ and $\bG_\lambda(\psi)=\bG_0(\phi)$, so \eqref{eq:sense-2} implies \eqref{eq:sense-3}. Given \eqref{eq:sense-3} and $\psi\in \delta^\gamma L^\infty(\Omega)$, write $\zeta=\bG_\lambda(\psi)$ and choose $\phi=\psi+\lambda\zeta\in \delta^\gamma L^\infty(\Omega)$, so that $\phi-\lambda\bG_0(\phi)=\psi$ and $\bG_0(\phi)=\bG_\lambda(\psi)$. While $\phi$ is not a valid test function for \eqref{eq:sense-3}, we choose $\phi_K=\phi\mathbf{1}_K\in L_c^\infty(\Omega)$ for any $K\Subset\Omega$. Integrating \eqref{eq:sense-3} with $\phi_K$ by parts using \Cref{lem:by-parts-G0} yields
\[
\int_{\Omega}
    (u-\lambda\bG_0(u))\phi_K
\,dx
=\int_{\Omega}
    \bG_0(f)\phi_K
\,dx,
\]
which, as a $L^1 (\Omega; \delta^\gamma)$--$\delta^\gamma L^\infty (\Omega)$ pairing, can be passed to the limit $K\nearrow\Omega$ using  the  Dominated Convergence Theorem. The resulting identity becomes \eqref{eq:sense-2}.

\eqref{eq:sense-3} $\Leftrightarrow$ \eqref{eq:sense-4}: Since $C_c^\infty(\Omega)\subset L_c^\infty(\Omega)$, ``$\Rightarrow$'' is trivial. Assume now \eqref{eq:sense-4}. Given $\phi\in L_c^\infty(\Omega)$, take $\varphi_\eps=\phi\ast\eta_\eps\in C_c^\infty(\Omega)$, where $\eps>0$ and $\eta_\eps$ is the standard mollifier. Integrating \eqref{eq:sense-4} with $\varphi_\eps$ by parts using \Cref{lem:by-parts-G0}, we have
\[
\int_{\Omega}
    (u-\lambda\bG_0(u))\varphi_\eps
\,dx
=\int_{\Omega}
    \bG_0(f)\varphi_\eps
\,dx,
\]
which, as a $L^1 (\Omega; \delta^\gamma)$--$\delta^\gamma L^\infty (\Omega)$ pairing, can be passed to the limit $\eps\searrow0$ using  the  Dominated Convergence Theorem. The resulting identity becomes \eqref{eq:sense-3}.

\eqref{eq:sense-3} $\Leftrightarrow$ \eqref{eq:sense-5}:\footnote{ Note that if we already know that $f\in L^2(\Omega)$, then this part is not needed because \Cref{lem:spec-Green} implies (1) $\Leftrightarrow$ (2).} Integrate \eqref{eq:sense-3} by parts, we have
\[
\int_{\Omega}
    (u-\lambda\bG_0(u)-\bG_0(f))\phi
\,dx
=0
    \quad \forall \phi\in L_c^\infty(\Omega).
\]
For any $K\Subset\Omega$, choosing $\phi=\sign(u-\lambda\bG_0(u)-\bG_0(f))\mathbf{1}_K$ yields \eqref{eq:sense-5}. The converse is trivial, indeed by just testing \eqref{eq:sense-5} against $\phi\in L_c^\infty(\Omega)$ one obtains \eqref{eq:sense-3}.
\end{proof}

\subsection{Local boundedness}
\label{sec:lin-loc}

In this section we obtain a local boundedness result for $\bG_0$-Green solutions, which are also $\bG_\lambda(L^\infty(\Omega))$-weak solutions according to \Cref{thm:notions}, of the equation
\[\begin{cases}
{\cL}u-\lambda u=f
    & \textin \Omega\\
B u=0
    & \texton \p\Omega.
\end{cases}\]
Thus we study the equivalent integral equation
\begin{equation}\label{eq:lin-int}
u=\lambda \bG_0(u)+\bG_0(f)
	\quad \textin L^1(\Omega;\delta^\gamma).
\end{equation}
The main estimates come from \Cref{thm:G0-map}, which are in turn based on a localized version of the classical Hardy--Littlewood--Sobolev inequality.

\begin{prop}
\label{prop:lin-loc}
The operator $f\mapsto u=\bG_\lambda (f)$ is continuous
	\begin{align*}
L^1(\Omega;\delta^\gamma)
&\longrightarrow
L^{p}_\loc(\Omega),
\quad \textfor
p \in [1,\tfrac{n}{n-2s}) ,\\
L^1(\Omega;\delta^\gamma)
\cap L^{p_0}_\loc(\Omega)
&\longrightarrow
L^{p_1}_\loc(\Omega),
\quad \textfor
p_0\in (1,\tfrac{n}{2s}), \tfrac{1}{p_1} = \tfrac{1}{p_0} - \tfrac {2s} n ,\\
L^1(\Omega;\delta^\gamma)
\cap L^{q_0}_\loc(\Omega)
&\longrightarrow
L^{\infty}_\loc(\Omega),
\quad \textfor q_0\in(\tfrac{n}{2s},+\infty),
\end{align*}
Moreover, for $K\Subset K_0\Subset \Omega$,
\begin{align*}
\norm[L^p(K)]{u} &\leq C(\bar \lambda, p, K)\left(
\norm[L^1(\Omega)]{ u \delta^\gamma} + \norm[L^1(\Omega)]{ f \delta^\gamma}\right), \\
\norm[L^{p_1}(K)]{u}
&\leq C(\bar\lambda,p_0,K_0,K)
\left(
\norm[L^1(\Omega)]{u \delta^\gamma }
+\norm[L^{p_0}(K_0)]{f}
+\norm[L^1(\Omega)]{f\delta^\gamma}
\right),\\
\norm[L^\infty(K)]{u}
&\leq C(\bar\lambda,q_0,K_0,K)
    \left(
		\norm[L^1(\Omega)]{u \delta^\gamma }
		+\norm[L^{q_0}(K_0)]{f}
		+\norm[L^1(\Omega)]{f\delta^\gamma}
	\right).
\end{align*}
\end{prop}

\begin{proof}
We use a standard bootstrap argument using \Cref{prop:local boundedness}.
\begin{enumerate}
\item
	From $\lambda u+f\in L^1(\Omega;\delta^\gamma)$, we see from \eqref{eq:lin-int} that
\begin{align*}
\norm[L^{p}(K)]{u}
&\leq C(p, K)
		\norm[L^1(\Omega)]{(\lambda u + f)\delta^\gamma} \\
&\leq C(p, K)
\norm[L^1(\Omega)]{(\lambda u + f)\delta^\gamma} \\
&\leq C(p, K) (1 + |\lambda|) \left(
\norm[L^1(\Omega)]{ u \delta^\gamma} + \norm[L^1(\Omega)]{ f \delta^\gamma}\right),
\end{align*}
for some fixed $p\in(1,\frac{n}{n-2s})$.
\item
	Fix any $p_0\in(1,\frac{n}{n-2s}\wedge \frac{n}{2s})$. Define $p_k$ by
\[
\frac{1}{p_k}=\frac{1}{p_0}-k\frac{2s}{n}.
\]
Let $k_0$ be smallest integer such that $p_{k_0}\in(\frac{2s}{n},+\infty)$. (One may avoid the critical situation $p_{k_0}=\frac{2s}{n}$ by increasing $p_0$ if necessary.) For any $K\Subset K_0\Subset \Omega$, we choose a sequence of compact subsets
\[
K\Subset K_{k_0} \Subset K_{k_0-1} \Subset \cdots \Subset K_1 \Subset K_0 \Subset \Omega.
\]
\item
	For $k=0,1,\dots,k_0-1$, from $u\in L^{1}(\Omega;\delta^\gamma) \cap L^{p_k}(K_k)$ and $f\in L^1(\Omega;\delta^\gamma)\cap L^\infty(K_0)$, we deduce from \eqref{eq:lin-int} and \eqref{eq:G0-Lp-loc} that
\[\begin{split}
\norm[L^{p_{k+1}}(K_{k+1})]{u}
&\leq C(\bar\lambda,K,K_0)
    \left(
		\norm[L^{p_k}(K_k)]{u}
		+\norm[L^1(\Omega)]{u\delta^\gamma}
		+\norm[L^\infty(K_0)]{f}
		+\norm[L^1(\Omega)]{f\delta^\gamma}
	\right)\\
&\leq C(\bar\lambda,K,K_0)
    \left(
		\norm[L^1(\Omega)]{u\delta^\gamma}
		+\norm[L^\infty(K_0)]{f}
		+\norm[L^1(\Omega)]{f\delta^\gamma}
	\right).
\end{split}\]
\item
	Since $u\in L^1(\Omega;\delta^\gamma) \cap L^{p_{k_0}}(K_{k_0})$ with $p_{k_0}>\frac{2s}{n}$, and $f\in L^1(\Omega;\delta^\gamma)\cap L^\infty(K_0)$, \eqref{eq:lin-int} and \eqref{eq:G0-Lq-loc} implies
\[\begin{split}
\norm[L^\infty(K)]{u}
&\leq C(\bar\lambda,K,K_0)
    \left(
		\norm[L^{p_{k_0}}(K_{k_0})]{u}
		+\norm[L^1(\Omega)]{u\delta^\gamma}
		+\norm[L^\infty(K_0)]{f}
		+\norm[L^1(\Omega)]{f\delta^\gamma}
	\right)\\
&\leq C(\bar\lambda,K,K_0)
    \left(
		\norm[L^1(\Omega)]{u\delta^\gamma}
		+\norm[L^\infty(K_0)]{f}
		+\norm[L^1(\Omega)]{f\delta^\gamma}
	\right). \qedhere
\end{split}\]
\end{enumerate}
\end{proof}

\begin{proof}[Proof of \Cref{thm:lin} (2)]
It follows from \Cref{prop:lin-loc}, and \Cref{thm:lin} (1) which bounds $\norm[L^1(\Omega)]{u\delta^\gamma}$ in terms of $f$ and $d_\Sigma(\lambda)$.
\end{proof}

\section{Projected linear theory}
\label{sec:plin}

Given $\lambda\in\R\setminus\Sigma$, we let $\bar\lambda$, $d_\Sigma(\lambda)$ be as in \eqref{eq:bar-lambda}, \eqref{eq:d-lambda}. We decompose $L^2(\Omega)=E\oplus E^\perp$, where $E$ is the span of eigenfunctions associated to eigenvalues from $\lambda_1$ up to $\bar\lambda$, and $E^\perp$ is its orthogonal complement. For the orthogonal component (i.e. in $E^\perp$) of the solution, we derive estimates independent of $d_\Sigma(\lambda)$.

For any datum $f\in L^1(\Omega;\delta^\gamma)$, we project the equation
\begin{equation}\label{eq:lin-unproj}
\begin{cases}
{\cL}u-\lambda u=f
    & \textin \Omega\\
B u=0
    & \texton \p\Omega,
\end{cases}\end{equation}
into $\widetilde{E^\perp}$, according to
\begin{equation}\label{eq:proj-fu}\begin{split}
f^\perp(x)
&=f(x)
	-\sum_{\varphi_j\in E_i}
		\angles{f,\varphi_j}
		\varphi_j(x)\\
u^\perp(x)
&=u(x)
	-\sum_{\varphi_j\in E_i}
		\dfrac{
			\angles{f,\varphi_j}
		}{
			\lambda_j-\lambda
		}
		\varphi_j(x),
\end{split}\end{equation}
to arrive at
\begin{equation}\label{eq:plin}\begin{cases}
{\cL}u^\perp-\lambda u^\perp=f^\perp
    & \textin \Omega\\
B u^\perp=0
    & \texton \p\Omega.
\end{cases}\end{equation}

\begin{lem}[Projection]
\label{lem:proj}
Let $p\in[1,\frac{n-2s}{n})$ and $q\in(\frac{n}{2s},+\infty)$. Under the projection \eqref{eq:proj-fu},
\begin{enumerate}
\item
	$u\in \bH^{2}_{\cL}(\Omega)$ is a spectral solution of \eqref{eq:lin-unproj} for $f\in L^2(\Omega)$ if and only if $u^\perp\in \bH^{2}_{\cL}(\Omega)\cap E^\perp$ is a spectral solution of \eqref{eq:plin} for $f^\perp\in L^2(\Omega)\cap E^\perp$.
\item
	$u\in L^\infty(\Omega)$ is a spectral solution of \eqref{eq:lin-unproj} for $f\in L^q(\Omega)$ if and only if $u^\perp\in  L^\infty(\Omega)\cap E^\perp$ is a spectral solution of \eqref{eq:plin} for $f^\perp\in L^q(\Omega)\cap E^\perp$.
\item
	$u\in L^1(\Omega;\delta^\gamma)$ is a $\bG_\lambda(\delta^\gamma L^\infty(\Omega))$-weak solution of \eqref{eq:lin-unproj} for $f\in L^1(\Omega;\delta^\gamma)$ if and only if $u^\perp\in %
\widetilde{  E^\perp }$ is a $\bG_\lambda(\delta^\gamma L^\infty(\Omega))$-weak solution of \eqref{eq:plin} for $f^\perp\in %
\widetilde{  E^\perp }$.
\end{enumerate}
\end{lem}

\begin{proof}
Since there are finitely many smooth eigenfunctions involved in \eqref{eq:proj-fu}, the projected and unprojected functions lie in the same space, and their equations are directly verified.
\end{proof}

\begin{cor}[Well-posedness]
Let $p\in[1,\frac{n-2s}{n})$ and $q\in(\frac{n}{2s},+\infty)$. Under the projection \eqref{eq:proj-fu}:
\begin{enumerate}
\item
	For any $f^\perp\in L^2(\Omega)\cap E^\perp$, there exists a unique spectral solution $u^\perp\in \bH^{2}_{\cL}(\Omega)\cap E^\perp$ of \eqref{eq:plin}.
\item
	For any $f^\perp\in L^q(\Omega)\cap E^\perp$, there exists a unique spectral solution $u^\perp\in L^\infty(\Omega) \cap E^\perp$ of \eqref{eq:plin}.
\item
	For any $f^\perp\in \widetilde{  E^\perp }$, there exists a unique $\bG_\lambda(\delta^\gamma L^\infty(\Omega))$-weak solution $u^\perp\in \widetilde{  E^\perp }$ of \eqref{eq:plin}.
\end{enumerate}
\end{cor}

\begin{proof}
The result follows directly from \Cref{thm:lin} and \Cref{lem:proj}.
\end{proof}

The underlying reason for the uniform estimate is seen in the $L^2$-theory.

\begin{lem}[Uniform estimate for $L^2$ data]
\label{lem:plin-L2}
Suppose $f^\perp\in L^2(\Omega)\cap E^\perp$ and $u^\perp\in \bH^{2}_{\cL}(\Omega)\cap E^\perp$ is a spectral solution of \eqref{eq:plin}. Then
\[
u^\perp (x)
=\sum_{j = I+1}^{+\infty}
	\dfrac{
		\angles{f,\varphi_j}
	}{
		\lambda_j-\lambda
	}\varphi_j(x).
\]
In particular,
\[
\norm[L^2(\Omega)]
	{u^\perp}
\leq C(\bar\lambda)
	\norm[L^2(\Omega)]
		{f^\perp}.
\]
\end{lem}

\begin{proof}
This is an immediate consequence of the eigenfunction expansion.
\end{proof}

Then we obtain boundary regularity with uniform estimates.

\begin{prop}[Uniform estimate for highly-integrable data]
\label{prop:plin-Lq}
Let $f^\perp\in L^q(\Omega)\cap E^\perp$ for $q\in(\frac{n}{2s},+\infty)$, and $u^\perp\in L^\infty(\Omega)\cap E^\perp$ be a spectral solution of \eqref{eq:plin}. Then
\[
\norm[L^\infty(\Omega)]
    {u^\perp}
\leq C(q,\bar\lambda)
	\norm[L^q(\Omega)]
		{f^\perp}.
\]
Moreover, if $f^\perp\in \delta^\alpha L^\infty(\Omega)\cap E^\perp$  for $\alpha>-1-\gamma$ , then $u\in \bG_0(\delta^\alpha)L^\infty(\Omega)\cap E^\perp$ satisfies
\begin{equation*}
\norm[L^\infty(\Omega)]
{ \frac {u^\perp} {\bG_0 (\delta^\alpha)}  }
\leq C(\alpha, \bar\lambda)
\norm[L^\infty(\Omega)]
{ \frac{ f^\perp } {\delta^\alpha} }.
\end{equation*}
\end{prop}

\begin{proof}
By \Cref{prop:lin-Lq} applied to \eqref{eq:plin},
\[
\norm[L^\infty(\Omega)]{u^\perp}
\leq
    C(q,\bar\lambda)
    \left(
        \norm[L^2(\Omega)]{u^\perp}
        +\norm[L^q(\Omega)]{f^\perp}
    \right).
\]
Then we use \Cref{lem:plin-L2} to bound
\[
\norm[L^2(\Omega)]{u^\perp}
\leq C(\bar\lambda)
	\norm[L^2(\Omega)]{f^\perp}
\leq C(\bar\lambda)
	\norm[L^\infty(\Omega)]{f^\perp}.
\]
Analogously with weights.
\end{proof}

The next result on the uniform estimate in the projected space is crucial and we call it a theorem due to its importance.
\begin{thm}[Uniform estimate for low integrability data]
\label{thm:plin-L1}
 Let $n>2s$. Let $\alpha>-1-\gamma$. 
Let $f^\perp\in L^1(\Omega; \bG_0 (\delta^\alpha) )\cap \widetilde{ E^{\perp} } $, and $u^\perp\in L^1(\Omega; \delta^\alpha)\cap \widetilde{ E^{\perp} } $ be the $\bG_\lambda( \delta^\gamma L^\infty(\Omega))$-weak solution of \eqref{eq:plin}. Then
\[
\norm[L^1(\Omega)]{u^\perp \delta^\alpha }
\leq C\norm[L^1(\Omega)]
	{f^\perp \bG_0 (\delta^\alpha) },
\]
\[
\norm[L^1(\Omega)]{u^\perp}
\leq C\norm[L^1(\Omega)]
{f \delta^\gamma }, \quad \text{ and }
\quad
\norm[L^p(\Omega)]{u^\perp}
\leq C\norm[L^1(\Omega)]
    {f}, \ \forall p\in[1, \tfrac{n}{n-2s}).
\]
Here $C=C(\bar\lambda)$.
\end{thm}

\begin{remark}
	
	In fact, the stated dependence on $\bar \lambda = \lambda_{I+1}$ can be actually improved to the smaller quantity $\lambda_{I+1} - \lambda_{I}$. Note that all the differences $\lambda_{j+1}-\lambda_j$ depend only on the domain and the operator and are therefore universal.
\end{remark}

\begin{remark}
	The corresponding $L^{p_0}$--$L^{p_1}$
	estimates in \Cref{thm:lin}  also hold, with a uniform constant depending on $\bar \lambda$.
\end{remark}

\begin{proof}
Recall that $\bG_\lambda(\psi)$ is the spectral solution of
\[\begin{cases}
{\cL}\bG_\lambda(\psi)-\lambda\bG_\lambda(\psi)=\psi
    & \textin\Omega\\
B\bG_\lambda(\psi)=0
    & \texton\p\Omega.
\end{cases}\]
Projecting this equation into $\widetilde{E^\perp}$ according to \eqref{eq:proj-fu}, we have
\[\begin{cases}
{\cL}\bG_\lambda(\psi)^\perp
-\lambda\bG_\lambda(\psi)^\perp
=\psi^\perp
    & \textin\Omega\\
B\bG_\lambda(\psi)^\perp
=0
    & \texton\p\Omega,
\end{cases}\]
where
\[\begin{split}
\psi^\perp(x)
&=\psi(x)
	-\sum_{j=1}^{I}
		\angles{\psi,\varphi_j}
		\varphi_j(x)\\
\bG_\lambda(\psi)^\perp(x)
&=\bG_\lambda(\psi)(x)
	-\sum_{j=1}^{I}
		\dfrac{
			\angles{\psi,\varphi_j}
		}{
			\lambda_j-\lambda
		}
		\varphi_j(x) = \bG_\lambda (\psi^\perp).
\end{split}\]
By \Cref{prop:plin-Lq}
for $\psi \in \delta^\alpha L^\infty(\Omega)$
\begin{equation*}
	|\bG_\lambda (\psi )^\perp |  \le C(\alpha, \bar \lambda) \left \| \frac{ \psi^\perp  } {\delta^\alpha}\right \|_{L^\infty (\Omega)} \bG_0 (\delta^\alpha).
\end{equation*}

 For  any $f^\perp\in E^\perp$,
\[\begin{split}
\int_{\Omega}
    f^\perp \bG_\lambda(\psi)
\,dx
&=\int_{\Omega}
    f^\perp\left(
        \sum_{j=1}^{I}
            \dfrac{
                \angles{\psi,\varphi_j}
            }{
                \lambda_j-\lambda
            }\varphi_j
        +\bG_\lambda(\psi)^\perp
    \right)
\,dx
=\int_{\Omega}
    f^\perp \bG_\lambda(\psi)^\perp
\,dx.
\end{split}\]
Therefore, if $f \in L^1 (\Omega; \bG_0(\delta^\alpha))$ we have that
\begin{equation}\label{eq:plin-weak-perp}
\int_{\Omega}
    u^\perp \psi
\,dx
=\int_{\Omega}
    f^\perp \bG_\lambda(\psi)^\perp
\,dx,
    \quad \forall \psi\in \delta^\alpha L^\infty(\Omega).
\end{equation}
We take $\psi=\sign(u^\perp) \delta^\alpha \in \delta^\alpha L^\infty(\Omega)$ such that
\[
\int_{\Omega}
    |u^\perp| \delta^\alpha
\,dx
=\int_{\Omega}
    f^\perp
    \bG_\lambda(\psi)^\perp
\,dx
\leq
    C(\alpha ,\bar\lambda)
    \int_{\Omega}
        |f^\perp| \bG_0(\delta^\alpha )
    \,dx.
\]

 When  $p\in(1,\frac{n}{ n - 2s})$, for $k=1,2,\dots$ we take $\psi=(|u^\perp|\wedge k)^{p-1}\sign(u^\perp)\in L^\infty(\Omega)$, which has uniformly bounded $L^{\frac{p}{p-1}}(\Omega)$-norm, such that
\[
\int_\Omega (|u^\perp|\wedge k)^{p} \le
\int_{\Omega}
    |u^\perp|(|u^\perp|\wedge k)^{p-1}
\,dx
\leq
    C(q,\bar\lambda)
    \norm[L^{\frac{p}{p-1}}(\Omega)]
        {\psi^\perp}
    \int_{\Omega}
        |f^\perp|
    \,dx.
\]
Since, by H\"{o}lder's inequality,
\[\begin{split}
\norm[L^{\frac{p}{p-1}}(\Omega)]{\psi^\perp}
&\leq
    \norm[L^{\frac{p}{p-1}}(\Omega)]{\psi}
    +\sum_{j=1}^{I}
        \angles{\psi,\varphi_j}
        \norm[L^{\frac{p}{p-1}}(\Omega)]{\varphi_j}\\
&\leq
    \norm[L^{\frac{p}{p-1}}(\Omega)]{\psi}
    \left(
        1+\sum_{j=1}^{I}
            \norm[L^{p}(\Omega)]{\varphi_j}
            \norm[L^{\frac{p}{p-1}}(\Omega)]{\varphi_j}
    \right)\\
&\leq C(p,\bar\lambda)
    \norm[L^{\frac{p}{p-1}}(\Omega)]{\psi}
\end{split}\]
and
\[\begin{split}
\norm[L^{\frac{p}{p-1}}(\Omega)]{\psi}
&=\left(
    \int_{\Omega}
        (|u^\perp|\wedge k)
    \,dx
\right)^{p-1}
\leq
    \left(
        \int_{\Omega}
            (|u^\perp|\wedge k)^p
        \,dx
    \right)^{\frac{p-1}{p}},
\end{split}\]
We conclude that
\[
\norm[L^p(\Omega)]{|u^\perp|\wedge k}
\leq C(p,\bar\lambda)
    \int_{\Omega}
        |f^\perp|
    \,dx.
\]
The desired estimate then follows by taking $k\to+\infty$.

A similar argument applies for $p = 1$.
\end{proof}

\begin{proof}[Proof of \Cref{thm:plin}]
The existence-uniqueness of solutions of the projected equation follows from \Cref{lem:proj} and \Cref{thm:lin}. Thus $\bG_\lambda$ is well-defined in corresponding subspaces of $E^\perp$. The uniform estimates follow from \Cref{lem:plin-L2}, \Cref{prop:plin-Lq}, and  \Cref{thm:plin-L1}.
\end{proof}

We turn to the uniform $L^\infty$-bound in compact subsets.
\begin{prop}
\label{prop:plin-loc}
Suppose $f^\perp\in  L^\infty_{\loc} (\Omega) \cap \widetilde{ E^\perp }$, and $u^\perp\in \widetilde{E^\perp}$ is the $\bG_\lambda(\delta^\gamma L^\infty(\Omega))$-weak solution of \eqref{eq:plin}. Then for any $K\Subset K_0\Subset \Omega$,
\[
\norm[L^\infty(K)]{u^\perp}
\leq C(\bar\lambda,K,K_0)
	\left(
	\norm[L^\infty(K_0)]{f^\perp}
	+\norm[L^1(\Omega)]
		{f^\perp\delta^\gamma}
	\right).
\]
\end{prop}

\begin{proof}
By \Cref{prop:lin-loc} applied to \eqref{eq:plin},
\[
\norm[L^\infty(K)]{u^\perp}
\leq C(\bar\lambda,K,K_0)
	\left(
		\norm[L^1(\Omega)]{u^\perp \delta^\gamma }
		+\norm[L^\infty(K_0)]{f^\perp}
		+\norm[L^1(\Omega)]
			{f^\perp\delta^\gamma}
	\right).
\]
Now \Cref{thm:plin-L1} with $\alpha=\gamma$ allows $\norm[L^1(\Omega)]{u^\perp \delta^\gamma }$ to be absorbed by $\norm[L^1(\Omega)]{f^\perp\delta^\gamma}$ independently of $d_\Sigma(\lambda)$.
\end{proof}

We conclude this section by controlling  the  norms of $f^\perp$ in terms of those of $f$.
\begin{lem}[Norms of $f^\perp$]
\label{lem:norm-f}
For any $f\in L^1(\Omega;\delta^\gamma)\cap L^\infty_\loc(\Omega)$,
\[
\norm[L^1(\Omega)]{f^\perp \delta^\gamma}
\leq C(\bar\lambda)
	\norm[L^1(\Omega)]{f \delta^\gamma},
\]
and for any $K_0\Subset\Omega$,
\[
\norm[L^\infty(K_0)]{f^\perp}
\leq C(\bar\lambda,K_0)
	\left(
		\norm[L^\infty(K_0)]{f}
		+\norm[L^1(\Omega)]{f\delta^\gamma}
	\right).
\]
\end{lem}

\begin{proof}
Using \Cref{prop:eigenfunctions regularity}, we compute
\[\begin{split}
\int_{\Omega}
    |f^\perp|\delta^\gamma
\,dx
&\leq
\int_{\Omega}
    |f|\delta^\gamma
\,dx
+\sum_{j=1}^{I}
    \int_{\Omega}
        |f||\varphi_j|
    \,dx
    \int_{\Omega}
        |\varphi_j|\delta^\gamma
    \,dx\\
&\leq
    \int_{\Omega}
        |f|\delta^\gamma
    \,dx
    \left(
        1+\sum_{j=1}^{I}
        \norm[L^\infty(\Omega)]{\varphi_j/\delta^\gamma}
        \norm[L^1(\Omega)]{\varphi_j\delta^\gamma}
    \right),
\end{split}\]
and for any $K_0\Subset \Omega$,
\begin{align*}
\norm[L^\infty(K_0)]{f^\perp}
&\leq
	\norm[L^\infty(K_0)]{f}
	+\sum_{j=1}^{I}
		\int_{\Omega}
			|f||\varphi_j|
		\,dx
		\cdot\norm[L^\infty(K_0)]{\varphi_j}\\
&\leq
	\norm[L^\infty(K_0)]{f}
	+\left(
		\sum_{j=1}^{I}
			\norm[L^\infty(\Omega)]
				{\varphi_j/\delta^\gamma}
			\norm[L^\infty(\Omega)]
				{\varphi_j}
	\right)
	\norm[L^1(\Omega)]{f\delta^\gamma}.\qedhere
\end{align*}
\end{proof}

\begin{proof}[Proof of \Cref{thm:plin-2}]
The existence-uniqueness follows from \Cref{lem:proj} and \Cref{thm:lin}. The uniform global $L^1(\Omega;\delta^\gamma)$ estimate follows from \Cref{thm:plin-L1} as well as \Cref{lem:norm-f}, while the uniform local $L^\infty$ bound follows from \Cref{prop:plin-loc} and \Cref{lem:norm-f}.
\end{proof}

\section{Maximum principle}
\label{sec:MP}

In this section we prove a maximum principle for weighted $L^1$ functions when $\lambda<\lambda_1$, where the operator is known to be positive in the $L^2$-sense.  We assume (K1) throughout the section.  Suppose $f\in L^1(\Omega;\delta^\gamma)$ and consider the $\bG_\lambda(\delta^\gamma L^\infty(\Omega))$-weak solution $u\in L^1(\Omega;\delta^\gamma)$ of
\begin{equation}\label{eq:MP-f}\begin{cases}
{\cL}u-\lambda u=f
    & \textin \Omega\\
Bu=0
    & \texton \p\Omega.
\end{cases}\end{equation}
We will use the Poincar\'{e} inequality \cite[Proposition 5.1]{BFV}
\begin{equation}
\label{eq:Poincare}
	\lambda_1 \int_\Omega \phi \bG_{0} (\phi)\,dx
\le \int_\Omega \phi^2 \,dx,
    \quad \forall \phi\in L^2(\Omega).
\end{equation}

\begin{lem}[Maximum principle for spectral solutions]
\label{lem:MP-L2}
Suppose $f\in L^2(\Omega)$ and $u\in \bH^{2}_{\cL}(\Omega)$ is a spectral solution of \eqref{eq:MP-f}. If $f\geq 0$ and $\lambda<\lambda_1$, then $u\geq 0$.
\end{lem}

\begin{proof}
We write
\begin{equation*}
	\int_\Omega u \psi\,dx
= \lambda \int_\Omega u \bG_0 (\psi) \,dx
 + \int_\Omega f \bG_0 (\psi)\,dx.
\end{equation*}
Taking $ \psi = - u_- \le 0 $  where $u_-=\max\set{-u,0}$,  we have that
\begin{align*}
	\int_\Omega u_-^2 \,dx
&= - \lambda \int_\Omega u_+ \bG_0 (u_-) \,dx
    +  \lambda \int_\Omega u_- \bG_0 (u_-) \,dx
    -  \int_\Omega f \bG_0 (u_-) \,dx.
\end{align*}
Since $u_- \ge 0$, it is clear that $\bG_0 (u_-) \ge 0$. Then we have
\begin{align*}
		\int_\Omega u_-^2 \,dx
&\le   \lambda \int_\Omega u_- \bG_0 (u_-) \,dx.
\end{align*}
Applying the Poincar\'{e}  inequality \eqref{eq:Poincare}, 
\begin{align*}
	\lambda_1 \int_\Omega u_- \bG_0 (u_- ) \,dx
 &\le \lambda \int_\Omega u_- \bG_0 (u_-) \,dx.
\end{align*}
Since $\lambda < \lambda_1$ we have that
\begin{equation*}
	\int_\Omega u_- \bG_0 (u_- ) \,dx= 0.
\end{equation*}
If $u_- \not \equiv 0$, then $\bG_0 (u_-) > 0$ in $\Omega$ and we arrive at a contradiction. Hence, we deduce that $u_- = 0$.
\end{proof}

Now \Cref{thm:MP} is restated and proved in the following

\begin{lem}[Maximum principle for $\bG_\lambda(\delta^\gamma L^\infty(\Omega))$-weak solutions]
\label{lem:MP-L1}
Let $f\in L^1(\Omega;\delta^\gamma)$ and $u\in L^1(\Omega;\delta^\gamma)$ be a $\bG_\lambda(\delta^\gamma L^\infty(\Omega))$-weak solution of \eqref{eq:MP-f}. If $f\geq 0$ and $\lambda<\lambda_1$, then $u\geq 0$.
\end{lem}

\begin{proof}
Split $u=u_+-u_-$, with $u_+,u_-\in L^1(\Omega;\delta^\gamma)$ non-negative. Let $K\Subset\Omega$. Testing \eqref{eq:MP-f} with $\psi=\mathbf{1}_{\set{u<0}\cap K}\geq 0$, we have
\[
-\int_{\set{u<0}\cap K}
    u_-
\,dx
=\int_{\Omega}
    u\psi
\,dx
=\int_{\Omega}
    f\bG_\lambda(\psi)
\,dx
\geq0,
\]
in view of \Cref{lem:MP-L2}. This forces $u_-=0$ a.e. in any $K\Subset\Omega$, i.e. $u\geq0$ a.e. in $\Omega$.
\end{proof}

\section{Inhomogeneous eigenvalue problem}
\label{sec:blow}

In this section we complete the proof of \Cref{thm:blow}.

\begin{proof}[Proof of \Cref{thm:blow}]
We split the proof into a few steps.
\begin{enumerate}
\item
	Taking the difference $u_\lambda=v_\lambda-v_h$, where $v_\lambda$ and $v_h$ are $\bG_0(C_c^\infty(\Omega))$-weak solutions of \eqref{eq:main} and \eqref{eq:vh} respectively, we see that $u_\lambda$ is a $\bG_0(C_c^\infty(\Omega))$-weak solution (hence in any of the sense \eqref{eq:sense-1}--\eqref{eq:sense-5} in \Cref{thm:notions}) of
\begin{equation}\label{eq:main-u}
\begin{cases}
{\cL}u_\lambda-\lambda u_\lambda=g+\lambda v_h
	& \textin \Omega\\
B u_\lambda=0
	& \texton \p\Omega.
\end{cases}
\end{equation}
We have that $v_h\in\delta^{-b}L^\infty(\Omega)\subset L^1(\Omega;\delta^\gamma)\cap L^\infty_\loc(\Omega)$; indeed,
\begin{equation}\label{eq:vh-dgamma}\begin{split}
\norm[L^1(\Omega)]{v_h\delta^\gamma}
&\leq
	\norm[L^1(\Omega)]{v_h\delta^b\cdot\delta^{2s-1}}
\leq
	\norm[L^\infty(\Omega)]{v_h\delta^b}
	\norm[L^1(\Omega)]{\delta^{2s-1}}
\leq C\norm[L^\infty(\Omega)]{v_h\delta^b}.
\end{split}\end{equation}
Hence, \Cref{thm:lin} implies the existence of a unique solution $u_\lambda$ of \eqref{eq:main-u} in $L^1(\Omega; \delta^\gamma )\cap L^\infty_\loc(\Omega)$. In particular, we obtain the existence of a unique solution
\[
v_\lambda
=v_h+u_\lambda
\in L^1(\Omega;\delta^\gamma)
	\cap L^\infty_\loc(\Omega)
\]
of \eqref{eq:main}, for any $\lambda\in\R\setminus\Sigma$.

\item
	The representation formula follows directly from \Cref{thm:plin-2} applied to $f=g+\lambda v_h\in L^1(\Omega;\delta^\gamma)\cap L^\infty_\loc(\Omega)$, under the projection \eqref{eq:proj-fu}.

\item
	By \Cref{thm:plin-2}, we have the estimates
\[
\norm[L^{1}(\Omega)]{u_\lambda^\perp \delta^\gamma }
\leq C(\bar\lambda)
	\norm[L^{1}(\Omega)]{(g+\lambda v_h) \delta^\gamma  }
\]
and, for $K\Subset K_0\Subset \Omega$,
\[
\norm[L^\infty(K)]{u_\lambda^\perp}
\leq C(\bar\lambda,K,K_0)
	\left(
		\norm[L^\infty(K_0)]
			{g+\lambda v_h}
		+\norm[L^1(\Omega)]
			{(g+\lambda v_h)\delta^\gamma}
	\right).
\]
But then the norms of $v_h$ can be controlled in terms of $\norm[L^\infty(\p\Omega)]{h}$ by \eqref{eq:vh-loc} and \eqref{eq:vh-dgamma}.

\item Let us write
\begin{align*}
	v_\lambda &= v_h + \sum_{\substack{j=1\\ \lambda_j \ne \lambda_i }}^{{I}} \dfrac{
		\angles{g+\lambda v_h,\varphi_j}
	}{
		\lambda_j-\lambda
	}\varphi_j
+\frac 1  {
	\lambda_i-\lambda
} \sum_{\substack{j : \lambda_j = \lambda_i }} {
	\angles{g+\lambda v_h,\varphi_j}
}\varphi_j
+ \bG_\lambda (  (g + \lambda v_h )^\perp  ) \\
&=v_h
+\frac {P_{E_i} (g + \lambda v_h)}  {
	\lambda_i-\lambda
}
+  \bG_{\lambda} (  g + \lambda v_h - P_{E_i} (g + \lambda v_h)  ).
\end{align*}

By our previous construction $v_h + \bG_\lambda (  (g + \lambda v_h )^\perp  )$ is uniformly bounded in $L^\infty (K)$ for any $K \Subset \Omega$. It is easy to check that
\begin{equation*}
	v_h + \bG_\lambda (  (g + \lambda v_h )^\perp  ) \to v_h + \bG_{\lambda_i} (  (g + \lambda_i v_h )^\perp  )
\end{equation*}
over compacts.
The term
\begin{equation*}
	 \sum_{\substack{j=1\\ \lambda_j \ne \lambda_i }}^{{I}} \dfrac{
		\angles{g+\lambda v_h,\varphi_j}
	}{
		\lambda_j-\lambda
	}\varphi_j
\longrightarrow
\sum_{\substack{j=1\\ \lambda_j \ne \lambda_i }}^{I} \dfrac{
	\angles{g+\lambda_i v_h,\varphi_j}
}{
\lambda_j-\lambda_i
}\varphi_j
\end{equation*}
is in $\delta^\gamma L^\infty (\Omega)$. Hence,
\begin{equation*}
	v_h +  \bG_{\lambda} \Big(  g + \lambda v_h - P_{E_i} (g + \lambda v_h) \Big ) \to v_h +  \bG_{\lambda_i} \Big  (  g + \lambda v_h - P_{E_i} (g + \lambda_i v_h)  \Big )
\end{equation*}
in $L^1 (\Omega; \delta^\gamma)$.

The last term is given by the projection
\begin{equation*}
	P_{E_i} ( g + \lambda v_h) = \sum_{\substack{j : \lambda_j = \lambda_i }}{
		\angles{g+\lambda v_h,\varphi_j}
	}\varphi_j.
\end{equation*}
By continuity of the projection,
\begin{equation*}
		P_{E_i} ( g + \lambda v_h) \longrightarrow  P_{E_i} ( g + \lambda_i v_h)
\end{equation*}
as $\lambda \to \lambda_i$. Hence, either $P_{E_i} ( g + \lambda_i v_h) = 0$ or there exists $K \Subset \Omega$ and $\varepsilon$ of positive measure such that $|P_{E_i} ( g + \lambda_i v_h) | > \varepsilon$ on $K$ and therefore
$$
	 \left |\frac{P_{E_i} ( g + \lambda v_h)}{\lambda_i - \lambda }  \right |  \to +\infty \text{ on } K.
$$
Since the rest of the terms of $v_h$ are uniformly  locally essentially bounded, this $\|v_h \|_{L^\infty (K)}$ blows up.

If $P_{E_i} (g + \lambda_i v_h) = 0$, since all the considered eigenfunctions are linearly independent, we have that $\langle g + \lambda_i v_h , \varphi_j  \rangle = 0$. Then, we apply L'Hôpital's rule to deduce that
\begin{equation*}
	\lim_{\lambda \to \lambda_i} \frac {\angles{g+\lambda v_h,\varphi_j}}{ \lambda_i - \lambda } = \lim_{\lambda \to \lambda_i} \frac {\angles{g,\varphi_j} + \lambda \angles{v_h, \varphi_j} }{ \lambda_i - \lambda }  = -\angles{v_h,\varphi_j}.
\end{equation*}
Therefore,
\begin{equation*}
	\frac {P_{E_i} (g + \lambda v_h)}  {
		\lambda_i-\lambda }= \frac 1  {
		\lambda_i-\lambda
	} \sum_{\substack{j : \lambda_j = \lambda_i }} {
		\angles{g+\lambda v_h,\varphi_j}
	}\varphi_j \to -  \sum_{\substack{j : \lambda_j = \lambda_i }} {
\angles{ v_h,\varphi_j} = - P_{E_i} (v_h)
}
\end{equation*}
in $\delta^\gamma L^\infty (\Omega)$, and the limit is characterised.
\item
To prove the global blow-up behavior, suppose $g\geq0$ and $h>0$. By \eqref{eq:vh-martin}, $v_h>0$. In fact, \eqref{eq:vh-bdry} and \eqref{eq:vh-martin} imply that
\begin{equation}\label{eq:vh-lower}
v_h(x)
\geq C
	\left(\inf_{\p\Omega}h\right)
	\delta^{-b}(x)
		\quad \textfor x\in\Omega.
\end{equation}
Also, since $g+\lambda v_h>0$, by the maximum principle \Cref{thm:MP}, $u_\lambda=v-v_h\geq0$.

The solution $v_\lambda$ can be expressed in two ways, namely
\begin{equation}\label{eq:v-lambda-1}
v_\lambda
=v_h+u_\lambda,
	\quad
u_\lambda\geq 0,
\end{equation}
and
\begin{equation}\label{eq:v-lambda-2}
v_\lambda
=
	\dfrac{1}{\lambda_1-\lambda}
	\angles{g+\lambda v_h,\varphi_1}\varphi_1
	+v_h+u_\lambda^\perp,
		\quad
u_\lambda^\perp \in L^\infty_\loc(\Omega).
\end{equation}
Given any $\bar{C}>0$, there exists $K=K(\bar{C})\Subset\Omega$ such that $v_h\geq \bar{C}$ in $\Omega\setminus K$ by \eqref{eq:vh-lower}. Then \eqref{eq:v-lambda-1} yields
\[
\inf_{\Omega\setminus K}v_\lambda
\geq \bar{C}.
\]
On the other hand, \eqref{eq:u-perp-loc} implies that
\[
\norm[L^\infty(K)]{u^\perp}
\leq C(\lambda_1,K)
	\left(
		\norm[L^\infty(\set{\delta>\dist(K,\p\Omega)/2})]
			{g}
		+\norm[L^1(\Omega)]{g\delta^\gamma}
		+\norm[L^\infty(\p\Omega)]{h}
	\right).
\]
Then \eqref{eq:v-lambda-2} implies
\[\begin{split}
\inf_{K}v_\lambda
&\geq
	\dfrac{\lambda}{\lambda_1-\lambda}
	\cdot C
		\left(\inf_{\p\Omega}h\right)
		\angles{\delta^{-b},\varphi_1}
		\inf_{K}\varphi_1\\
&\quad
	-C(K)
	\left(
		\norm[L^\infty(\set{\delta>\dist(K,\p\Omega)/2})]
			{g}
		+\norm[L^1(\Omega)]{g\delta^\gamma}
		+\norm[L^\infty(\p\Omega)]{h}
	\right)\\
&\geq
	\dfrac{
		C(K,h)
	}{
		\lambda_1-\lambda
	}
	-C(K,g,h)\\
&\geq \bar{C},
\end{split}\]
for all $\lambda\in(\lambda_{\bar{C}}(g,h),\lambda_1)$. This completes the proof. \qedhere
\end{enumerate}
\end{proof}

\medskip

\section{The limit as $s$ tends to 1}\label{se.s1}
\label{sec:sto1}

In this section we address the issue of convergence of our problems in the limit when $s$ approaches 1. For convenience, we will index here our operators as $\cL_s$. Originally, our  interest was treating the RFL and SFL operators, where the passage to the limit is relatively simple and  produces in the limit a classical problem involving the Laplacian. The interesting feature is that we may thus observe how the boundary blow-up disappears in the limit, since the corresponding blow-up exponent goes to zero a $s\to 1$.

We propose a different writing, since we found interesting to state a list of hypotheses  on the general class $\cL_s$ for $s$ near 1 that  allow to show convergence of the original problems involving $\cL_s$  to a limit problem as $s\to1$. Naturally, we require that  the aforementioned theory holds true for each $s\in(0,1]$, so we assume 
\begin{equation}
\label{eq:s to 1 assumption 0}
	\cL_s \textrm{ satisfies (K1)-(K2) for all } s \in (0,1].
\end{equation}
 and we will denote the limit operator by $\cL_1$. In our main case of interest $\cL_1 = -\Delta$. At the end, we show that these assumptions hold for the RFL and SFL, see Subsection \ref{ss8.5}.

\medskip

 On  a first stage, let us simply assume that $\cL_1$ is indeed the limit of $\cL_s$, in the sense of resolvents:
\begin{align}
\label{eq:s to 1 assumption 1}
\|\bG_{\cL_s} (f) - \bG_{\cL_1} (f)\|_{L^2  (\Omega)} &\le \omega(1-s) \| f \|_{L^2} , &&\text{where } \omega(1-s) \searrow 0 \text{ as } s \to 1.
\end{align}
We will show that it suffices to have
\begin{align*}
\bG_{\cL_s -\lambda } (f) & \longrightarrow  \bG_{\cL_1 - \lambda} (f) && \textin L^2 (\Omega), \textforall f \in L^2 (\Omega),
\end{align*}
at least when $\lambda \notin \Sigma (\cL_1)$, so the nonsingular eigenvalue problem is studied.

To study the large eigenvalue problem, we will also  assume that $\cL_s$ satisfies (K2).
Since we are mostly  interested  in approximating the Laplacian's behaviour, let us assume that
\begin{equation}
\label{eq:s to 1 assumption 2}
2s > \gamma(s).
\end{equation}
for $s$ sufficiently close to 1. This is satisfied for the RFL and the SFL with $s > \frac 1 2$. In this setting $\bM_{\cL_s } (h) \in  L^1 (\Omega)$ for $h\in L^\infty(\p\Omega)$. We need to assume some convergence of the Martin operator. In particular, we will assume that
\begin{align}
\label{eq:s to 1 assumption 3}
\bM_{\cL_s} (h) &\longrightarrow \bM_{\cL_1} (h) , && \textin L^1 (\Omega) \textforall h \in L^\infty (\partial \Omega).
\end{align}

\begin{remark}
	Notice that $\bM_{-\Delta}$ is the usual Poisson integral operator, i.e. $u = \bM_{-\Delta}(h)$ is the unique solution of
	\begin{equation*}
		\begin{dcases}
		-\Delta u = 0 & \textin \Omega, \\
		u = h & \texton \partial \Omega.
		\end{dcases}
	\end{equation*}
\end{remark}
The last condition we will ask is that there are uniform $L^p \to L^q$ bounds in two special cases
\begin{align}
\label{eq:s to 1 assumption 4}
\| \bG_{\cL_s - \lambda} (f) \|_{L^{\frac{n}{n-1} } (\Omega)} &\le C (\lambda) \| f \|_{L^1 (\Omega)}  \\
\label{eq:s to 1 assumption 5}
\| \bG_{\cL_s - \lambda} (f) \|_{L^{\infty } (\Omega)} &\le C (\lambda) \| f \|_{L^\infty (\Omega)}
\end{align}
where the constant $C(\lambda)$ does not depend on $s$.

We will check that these two conditions are implied by
\begin{equation}
\label{eq:s to 1 assumption 6}
\cG_{\cL_s} (x,y) \le C_\Omega |x-y|^{-(n-2s)} \quad \text{ where } C_\Omega \text{ does not depend on } s \text{ close to 1}.
\end{equation}
This condition is reasonable since all operators in the family satisfy (K1). This will be the condition that we check in the examples.

Under these  hypotheses  we will able to prove the convergence of large solutions as $s\to 1^-$.  Furthermore, we will have that, if $f_s$ are bounded in $L^2 (\Omega)$, then $\bG_s (f_s)$ has a strongly convergent subsequence.

\subsection{Compactness theorem}
\begin{thm}
	\label{thm:s to 1 convergent subsequence}
	Assume \eqref{eq:s to 1 assumption 0}, \eqref{eq:s to 1 assumption 1} and let $f_s$ be such that $\|f_s\|_{L^2(\Omega)}$ is bounded. Then, there exists a sequence $s_m \to 1$ such that $f_{s_m} \rightharpoonup f$  weakly  in $L^2 (\Omega)$ and $ \bG_{\cL_{s_n}} (f_{s_n}) \to \bG_{\cL_1} (f)$ strongly in $L^2 (\Omega)$.
\end{thm}

For each $s$, the compactness of $\bG_{\cL_s}$ proved in \cite[Proposition 5.1]{BFV} uses the  Riesz--Fr\'{e}chet--Kolmogorov  theorem and the fact that, since\footnote{Note that in $\bG_{\cL_s}^{(2)}$ the superscript denotes a second part, a notation borrowed from \cite{BFV}, as opposed to a power used in the rest of this paper.} $\bG_{\cL_s}^{(2)} (x,y) = \bG_{\cL_s}(x,y) \chi_{|x-y|>\varepsilon} \in L^2 (\Omega)$, then $\|\bG_{\cL_s}^{(2)} (\cdot + h, \cdot ) - \bG_{\cL_s}^{(2)} \|_{L^2 (\Omega \times \Omega)}\to 0$ as $h \to 0$.
The authors prove that under (K1), for each $s \in (0,1]$,
\begin{equation*}
	\lim_{|h| \to 0} \sup_{\|f \|_{L^2 (\Omega)} \le 1 } \| \tau_h \bG_{\cL_{s}} (f_{s}) -  \bG_{\cL_{s}} (f_{s}) \| _{L^2 (\Omega)} = 0.
\end{equation*}
Let us denote the modulus of continuity by $\omega_s$,
\begin{equation*}
	\omega_s (\varepsilon ) = \sup_{|h| \le \varepsilon} \sup_{\|f \|_{L^2 (\Omega)} \le 1 } \| \tau_h \bG_{\cL_{s}} (f_{s}) -  \bG_{\cL_{s}} (f_{s}) \|_{L^2 (\Omega)}.
\end{equation*}
The condition that this happens uniformly would be too severe.
In order to prove our result we need a very specific form of the Riesz--Fr\'{e}chet--Kolmogorov theorem. For $h \in \mathbb R^n$ we define
\begin{equation*}
	\tau_h u (x) =
    \begin{cases}
	u(x+h) & x+h \in \Omega, \\ 
    0 & x+h\notin \Omega.
	\end{cases}
\end{equation*}
\begin{lemma}[Remark 6 (II) in \cite{HH}]
	\label{lem:sequence Frechet-Kolmogorov}
	Let $u_m \in L^p(\Omega)$ ($1\le p <+\infty$) be a sequence such that
	\begin{equation*}
		\| \tau _h u_m - u_m \|_{L^p (\Omega)} \le \omega_1 (|h|) + \omega_2(\tfrac 1 m),
	\end{equation*}
	where $\omega_1(\varepsilon), \omega_2(\varepsilon) \to 0$ as $\varepsilon \to 0$. Then, $u_m$ has a subsequence that converges strongly in $L^p (\Omega)$.
\end{lemma}

\begin{proof}[Proof of \Cref{thm:s to 1 convergent subsequence}]
	We write
	\begin{align*}
		\| \tau_h \bG_{\cL_{s}} (f_{s}) -  \bG_{\cL_{s}} (f_{s}) \| _{L^2 (\Omega)} & \le \| \tau_h \bG_{\cL_{s}} (f_{s}) -  \tau_h  \bG_{\cL_{1}} (f_{s}) \|_{L^2 (\Omega)} \\
		&\quad + \| \tau_h \bG_{\cL_{1}} (f_{s}) -  \bG_{\cL_{1}} (f_{s}) \|_{L^2 (\Omega)} \\
		&\quad + \|  \bG_{\cL_{1}} (f_{s}) -  \bG_{\cL_{s}} (f_{s}) \|_{L^2 (\Omega)} \\
		&\le 2 \omega (1-s) + \omega_1 (|h|) \| f_s \|_{L^2 (\Omega)},
	\end{align*}
	where $\omega_1(|h|)$ follows from the continuity of $\bG_{\cL_1}$. Take any sequence $s_m \to 1^-$. Since $f_{s_m}$ is bounded, up to a subsequence there exists $f \in L^2 (\Omega)$ such that $f_{s_m} \rightharpoonup f$ weakly in $L^2 (\Omega)$. Applying \Cref{lem:sequence Frechet-Kolmogorov}, we have that $u_m = \bG_{\cL_{s_m}} (f_{s_m})$ has a strongly $L^2 (\Omega)$ convergent subsequence and its limit be $v$. Using the duality formula and \eqref{eq:s to 1 assumption 1} we have that $v = \bG_{\cL_1} (f)$.
\end{proof}

\subsection{Spectral convergence}
\label{sec:s to 1 spectral}
It is well known that
\begin{equation*}
	\sup_j \dist (  \lambda_j (\bG_{\cL_s} ), \Sigma (\bG_{\cL_1} )   ) \le \|\bG_{\cL_s} - \bG_{\cL_1} \|_{\cL (L^2 , L^2)} \le \omega(1-s).
\end{equation*}
(see \cite[Theorem 4.10 in Chapter 5] {Kato}). Therefore,
\begin{equation*}
	\sup_j  \dist (  \lambda_j ({\cL_s}) ^{-1}, \Sigma ({\cL_1}) ^{-1}   ) \le \omega(1-s).
\end{equation*}
\begin{remark}
	The first eigenvalues are uniformly bounded below. Assuming \eqref{eq:s to 1 assumption 1} we have that
	\begin{equation*}
		\lambda_1 (\cL_s)^{-1} \le \lambda_1 (\cL_1)^{-1} + \dist (  \lambda_1 ({\cL_s}) ^{-1}, \Sigma ({\cL_1}) ^{-1}   ) \le \lambda_1 (\cL_1)^{-1} + \omega(1-s).
	\end{equation*}
	Hence,
	\begin{equation*}
		\lambda_1 (\cL_s) \ge \frac{1}{\lambda_1 (\cL_1)^{-1} + \omega(1-s)}.
	\end{equation*}
\end{remark}
\begin{remark}
	Assuming only \eqref{eq:s to 1 assumption 6} also gives some information.
	Applying the Rayleigh--Faber--Krahn theorem guaranties and \Cref{prop:Riesz potentials} we have that  (recall that $\mathcal{I}_{n-2s}$ is the Riesz potential of order $n-2s$) 
	\begin{equation*}
		\frac{1}{\lambda_1 (\cL_s)} = \inf_{  f \ne 0 } \frac{ \| \bG_{\cL_s} (f) \|_{L^2 (\Omega)} }{ \| f \|_{L^2 (\Omega)} } \le C_\Omega  \inf_{  f \ne 0 } \frac{  \| \mathcal{I}_{n-2s} (f \chi_ \Omega ) \|_{L^2 (\Omega)} }{ \| f \|_{L^2 (\Omega)} } = C_\Omega \lambda_1  (\mathcal{I}_{n-2s}).
	\end{equation*}
\end{remark}

Furthermore, we have strong convergence of the eigenspaces.
\begin{prop}
	Assume \eqref{eq:s to 1 assumption 1} and $\lambda_j(\cL_s)\to \lambda_i(\cL_1)$ for some $j,i$. Then $E_j (\cL_s) \to E_i (\cL_1)$ in $L^2 (\Omega)$ for some $i$, in the sense that if $\phi_s \in E_j (\cL_s)$ are such that $\| \phi_s \|_{L^2} = 1$, then there exists $\phi_1 \in E_i( \cL _1)$ such that, up to a subsequence, $\phi_s \to \phi_1$ strongly in $L^2 (\Omega)$.
\end{prop}
\begin{proof}
	Since $(\phi_s)$ satisfies $\norm[L^2(\Omega)]{\phi_s}=1$ and $\phi_s = \lambda_j (\cL_s) \bG_{\cL_s} (\phi_s)$, we can apply \Cref{thm:s to 1 convergent subsequence} to show that there exists $\phi_{s_m} \rightharpoonup \phi_1$ in $L^2(\Omega)$ with $\bG_{\cL_{s_m}} (\phi_{s_m}) \to \bG_{\cL_1} (\phi_1)$ in $L^2(\Omega)$. But then, using $\lambda_j (\cL_s) \to \lambda_i (\cL_1)$,
	\begin{equation*}
		\phi_{s_m} = \lambda_j (\cL_{s_m}) \bG_{\cL_{s_m}} (\phi_{s_m}) \to  \lambda_i (\cL_{1}) \bG_{\cL_{1}} (\phi_{1})
	\end{equation*}
	strongly in $L^2 (\Omega)$, so the convergence $\phi_{s_m} \to \phi_1$ is strong,  $\| \phi_1 \|_{L^2(\Omega)} = 1$ and
	\[
	\phi_1=\lambda_i(\cL_1)\bG_{\cL_1}(\phi_1)
	    \quad \textin L^2(\Omega),
	\]
	i.e. $\phi_1\in E_i(\cL_1)$.
\end{proof}

\subsection{Convergence of $\cL_s - \lambda$ }

Let $\lambda \notin \Sigma(\cL_1)$, so $\lambda_i (\cL_1) < \lambda < \lambda_{i+1} (\cL_1)$ for some $i$. Let $s_0=s_0(\lambda)$ be sufficiently close to $1$ such that
\[
\omega(1-s)\leq \dfrac12 \dist(\lambda^{-1},\Sigma(\cL_1)^{-1}),
    \quad \textfor s\in(s_0,1).
\]
Then, for $s\in(s_0,1)$, we have that
$$
	\dist( \lambda^{-1} , \Sigma (\cL_s)^{-1} )
\ge \dist( \lambda^{-1} , \Sigma (\cL_1)^{-1} )  - \omega(1-s)
\geq \dfrac12 \dist( \lambda^{-1} , \Sigma (\cL_1)^{-1} ).
$$
Therefore,
\begin{equation*}
		\dist( \lambda , \Sigma (\cL_s) ) \ge c > 0.
\end{equation*}
for $s\in(s_0,1)$.

\begin{prop} Let $\lambda \notin \Sigma(\cL_1)$ and assume \eqref{eq:s to 1 assumption 1}. Then
	\begin{align*}
	\bG_{\cL_s -\lambda } (f) & \to \bG_{\cL_1 - \lambda} (f) && \textin L^2 (\Omega), \textforall f \in L^2 (\Omega).
	\end{align*}
\end{prop}
\begin{proof}
Let $f \in L^2 (\Omega)$ be fixed. Let $u_s = \bG_{\cL_s - \lambda } (f)$.
We know that
\begin{equation*}
	\|u_s \|_{L^2 (\Omega)} \le \frac{1}{\dist(\lambda, \Sigma (\cL_s))} \| f \|_{L^2 (\Omega)}.
\end{equation*}
Let us write the alternative formulation $u_s = \bG_{\cL_s} (\lambda u_s + f)$. Since $\lambda u_s + f$ is a bounded sequence in $L^2(\Omega)$, applying \Cref{thm:s to 1 convergent subsequence} there exists a subsequence $\lambda u_{s_m} + f \rightharpoonup g$ such that $u_{s_m} \to v = \bG_{\cL_1} (g)$ in $L^2 (\Omega)$. Hence $g = f + \lambda v$ and $v = \bG_{\cL_1} (\lambda v + f)$. By uniqueness $v = \bG_{\cL_1 - \lambda} (f)$. Since every sequence of $u_s$  has  a convergent subsequence converging to the same $v$, the whole sequence converges.
\end{proof}

\subsection{Inhomogeneous eigenvalue value problem}

\begin{prop}
	\label{prop:s to 1 inhomogeneous}
	Assume \eqref{eq:s to 1 assumption 0}--\eqref{eq:s to 1 assumption 5}, let $\lambda \notin \Sigma(\cL_1)$, $g \in L^2 (\Omega)$ and $h \in L^\infty (\partial \Omega)$. Then the solution of \eqref{eq:main} weakly converges in $L^1 (\Omega)$ to
	$$v_1 = \bM_{\cL_1} (h) + \bG_{\cL_1- \lambda} (g + \lambda \bM_{\cL_1} (h)  ).$$
\end{prop}
In some sense, this limit is the solution of
\begin{equation*}
\begin{dcases}
\cL_1 u - \lambda u = g & \Omega \\
B_1 u = h & \partial \Omega
\end{dcases}
\end{equation*}

\begin{proof}
	Let  us  consider the explicit form of the solution $$v_s = \bM_{\cL_s} (h) + \bG_{\cL_s- \lambda} (g + \lambda \bM_{\cL_s} (h)  ).$$
We know that $\bM_{\cL_s} (h)$ and $\bG_{\cL_s- \lambda} (g)$ converge in $L^1(\Omega)$ to $\bM_{\cL_1} (h)$ and $\bG_{\cL_s- \lambda} (g)$. Now we focus on $w_s = \bG_{\cL_s- \lambda} (\bM_{\cL_s} (h)  )$.  We have that
	\begin{equation}
		\| w_s \|_{L^{\frac{n}{n-1} } (\Omega)} \le C (\lambda) \| \bM_{\cL_s} (h)\|_{L^1 (\Omega)}  \le C.
	\end{equation}
	Therefore, $w_s$ as a weakly convergent subsequence in $L^{\frac n {n-1}}$ to some $w_1$. Writing the weak formulation
	\begin{equation*}
		\int_\Omega w_s \psi = \int_{\Omega} \bM_{\cL_s } (h) \bG_{\cL_s- \lambda} (\psi  ), \qquad \forall \psi \in L^\infty_c (\Omega).
	\end{equation*}
	Since $\bG_{\cL_s - \lambda} (\psi)$ is bounded in $L^\infty$, $(w_s)$ has a weak-$\star$-$L^\infty$ convergent subsequence. Since it converges weakly in $L^2$ to $\bG_{\cL_1 - \lambda} (\psi) $, this is its weak-$\star$-$L^\infty$ limit. Passing to the limit
	\begin{equation*}
	\int_\Omega w_1 \psi = \int_{\Omega} \bM_{\cL_1 } (h) \bG_{\cL_1- \lambda} (\psi  ), \qquad \forall \psi \in L^\infty_c (\Omega).
	\end{equation*}
	There $w_1 = \bG_{\cL_1 - \lambda} {(\bM_{ \cL_1 } (h))}$.
\end{proof}

\subsection{Uniform embeddings \eqref{eq:s to 1 assumption 4}--\eqref{eq:s to 1 assumption 5}}
\begin{remark}
	Notice that \eqref{eq:s to 1 assumption 4}--\eqref{eq:s to 1 assumption 5} are only used in \Cref{prop:s to 1 inhomogeneous}. Here, they are used to show  that  there exists a convergent subsequence.
	They can be replaced by $\| \bG_{\cL_s} (f) - \bG_{\cL_1} (f) \|_{L^1} \le \omega(1-s) \| f \|_{L^1}$.
	If one  assumes  that $\bM_{\cL_s} (h)$  converges  in $L^2 (\Omega)$, these  hypotheses  can be removed. This would not be abusive since $\bM_{-\Delta} (h)$ is bounded.
\end{remark}

\begin{prop}
	 Assume  \eqref{eq:s to 1 assumption 6}.
	Then \eqref{eq:s to 1 assumption 4} and \eqref{eq:s to 1 assumption 5} hold.
\end{prop}
\begin{proof}
We already know that
\begin{equation*}
	\| \bG_ {\cL_s - \lambda} (f) \|_{L^2} \le \frac{\| f \|_{L^2 (\Omega)} }{\dist (\lambda, \Sigma (\cL_s)) }\le 2 \frac{\| f \|_{L^2 (\Omega)} }{\dist (\lambda, \Sigma (\cL_1)) }
\end{equation*}
for $s$ close to $1$.

Let us go over the ideas in \Cref{sec:lin}, being quite careful on the dependence on $s$. Let $f \in L^{q}$ for $q > \frac n {2s}$. The idea for the $L^\infty$ estimate is to bootstrap via
\begin{equation*}
	u = \lambda^k \bG_0^k (u) + \sum_{m=1}^{k} \lambda^{m-1} \bG_0^m (f).
\end{equation*}
so that the sequence $p_0 = 2 - \varepsilon$,
\begin{equation*}
	\frac 1 { p_{m+1} } = \frac{1}{p_m} - \frac{2s}{n}
\end{equation*}
and pick $\varepsilon_s$ and $k_s$ so that $p_{k_s-1} < \frac{n}{2s} < p_{k_s}$.  In order to do this uniformly as $s \to 1$, we take $s > \frac 3 4$ 
so that the sequence $p_0 = 2 - \varepsilon$,
\begin{equation*}
	\frac 1 { p_{m+1} } = \frac{1}{p_m} - \frac{3}{2n}
\end{equation*}
and pick $\varepsilon$ and $k_s$ so that $p_{k-1} < \frac{n}{2s} < \frac {2n} 3 < p_{k}$.

Taking into account the embedding $L^2 \hookrightarrow L^{2-\varepsilon}$, we have that
\begin{align*}
	\|\bG_{\cL_s - \lambda} (f) \|_{L^\infty (\Omega)}
	&\le 2 |\Omega|^{\frac 2 \varepsilon} |\lambda|^k C_s(p_0,p_1) \cdots C_s(p_k,p_{k+1}) C_s (p_{k+1},\infty)\frac{ \|f \|_{L^2} }{\dist(\lambda, \Sigma (\cL_1))}\\
	&\quad + \|f\|_{L^q (\Omega)}  \sum_{m=1}^k |\lambda|^{m-1} C_s(q, \infty) C_s(\infty,\infty)^{m-1}
\end{align*}
where $C_s(p,q) = \| \bG_{\cL_s} \|_{\cL (L^p, L^q)}$ is the corresponding continuity modulus. As seen in \Cref{sec:lin}, this embedding constant is computed in two parts: the  first has to do with the bound 
\begin{equation*}
	\cG_{\cL_s}(x,y) \le C_\Omega |x-y|^{-(n+2s)}.
\end{equation*}
 Then one  uses the fact that
\begin{equation*}
	|\bG_{\cL_s} (f)| \le C_\Omega  \mathcal{I}_{n-2s}  (|f| \chi_\Omega)
\end{equation*}
 where $\mathcal{I}_\alpha$  is the convolution with the Riesz potential. Hence,
\begin{equation*}
	\|\bG_{\cL_s} (f)\|_{L^{q}} \le C_\Omega \| \mathcal{I}_{n-2s}  (|f| \chi_\Omega) \|_{L^{q}}.
\end{equation*}
Thus
\begin{equation*}
	C_s(p,q) \le C_\Omega \|  \mathcal{I}_{n-2s}  ( \cdot \chi_\Omega )\|_{\cL (L^p(\Omega), L^q (\Omega))  }.
\end{equation*}
Notice that
the right hand side can be  bounded uniformly as $s \to 1$, even though it depends on the order $n-2s$. On the other hand, following the proof of \Cref{prop:lin-L1}, we have that
\begin{equation*}
\|\bG_{\cL_s - \lambda} \|_{\cL (L^{1},L^{p})} \le \|\bG_{\cL_s - \lambda} \|_{\cL (L^{p'},L^{+\infty})}. \qedhere
\end{equation*}
\end{proof}

\subsection{The two main examples}\label{ss8.5}

We are ready to get down into practical  examples.

\medskip

\noindent {\sl 1. Restricted Fractional Laplacian. }
For this operator we have $\gamma(s) = s$. Hence, \eqref{eq:s to 1 assumption 2} holds for every $s$.

When $\Omega=B_r$, the explicit form of the  kernel is known (see, e.g., \cite{Bucur}), 
\[
\cG_{\Ds}(x,y)
=\dfrac{
    \Gamma(\frac{n}{2})
}{
    2^{2s}\Gamma^2(s)\pi^{\frac{n}{2}}
}
    \dfrac{1}{
        |x-y|^{n-2s}
    }
    \int_{0}^{
        \frac{
            (r^2-|x|^2)(r^2-|y|^2)
        }{
            r^2|x-y|^2
        }
    }
    \dfrac{
        t^{s-1}
    }{
        (t+1)^{\frac{n}{2}}
    }\,dt.
\]
It satisfies
\[
0\leq \cG_{\Ds}(x,y)
\leq C(n,r)|x-y|^{-(n-2s)},
\]
since the normalization constant is uniform as $s\to1$ and an upper bound for the last integral is the one on $(0,+\infty)$, therefore, we have \eqref{eq:s to 1 assumption 6}.

This formula is also valid when $s = 1$ and there is uniform convergence away from $x \ne y$. Hence, it is easy to see that \eqref{eq:s to 1 assumption 1} holds.
Furthermore, we can compute
\begin{align*}
	D_s \cG_{\Ds} (z,y) &= \lim_{x \to z} \frac{\cG_{\Ds}(x,y)}{\delta(x)^s} =  \lim_{x \to z}  \frac{\cG_{\Ds}(x,y)}{(r - |x|)^s}
\end{align*}
We can see that
\begin{align*}
	\frac{1}{(r - |x|)^s} \int_{0}^{
		\frac{
			(r^2-|x|^2)(r^2-|y|^2)
		}{
			r^2|x-y|^2
		}
	}
	\dfrac{
		t^{s-1}
	}{
		(t+1)^{\frac{n}{2}}
	}\,dt
	&= \frac{(r^2 - |x|^2)^s}{(r - |x|)^s}  \int_{0}^{
		\frac{
			r^2-|y|^2
		}{
			r^2|x-y|^2
		}
	}
	\dfrac{
		t^{s-1}
	}{
		( (r^2 -|x|^2) t+1)^{\frac{n}{2}}
	}\,dt \\
& \longrightarrow (2r)^s \int_{0}^{
	\frac{
		r^2-|y|^2
	}{
		r^2|z-y|^2
	}
}
{
	t^{s-1}
}\,dt
\end{align*}
as $x \to z\in \p B_r$.
Hence
\begin{align}
	D_s \cG_{\Ds} (z,y) &= \dfrac{
		\Gamma(\frac{n}{2})
	}{
		2^{s}\Gamma^2(s)\pi^{\frac{n}{2}}
	}
	\dfrac{r^s}{
		|z-y|^{n-2s}
	}
	\int_{0}^{
		\frac{
			r^2-|y|^2
		}{
			r^2|z-y|^2
		}
	}
	{
		t^{s-1}
	}\,dt. \nonumber
\\
	&= \dfrac{
		\Gamma(\frac{n}{2})
	}{
		2^{s}s\Gamma^2(s)\pi^{\frac{n}{2}}
	}
	\dfrac{(r^2-|y|^2)^s}{
		r^s|z-y|^{n}
	}.
\end{align}
This formula for the $s$-normal derivative, which seems new in the literature, is also valid for $s =1$. Indeed, as $s\to1^-$, we have the uniform convergence in every compact set in $B_r$ to the classical Poisson kernel (recall $|\bS^{n-1}|=2\pi^{\frac{n}{2}}/\Gamma(\frac{n}{2})$)
\[
D_1\cG_{-\Delta}(z,y)
=\dfrac{1}{|\bS^{n-1}|}
    \dfrac{
        r^2-|y|^2
    }{
        r|z-y|^n
    },
    \quad
y\in B_r, z\in\p B_r.
\]
Applying the Dominated Convergence Theorem we see that \eqref{eq:s to 1 assumption 3} holds. Therefore, we are in the correct setting at least when $\Omega = B_r$.

\bigskip

\noindent \textbf{Comments on the general setting.} In the general setting, some additional work is needed.  Let us give some hints on a possible approach. To check  \eqref{eq:s to 1 assumption 1} a possible scheme is as follows. First, use the Rayleigh quotient and $\Gamma$-convergence to check the convergence of the first eigenvalue
\begin{equation*}
	\lambda_1 (\Ds) = \min_{ 0 \ne u \in H^s_0 (\Omega) } \frac{\|(-\Delta)^{\frac s 2}_{\text{RFL}} u   \|_{L^2 (\Omega)}}{\| u \|_{L^2 (\Omega)}}
\end{equation*}
This shows strong $L^2$ convergence of subsequence of $\bG_{\Ds} (f)$. Since $\Ds \varphi \to -\Delta \varphi $ for adequate test functions, it is easy to characterise the limit a the $\bG_{-\Delta} (f)$. By uniqueness of the limit the whole sequence converges. Via the weak formulation and the uniform bounds, a rate of convergence can be recovered from that of $\Ds \varphi$.

In  \cite{RS1}, the authors prove that $\bG_{\Ds} (f) / \delta^s$ in $C^\alpha (\bar \Omega)$. It is to be expected that the constants for this embedding are uniform for $s$ close to $1$, and hence have a uniformly convergent subsequence. By uniqueness of the limit one shows the whole sequence converges. This would imply that $D_s  \bG_{\Ds} (\psi)$ uniformly converges to $D_1 { \bG_{\Ds} (\psi)}$, where $D_1$ is the standard normal derivative (as shown in \Cref{lem:constant}).

In order to check \eqref{eq:s to 1 assumption 3} (which is not studied in the main reference
\cite{A}), we can use the weak formulation
\begin{equation*}
	\int_\Omega \bM_{ \Ds } (h) \psi = \int_{\partial \Omega} h D_s  \bG_{\Ds} (\psi), \qquad \forall \psi \in L^\infty_c (\Omega),
\end{equation*}
and some compactness to show that \eqref{eq:s to 1 assumption 3} holds. The proof of \eqref{eq:s to 1 assumption 4}--\eqref{eq:s to 1 assumption 5} can probably be done directly.

\medskip

\noindent {\sl 2. Spectral Fractional Laplacian}.
For this operator $\gamma(s) = 1$. Hence assumption \eqref{eq:s to 1 assumption 1} follows directly from the eigen-decomposition.

By \cite{AD}, for instance,
\[
\bG_{\Dssp}(f)(x)
=\int_{\Omega}
    \int_{0}^{+\infty}
        \cK(t,x,y)f(y)
        \dfrac{
            t^{s-1}
        }{
            \Gamma(s)
        }
    \,dt
\,dy,
\]
where $\cK$ is the heat kernel for $u_t - \Delta$, which has the known estimates (see \cite[Lemma 1.3]{Hui})
\begin{equation*}
	\cK (t,x,y) \le \frac{C}{t^{\frac n 2}} \exp \left( - \frac{|x-y|^2}{Ct}   \right)  \left(  \frac{\delta(x)}{\sqrt t} \wedge 1  \right)\left(  \frac{\delta(y)}{\sqrt t} \wedge 1  \right).
\end{equation*}
Hence
\begin{equation*}
	\cG_{\Dssp} (x,y) = \int_{0}^{+\infty}
	\cK(t,x,y)
	\dfrac{
		t^{s-1}
	}{
		\Gamma(s)
	}
	\,dt
\end{equation*}
so hypothesis \eqref{eq:s to 1 assumption 6} holds.

This also holds for $s=1$. Thus
\[
\norm[L^2(\Omega)]{
    \bG_{\Dssp}(f)-\bG_{-\Delta}(f)
}^2
=\int_{\Omega}
\left(
\int_{\Omega}
    \int_{0}^{+\infty}
        \cK(t,x,y)f(y)
        \left(
            \dfrac{
                t^{s-1}
            }{
                \Gamma(s)
            }
            -1
        \right)
    \,dt
\,dy
\right)^2
\,dx.
\]
From here \eqref{eq:s to 1 assumption 1} follows. Also,
\cite{AD} shows that
\[
\bM_{\Dssp}(h)(x)
=\int_{\p\Omega}
    \int_{0}^{+\infty}
        -\dfrac{
            \p\cK(t,x,z)
        }{
            \p\nu_z
        }
        h(z)
        \dfrac{
            t^{s-1}
        }{
            \Gamma(s)
        }
    \,dt
\,dz.
\]
Then
\begin{multline*}
\norm[L^1(\Omega)]{
    \bM_{\Dssp}(h)
    -\bM_{-\Delta}(h)
}
\\
\leq
\norm[L^\infty(\p\Omega)]{h}
\int_{\Omega}\int_{\p\Omega}
\int_{0}^{+\infty}
    \abs{
        \dfrac{\p\cK(t,x,z)}{\p\nu_z}
    }
    \left(
        \dfrac{t^{s-1}}{\Gamma(s)}-1
    \right)
\,dt\,dz\,dx.
\end{multline*}
From here it is not hard to deduce \eqref{eq:s to 1 assumption 3}.

\subsection{Some exotic examples}
It is known for any non-negative bounded potential $V$,  $\cL_s = \Ds + V(x)$ satisfies (A1), (A2) and (K1) (furthermore $\cG_{\cL_s} \le  \cG_{\Ds} $). If $V$ is smooth, then (K2) also holds. However, moving the parameter $s$ could lead to some strange behaviours. We could think about the family of operators
\begin{equation*}
	\cL_s = (-\Delta)^s_{\textrm{RFL}} +  \delta(x)^{-3} \wedge \frac{1}{1-s}.
\end{equation*}
These operators satisfy (K1), but not uniformly from below. The $\cL_1 = -\Delta + \delta^{-3}$ where the solutions are flat $|u| \le C \exp (  -\frac 1 {\delta} )$. The properties of the Green function of $\cL_1  = - \Delta + \delta^{-3}$ are still not well understood.

\section*{Comments, extensions and open problems}

We collect here further issues that we would like to comment or propose.

$\bullet$  It could be interesting to find better properties of the Green function for $\cL-\lambda$.

$\bullet$ The theory could be applied to other examples of operators, e.g. the relativistic version $\sqrt{-\Delta+m^2}$.

$\bullet$ Schauder estimates should be found if $\cG_0$ is smooth enough.

$\bullet$ Find versions of the main results when $g, h$ are measures under certain conditions.

$\bullet$ Consider problems of the form
\[\begin{cases}
\cL{u}=f(u)
& \textin \Omega,\\
u(x)\to+\infty
& \textas x\to\p\Omega,\\
u=0
& \textin \Omega^c
\quad \text{(if applicable)}
\end{cases}\]
with suitable growth conditions on the nonlinearity $f$.

\appendix

\section{Some comments on compactness and functional spaces}
\label{sec:compactness}
Denoting by $E_i$ the eigenspace corresponding to $\lambda_i$ ($i\geq1$), we have $\forall \varphi_j\in E_i$,
\[\begin{cases}
{\cL}\varphi_j
=\lambda_i\varphi_j,
& \textin \Omega\\
\varphi_j=0
& \textin \overline{\Omega^c}
\quad \textor \quad
\texton \p\Omega,
\end{cases}\]
and for any $f\in L^2(\Omega)$,
\[
f=\sum_{j\geq1}
\angles{f,\varphi_j}
\varphi_j.
\]
The solution operator in spectral form is given by
\begin{equation*}
\bG_0 (f)  = \sum_{j=1}^{+\infty} \frac{ \langle f, \varphi_j \rangle}{\lambda_j } \varphi_j .
\end{equation*}
which is a well defined sum in $L^2 (\Omega)$ since $\lambda_i \to +\infty$ and $\varphi_j$ are orthonormal in $L^2$ so
\begin{equation*}
\|	\bG_0 (f)  \|_{L^2 (\Omega)}^2 = \sum_{j=1}^{+\infty} \frac{ \langle f, \varphi_j \rangle^2}{\lambda_j^2 }.
\end{equation*}
Therefore, defining $\bH_\cL^2(\Omega) = \bG_0 (L^2 (\Omega))$ we easily see that
\[
\bH^{2}_{\cL}(\Omega)
=\set{
	v\in L^2(\Omega):
	\sum_{j\geq1}
	\lambda_j^{2}
	\angles{v,\varphi_j}^2
	<+\infty
}.
\]
In fact, we can define \emph{push-forward} norms for $k \ge 0$,
\begin{equation*}
\|u\|_{\bH^{k}_{\cL}(\Omega)} = \sqrt{ 	\sum_{j\geq1}
	\lambda_j^{k}
	\angles{u,\varphi_j}^2 },
\end{equation*}
so that $\|\bG_0 (f)\|_{\bH^{2}_{\cL}(\Omega)} = \|f \|_{L^2 (\Omega)}$.
Furthermore, $\cL:\bH^{2}_{\cL}(\Omega)\to L^2(\Omega)$ is also an isometry. As in the theory of $-\Delta$, a natural energy space is
\[
\bH^{1}_{\cL}(\Omega)
=\set{
	v\in L^2(\Omega):
	\sum_{j\geq1}
	\lambda_j
	\angles{v,\varphi_j}^2
	<+\infty
}.
\]
This space was studied in \cite{BSV}. Notice that
\begin{equation*}
\int_\Omega \bG_0 (f) \  f  = 	\sum_{j\geq1}
\frac{ \angles{f,\varphi_j}^2 }{\lambda_j } = \|\bG_0 (f) \|_{\bH^1_\cL}^2  .
\end{equation*}
Therefore, for weak solutions
\begin{equation*}
\|\bG_0 (f) \|_{\bH^1_\cL (\Omega)} \le \|f \|_{L^2} \|\bG_0 (f) \|_{L^2}.
\end{equation*}
We always have that
\begin{equation*}
\lambda_1 \| u\|_{L^2 (\Omega)} \le \|u\|_{\bH^1_\cL (\Omega)}
\end{equation*}
hence, since $\lambda_1 > 0$ then we can call this Poincaré inequality and deduce that
\begin{equation*}
\|\bG_0 (f) \|_{L^2(\Omega)} \le \frac{1}{\lambda_1} \| f \|_{L^2 (\Omega)}.
\end{equation*}
In \cite{BFV} the authors prove that this operator $\bG_0$ is compact in $L^2(\Omega)$, by a sharp application of the Riesz--Fr\'{e}chet--Kolmogorov theorem. By definition, $\bH_\cL^{2}(\Omega) = \bG_0 (L^2 (\Omega))$ with its norm is compactly embedded in $L^2 (\Omega)$. However, since the authors of \cite{BFV} estimate the translations $\|\tau_h \bG_0 (f) - \bG_0 (f) \|_{L^2}$ without rates, there is no estimate of $\bH^{2}_\cL (\Omega)$ in any of the Sobolev spaces $W^{t,2} (\Omega)$.

The question of whether $\bH_\cL^1 (\Omega)$ is compactly embedded in $L^2 (\Omega)$ is left open.

\bigskip

The readers are directed to \cite{BSV} for a comprehensive introduction to both RFL and SFL operators in Sobolev spaces. Note that in both cases, the natural domain is identified as\footnote{Here $H^{\frac12}_{00}(\Omega)$ is the Lions--Mag\`enes spaces \cite{LM}.
}
\[
\bH^{1}_{\Dssp}(\Omega) = \bH_{\Ds}^1 (\Omega)=
\begin{cases}
H^{s}(\Omega)
& \textfor s\in(0,\frac12),\\
H^{\frac 1 2}_{00} (\Omega)
& \textfor s= \frac 1 2,\\
H^{s}_{0}(\Omega)
& \textfor s\in (\frac12,1).
\end{cases}
\]
Note that exponent $1$ in left-hand side becomes $s$ in the right-hand side. Similar, the $\bH^{2}_{\Dssp}(\Omega) $ are related to spaces of type $H^{2s}$.

\section{The weighted trace for the restricted fractional Laplacian}
\label{sec:RFL-trace}

Let $\cL=\Ds$. Recall the definition of
the weighted trace operator
\[
Bu(z)
=
    \lim_{\Omega\ni x\to z}
    \dfrac{u(x)}{
        \bM(1)(x)
    }
        \quad \textfor z\in\p\Omega.
\]
Recall from \cite{A,AGV} that $\delta^{1-s}\bM(1)
$ is a positive, continuous function, bounded away from $0$ and $+\infty$. Thanks to the connection pointed out by Ros-Oton, we show that it is actually a constant, independent of the domain.

\begin{lem}%
\label{lem:constant}
Let $v\in
\delta^{s-1} C(\overline{\Omega})$. For any $z\in\p\Omega$,
\[
Bv(z)
=
	\Gamma(1+s) \Gamma(s)
	\normalcolor 
    \lim_{\Omega\ni x\to z}
    \delta^{1-s}(x)v(x).
\]
\end{lem}
We remark that as $s\nearrow 1$, we recover the usual trace. %

\begin{proof}
Let $u\in \Ints C_c^\infty(\Omega)$. Recall that the integration-by-parts formula \cite{A} applied to the functions $u$ and $x\cdot \nabla u$ reads
\begin{multline*}
    \int_{\p\Omega}
        E(x\cdot\nabla u)(z)
        \dfrac{u}{\delta^s}
    \,d\cH^{n-1}(z)\\
=
\int_{\Omega}
    (x\cdot\nabla u)\Ds u
\,dx
-\int_{\Omega}
	u\Ds(x\cdot\nabla u)
\,dx.
\end{multline*}
On the other hand, by the Poho\v{z}aev identity \cite{RS2},
\begin{multline*}
\int_{\Omega}
	(x\cdot\nabla u)\Ds u
\,dx\\
=
	-\dfrac{n-2s}{2}
	\int_{\Omega}
		u\Ds u
	\,dx
	-\dfrac{\Gamma(1+s)^2}{2}
	\int_{\p\Omega}
		\left(
			\dfrac{u}{\delta^s}
		\right)^2
		({z}\cdot\nu)
	\,d\cH^{n-1}({z}).
\end{multline*}
Using this identity together with the pointwise relation
\[
\Ds(x\cdot \nabla u)
=x\cdot \nabla \Ds u
	+2s \Ds u,
\]
we have
\begin{multline*}
\int_{\Omega}
	u\Ds(x\cdot\nabla u)
\,dx\\
=
	-\dfrac{n-2s}{2}
	\int_{\Omega}
		u\Ds u
	\,dx
	+\dfrac{\Gamma(1+s)^2}{2}
	\int_{\p\Omega}
		\left(
			\dfrac{u}{\delta^s}
		\right)^2
		({z}\cdot\nu)
	\,d\cH^{n-1}({z}).
\end{multline*}
Combining,
\begin{equation}\label{eq:E-int}
    \int_{\p\Omega}
        E(x\cdot\nabla u)({z})
        \dfrac{u}{\delta^s}
    \,d\cH^{n-1}({z})
=
	-\Gamma(1+s)^2
	\int_{\p\Omega}
		\left(
			\dfrac{u}{\delta^s}
		\right)^2
		({z}\cdot\nu)
	\,d\cH^{n-1}({z}).
\end{equation}
In a tubular neighborhood of $\p\Omega$, choose an orthonormal frame $\nu(x)$, $e_1(x), \dots, e_{n-1}(x)$, so that $\nu(x)$ agrees with the outward normal. We have for ${z}\in\p\Omega$,
\[\begin{split}
E(x\cdot\nabla u)({z})
&=
    \lim_{\Omega\ni x\to{z}}
    \dfrac{x\cdot \nabla u(x)}{
        \bM(1)(x)
    }\\
&=
	\lim_{\Omega\ni x\to{z}}
    \dfrac{
    	(x\cdot \nu(x))
	    (\nabla u(x)\cdot \nu(x))
	    +(x\cdot e_k(x))
	    (\nabla u(x)\cdot e_k(x))
	}{
        \bM(1)(x)
    }.
\end{split}\]
Note that
\[
\lim_{\Omega\ni x\to{z}}
	(x\cdot \nu(x))
={z} \cdot \nu({z}),
	\quad
\limsup_{\Omega\ni x\to{z}}
	|x\cdot e_k(x)|
\leq C,
\]
where $\nu({z})$ is the outward normal at ${z}\in\p\Omega$.
Then, since
\[
\nabla u(x)\cdot \nu(x)
=\lim_{h\to0}
	\dfrac{
		u(x+h\nu(x))-u(x)
	}{h},
\quad
\nabla u(x)\cdot e_k(x)
=\lim_{h\to0}
	\dfrac{
		u(x+he_k(x))-u(x)
	}{h},
\]
exist, one can take the particular sequence $h=\delta(x)$, since $\delta(x)\to0$ as $x\to{z}\in\p\Omega$. Recall that by \cite[Lemma 3.4]{A}, %
$\delta^{1-s}\bM(1)$ is bounded between positive constants. Also, by \cite{RS2} $u/\delta^s$ can be extended to a H\"{o}lder continuous function in $\overline{\Omega}$. We estimate the tangential terms by
\[\begin{split}
&\quad\;
\limsup_{\Omega\ni x\to{z}}
\abs{
	\dfrac{
	    (x\cdot e_k(x))
	    (\nabla u(x)\cdot e_k(x))
	}{
        \bM(1)(x)
    }
}\\
&\leq C\limsup_{\Omega\ni x\to{z}}
\abs{
	\delta(x)^{1-s}
	\dfrac{
		u(x+\delta(x)e_k(x))-u(x)
	}{
		\delta(x)
	}
}\\
&\leq C\limsup_{\Omega\ni x\to{z}}
\abs{
	\dfrac{
		u(x+\delta(x)e_k(x))
	}{
		\delta^s(x+\delta(x)e_k(x))
	}
	\dfrac{
		\delta^s(x+\delta(x)e_k(x))
	}{
		\delta^s(x)
	}
	-\dfrac{
		u(x)
	}{
		\delta^s(x)
	}
}\\
&\leq
	C\abs{
		\dfrac{u}{\delta^s}({z})
		\cdot1
		-\dfrac{u}{\delta^s}({z})
	}\\
&=0.
\end{split}\]
On the other hand, we notice that
\[
x+\delta(x)\nu(x)\in\p\Omega
	\quad \Longrightarrow \quad
u(x+\delta(x)\nu(x))=0.
\]
We expand
$
	\nabla u = \nabla \left( \frac{u}{\delta^s} \delta^s \right) = \delta^s \nabla \frac {u}{\delta^s} + s \frac{u}{\delta} \nabla \delta.
$
Since $u$ is differentiable in $\Omega$ and $u/\delta^s  \in C^\alpha (\overline \Omega)$, then up to a subsequence $\lim_{x_n \to  z} |\delta^{1-\alpha}(x_n) \nu (z) \cdot \nabla (u / \delta^s) (x_n) | \le C $ by the Mean Value Theorem. 
\normalcolor 
Therefore,
\[\begin{split}
E(x\cdot\nabla u)({z})
&=
	\lim_{\Omega\ni x\to{z}}
	\dfrac{
		s  u(x)  \nabla  \delta(x) \cdot x
	}{
		\delta(x)
        \bM(1)(x)        
	}\\
&=- 
s
\normalcolor 
	(z\cdot\nu({z}))
	\dfrac{u}{\delta^s}(z)
	\lim_{\Omega\ni x\to{z}}
	\dfrac{1}{
    	\delta^{1-s}(x)
        \bM(1)(x)
	}.
\end{split}\]
Plugging this back into \eqref{eq:E-int}, we have
\begin{multline*}
    s
    \normalcolor 
    \int_{\p\Omega}
    \left(
	    \lim_{\Omega\ni x\to{z}}
		\dfrac{1}{
    		\delta^{1-s}(x)
            \bM(1)(x)
		}
	\right)
    	\left(
			\dfrac{u}{\delta^s}
		\right)^2
		({z}\cdot\nu)
	\,d\cH^{n-1}({z})\\
=
	\Gamma(1+s)^2
	\int_{\p\Omega}
		\left(
			\dfrac{u}{\delta^s}
		\right)^2
		({z}\cdot\nu)
	\,d\cH^{n-1}({z}).
\end{multline*}
By \cite[Proposition 2]{A}, the limit
\[
a({z})
:=\lim_{\Omega\ni x\to{z}}
		\dfrac{1}{
    		\delta^{1-s}(x)
            \bM(1)(x)
		}
\]
is well-defined and continuous. Since the Poho\v{z}aev identity can be applied with any center $x_0\in\R^n$, we can write
\begin{equation}\label{eq:a-1}
    \int_{\p\Omega}
    \left(
	    a({z})
	    -\Gamma(1+s)
	    \Gamma(s)
	    \normalcolor 
	\right)
    	\left(
			\dfrac{u}{\delta^s}
		\right)^2
		(({z}-x_0)\cdot\nu)
	\,d\cH^{n-1}({z})
=0.
\end{equation}
We will show that $a(z)\equiv 
\Gamma(1+s) \Gamma(s)\normalcolor $ by a contradiction argument choosing appropriate $u$. For any $x_1\in\Omega$, we choose $u_\eps(x)=\Ints \eta_\eps(x-x_1)$ (where $\eps<\delta(x_1)$) where $\eta_\eps\in C_c^\infty(\Omega)$ is the standard mollifier centered at $x_1$. Then by \Cref{prop:D-gamma-def}, we have
\[\begin{split}
\dfrac{u_\eps}{\delta^s}({z})
=D_s u_\eps(z)
&=\int_{\Omega}
    D_s\cG_0(z,y)
	\eta_\eps(y-x_1)
\,dy.
\end{split}\]
Using \eqref{eq:a-1} with $u=u_\eps$,\begin{equation*}
\int_{\p\Omega}
    \left(
	    a({z})
	    -	
\Gamma(1+s) \Gamma(s)\normalcolor 
	\right)
	\left(
		\int_{\Omega}
            D_s\cG_0(z,y)
			\eta_\eps(y-x_1)
		\,dy
	\right)^2
	(({z}-x_0)\cdot\nu)
\,d\cH^{n-1}({z})
=0.
\end{equation*}
Taking $\eps\searrow0$,
\begin{equation}\label{eq:a-2}
\int_{\p\Omega}
    \left(
	    a({z})
	    -	
\Gamma(1+s) \Gamma(s)\normalcolor 
	\right)
    (D_s\cG_0(z,x_1))^2
	(({z}-x_0)\cdot\nu)
\,d\cH^{n-1}({z})
=0,
\end{equation}
for any $z\in\Omega$. Recall that
\[
D_s\cG_0(z,x_1)
\asymp
	\dfrac{\delta^s(x_1)}{|x_1-{z}|^n}.
\]
Suppose $a({z}_0)\neq 
\Gamma(1+s) \Gamma(s)\normalcolor $. Then there exists a neighborhood ${z}_0\in\omega\subset\p\Omega$ such that $a\neq	
\Gamma(1+s) \Gamma(s)\normalcolor $ on $\omega$. Moreover, one can choose $x_0$ such that $({z}-x_0)\cdot\nu\neq 0$ for ${z}\in\omega$. Dividing \eqref{eq:a-2} by $\delta^{2s}(x_1)$, we have
\[
\int_{\omega}
	\dfrac{
		|a({z})-	
\Gamma(1+s) \Gamma(s)\normalcolor |
		|({z}-x_0)\cdot\nu|
	}{
		|x_1-{z}|^{2n}
	}
\,d\cH^{n-1}({z})
\leq C(\omega),
\]
a contradiction as $x_1\to{z}_0$. Therefore, $a({z})\equiv 	
\Gamma(1+s) \Gamma(s)\normalcolor $ and the proof is complete.
\end{proof}

\section{Embeddings into Morrey spaces}
\label{sec:embed}

Since we mention a special case of the regularity results of Fall \cite{F}, we indicate the corresponding embedding results into Morrey spaces. Recall that the Morrey space $\cM_{\beta}$, $\beta\in[0,n]$, is defined by
\begin{equation}\label{eq:Morrey-def}
\cM_{\beta}(\Omega)
=\set{
    f\in L^1(\Omega):
    \norm[\cM_\beta(\Omega)]{f}
    :=\sup_{r\in(0,1),\,x\in\Omega}
        r^{\beta-n}
        \int_{B_r(x)\cap\Omega}|f(y)|\,dy
    <\infty
}
\end{equation}

\begin{lem}[High integrability data]
\label{lem:embed-Lp}
For any $p\in(\frac{n}{s},\infty]$,
\[
L^p(\Omega)\hookrightarrow \cM_\beta(\Omega)
\]
for $\beta=\frac{n}{p}\in[0,s)$.
\end{lem}

\begin{proof}
This is a direct consequence of H\"{o}lder inequality. Indeed, for any $B_r(x)$ with $x\in\Omega$,
\[
\int_{B_r(x)\cap\Omega}|f|\,dx
\leq
    \left(
        \int_{\Omega}|f|^p\,dx
    \right)^{\frac1p}
    |B_r(x)|^{\frac{p-1}{p}}
\leq C\norm[L^p(\Omega)]{f}
    r^{n-\frac{n}{p}},
\]
for finite $p$, and the same is true when $p=\infty$.
\end{proof}

It is also instructive to observe how the weighted $L^\infty$ functions are embedded into Morrey spaces. This leads directly to regularity properties for the inhomogeneous eigenvalue problem for RFL, with the large RFL-harmonic function as right hand side.

\begin{lem}[Weighted $L^\infty$ data]
For $\beta\in[0,1)$, we have the continuous inclusion
\[
\delta^{-\beta}L^\infty(\Omega)
\hookrightarrow
\cM_{\beta}(\Omega)
\]
Moreover, there holds
\[
\norm[\cM_\beta(\Omega)]{f}
\leq C\left(
    \norm[L^\infty(\Omega)]{f\delta^{\beta}}
    +\norm[L^1(\Omega)]{f}
\right),
\]
whenever the right hand side is finite.
\end{lem}

\begin{proof}
Let $C_0=\norm[L^\infty(\Omega)]{f\delta^\beta}+\norm[L^1(\Omega)]{f}$. We want to show that
\[
\int_{B_r(x)\cap\Omega}|f(y)|\,dy
\leq CC_0r^{n-\beta},
\]
for any $0<r<\diam(\Omega)$ and $x\in\Omega$. Indeed, we distinguish between three cases as follows.
\begin{enumerate}
\item
If $r<\delta(x)/2$, then $B_r(x)\subset B_{2r}(x)\subset\Omega$ and for any $y\in B_r(x)$, $\delta(y)\geq \delta(x)/2 > r$, so that
\[
\int_{B_r(x)\cap\Omega}
    |f(y)|
\,dy
\leq CC_0\int_{B_r(x)}
        \delta(y)^{-\beta}
    \,dy
\leq CC_0r^{-\beta}|B_r(x)|.
\]
\item
In a small enough neighborhood of $\p\Omega$, the Fermi coordinates are well-defined and used to flatten the boundary. More precisely, we write a subset of $\p\Omega$ as the image of a $C^{1,1}$ function
\[
\psi: B_{2r}'\subset\R^{n-1}\to\p\Omega,
\]
with $\psi(0)=\arg\min_{x_0\in\p\Omega}\dist(x,x_0)$. Then we define
\[\begin{split}
\Psi:B_{2r}'\times(0,4r)
    &\to \set{\delta<4r}\subset\Omega
\\
\Psi(z',z_n)
    &= \psi(z')-z_n \nu(\psi(z')),
\end{split}\]
where $\nu(\psi(z'))$ is the outward normal of $\Omega$ at $\psi(z')\in\p\Omega$.

If $r<r_1$ and $r_1$ is fixed sufficiently small, then $|\det\Psi(z',z_n)| \leq C$ for all $(z',z_n)\in B_{2r}'\times(0,4r)$. By Area Formula,
\[\begin{split}
\int_{B_r(x)\cap\Omega}
    |f(y)|
\,dy
&\leq
    \int_{\Psi(B_{2r}'\times(0,4r))}
        C_0\delta(y)^{-\beta}
    \,dy\\
&\leq
    CC_0
    \int_{B_{2r}'}\int_{0}^{4r}
        z_n^{-\beta}
    \,dz_n\,dz'\\
&\leq CC_0 r^{n-\beta}.
\end{split}\]
\item
Since $f\in L^1(\Omega)$, if $r\geq r_1$, then
\[
\int_{B_r(x)\cap\Omega}|f(y)|\,dy
\leq Cr_1^{n-\beta}\int_{\Omega}|f(y)|\,dy
\leq CC_0r^{n-\beta}.
\]
\end{enumerate}
By taking supremum over $r$ and $x$, we conclude that
\[
\norm[\cM_\beta(\Omega)]{f}\leq CC_0. \qedhere
\]
\end{proof}

\section*{Acknowledgements}

H.C. has received funding from the European Research Council under the Grant Agreement No 721675. He acknowledges the kind hospitality received in the Universidad Aut\'{o}noma de Madrid during his visit in January 2020. The work of D.G-C.  and JLV was funded by  grant PGC2018-098440-B-I00 from  the Spanish Government. J.~L.~V\'azquez is also an Honorary Professor at Univ.\ Complutense de Madrid. H.C. is grateful to Xavier Ros-Oton for pointing out the connection between the weighted trace operator and the Poho\v{z}aev identity, and to Yannick Sire for a helpful comment on interpolation spaces. H.C. also indebted to Alessio Figalli for an enlightening comment concerning regularity and for motivating encouragements.  We thank the anonymous referee for useful comments. 

\medskip

\bigskip

\end{document}